\documentclass[12pt]{amsart}
\usepackage{txfonts}      
\usepackage{amssymb}
\usepackage{eucal}
\usepackage{amsmath}
\usepackage{amscd}
\usepackage{xcolor}
\usepackage{multicol}
\usepackage[all]{xy}           
\usepackage{graphicx}
\usepackage{color}
\usepackage{colordvi}
\usepackage{xspace}
\usepackage{tikz}
\usepackage{makecell}
\usepackage{appendix}
\usepackage{amsthm}
\usepackage{enumitem}

\usepackage{ifpdf}
\ifpdf
\usepackage[colorlinks,final,backref=page,hyperindex]{hyperref}
\else
\usepackage[colorlinks,final,hyperindex,hypertex]{hyperref}
\fi

\usepackage{booktabs}  
\usepackage{amsmath, amssymb, amsthm}
\usepackage{float}  
\usepackage[active]{srcltx} 
\usepackage{longtable}

\begin{document}


\newtheorem{thm}{Theorem}[section]
\newtheorem{lem}[thm]{Lemma}
\newtheorem{cor}[thm]{Corollary}
\newtheorem{pro}[thm]{Proposition}
\theoremstyle{definition}
\newtheorem{defi}[thm]{Definition}
\newtheorem{ex}[thm]{Example}
\newtheorem{rmk}[thm]{Remark}
\newtheorem{pdef}[thm]{Proposition-Definition}
\newtheorem{condition}[thm]{Condition}

\renewcommand{\labelenumi}{{\rm(\alph{enumi})}}
\renewcommand{\theenumi}{\alph{enumi}}

\newcommand {\emptycomment}[1]{} 

\newcommand{\nc}{\newcommand}
\newcommand{\delete}[1]{}

\nc{\tred}[1]{\textcolor{red}{#1}}
\nc{\tblue}[1]{\textcolor{blue}{#1}}
\nc{\tgreen}[1]{\textcolor{green}{#1}}
\nc{\tpurple}[1]{\textcolor{purple}{#1}}
\nc{\tgray}[1]{\textcolor{gray}{#1}}
\nc{\torg}[1]{\textcolor{orange}{#1}}
\nc{\tmag}[1]{\textcolor{magenta}}
\nc{\btred}[1]{\textcolor{red}{\bf #1}}
\nc{\btblue}[1]{\textcolor{blue}{\bf #1}}
\nc{\btgreen}[1]{\textcolor{green}{\bf #1}}
\nc{\btpurple}[1]{\textcolor{purple}{\bf #1}}

\nc{\todo}[1]{\tred{To do:} #1}

    \nc{\mlabel}[1]{\label{#1}}  
    \nc{\mcite}[1]{\cite{#1}}  
    \nc{\mref}[1]{\ref{#1}}  
    \nc{\meqref}[1]{\eqref{#1}}  
    \nc{\mbibitem}[1]{\bibitem{#1}} 

\delete{
    \nc{\mcite}[1]{\cite{#1}{\small{\tt{{\ }(#1)}}}}  
    \nc{\mref}[1]{\ref{#1}{\small{\tred{\tt{{\ }(#1)}}}}}  
    \nc{\meqref}[1]{\eqref{#1}{{\tt{{\ }(#1)}}}}  
    \nc{\mbibitem}[1]{\bibitem[\bf #1]{#1}} 
}


\nc{\cm}[1]{\textcolor{red}{Chengming:#1}}
\nc{\yy}[1]{\textcolor{blue}{Yanyong: #1}}
\nc{\li}[1]{\textcolor{purple}{#1}}
\nc{\lir}[1]{\textcolor{purple}{Li:#1}}

\nc{\revise}[1]{\textcolor{blue}{#1}}


\nc{\tforall}{\ \ \text{for all }}
\nc{\hatot}{\,\widehat{\otimes} \,}
\nc{\complete}{completed\xspace}
\nc{\wdhat}[1]{\widehat{#1}}

\nc{\ts}{\mathfrak{p}}
\nc{\mts}{c_{(i)}\ot d_{(j)}}

\nc{\NA}{{\bf NA}}
\nc{\LA}{{\bf Lie}}
\nc{\CLA}{{\bf CLA}}

\nc{\cybe}{CYBE\xspace}
\nc{\nybe}{NYBE\xspace}
\nc{\ccybe}{CCYBE\xspace}

\nc{\ndend}{pre-Novikov\xspace}
\nc{\calb}{\mathcal{B}}
\nc{\rk}{\mathrm{r}}
\newcommand{\g}{\mathfrak g}
\newcommand{\h}{\mathfrak h}
\newcommand{\pf}{\noindent{$Proof$.}\ }
\newcommand{\frkg}{\mathfrak g}
\newcommand{\frkh}{\mathfrak h}
\newcommand{\Id}{\rm{Id}}
\newcommand{\gl}{\mathfrak {gl}}
\newcommand{\ad}{\mathrm{ad}}
\newcommand{\add}{\frka\frkd}
\newcommand{\frka}{\mathfrak a}
\newcommand{\frkb}{\mathfrak b}
\newcommand{\frkc}{\mathfrak c}
\newcommand{\frkd}{\mathfrak d}
\newcommand {\comment}[1]{{\marginpar{*}\scriptsize\textbf{Comments:} #1}}

\nc{\vspa}{\vspace{-.1cm}}
\nc{\vspb}{\vspace{-.2cm}}
\nc{\vspc}{\vspace{-.3cm}}
\nc{\vspd}{\vspace{-.4cm}}
\nc{\vspe}{\vspace{-.5cm}}


\nc{\disp}[1]{\displaystyle{#1}}
\nc{\bin}[2]{ (_{\stackrel{\scs{#1}}{\scs{#2}}})}  
\nc{\binc}[2]{ \left (\!\! \begin{array}{c} \scs{#1}\\
    \scs{#2} \end{array}\!\! \right )}  
\nc{\bincc}[2]{  \left ( {\scs{#1} \atop
    \vspace{-.5cm}\scs{#2}} \right )}  
\nc{\ot}{\otimes}
\nc{\sot}{{\scriptstyle{\ot}}}
\nc{\otm}{\overline{\ot}}
\nc{\ola}[1]{\stackrel{#1}{\la}}

\nc{\scs}[1]{\scriptstyle{#1}} \nc{\mrm}[1]{{\rm #1}}

\nc{\dirlim}{\displaystyle{\lim_{\longrightarrow}}\,}
\nc{\invlim}{\displaystyle{\lim_{\longleftarrow}}\,}

\nc{\bfk}{{\bf k}} \nc{\bfone}{{\bf 1}}
\nc{\rpr}{\circ}
\nc{\dpr}{{\tiny\diamond}}
\nc{\rprpm}{{\rpr}}

\nc{\mmbox}[1]{\mbox{\ #1\ }} \nc{\ann}{\mrm{ann}}
\nc{\Aut}{\mrm{Aut}} \nc{\can}{\mrm{can}}
\nc{\twoalg}{{two-sided algebra}\xspace}
\nc{\colim}{\mrm{colim}}
\nc{\Cont}{\mrm{Cont}} \nc{\rchar}{\mrm{char}}
\nc{\cok}{\mrm{coker}} \nc{\dtf}{{R-{\rm tf}}} \nc{\dtor}{{R-{\rm
tor}}}
\renewcommand{\det}{\mrm{det}}
\nc{\depth}{{\mrm d}}
\nc{\End}{\mrm{End}} \nc{\Ext}{\mrm{Ext}}
\nc{\Fil}{\mrm{Fil}} \nc{\Frob}{\mrm{Frob}} \nc{\Gal}{\mrm{Gal}}
\nc{\GL}{\mrm{GL}} \nc{\Hom}{\mrm{Hom}} \nc{\hsr}{\mrm{H}}
\nc{\hpol}{\mrm{HP}}  \nc{\id}{\mrm{id}} \nc{\im}{\mrm{im}}

\nc{\incl}{\mrm{incl}} \nc{\length}{\mrm{length}}
\nc{\LR}{\mrm{LR}} \nc{\mchar}{\rm char} \nc{\NC}{\mrm{NC}}
\nc{\mpart}{\mrm{part}} \nc{\pl}{\mrm{PL}}
\nc{\ql}{{\QQ_\ell}} \nc{\qp}{{\QQ_p}}
\nc{\rank}{\mrm{rank}} \nc{\rba}{\rm{RBA }} \nc{\rbas}{\rm{RBAs }}
\nc{\rbpl}{\mrm{RBPL}}
\nc{\rbw}{\rm{RBW }} \nc{\rbws}{\rm{RBWs }} \nc{\rcot}{\mrm{cot}}
\nc{\rest}{\rm{controlled}\xspace}
\nc{\rdef}{\mrm{def}} \nc{\rdiv}{{\rm div}} \nc{\rtf}{{\rm tf}}
\nc{\rtor}{{\rm tor}} \nc{\res}{\mrm{res}} \nc{\SL}{\mrm{SL}}
\nc{\Spec}{\mrm{Spec}} \nc{\tor}{\mrm{tor}} \nc{\Tr}{\mrm{Tr}}
\nc{\mtr}{\mrm{sk}}

\nc{\ab}{\mathbf{Ab}} \nc{\Alg}{\mathbf{Alg}}

\nc{\BA}{{\mathbb A}} \nc{\CC}{{\mathbb C}} \nc{\DD}{{\mathbb D}}
\nc{\EE}{{\mathbb E}} \nc{\FF}{{\mathbb F}} \nc{\GG}{{\mathbb G}}
\nc{\HH}{{\mathbb H}} \nc{\LL}{{\mathbb L}} \nc{\NN}{{\mathbb N}}
\nc{\QQ}{{\mathbb Q}} \nc{\RR}{{\mathbb R}} \nc{\BS}{{\mathbb{S}}} \nc{\TT}{{\mathbb T}}
\nc{\VV}{{\mathbb V}} \nc{\ZZ}{{\mathbb Z}}


\nc{\calao}{{\mathcal A}} \nc{\cala}{{\mathcal A}}
\nc{\calc}{{\mathcal C}} \nc{\cald}{{\mathcal D}}
\nc{\cale}{{\mathcal E}} \nc{\calf}{{\mathcal F}}
\nc{\calfr}{{{\mathcal F}^{\,r}}} \nc{\calfo}{{\mathcal F}^0}
\nc{\calfro}{{\mathcal F}^{\,r,0}} \nc{\oF}{\overline{F}}
\nc{\calg}{{\mathcal G}} \nc{\calh}{{\mathcal H}}
\nc{\cali}{{\mathcal I}} \nc{\calj}{{\mathcal J}}
\nc{\call}{{\mathcal L}} \nc{\calm}{{\mathcal M}}
\nc{\caln}{{\mathcal N}} \nc{\calo}{{\mathcal O}}
\nc{\calp}{{\mathcal P}} \nc{\calq}{{\mathcal Q}} \nc{\calr}{{\mathcal R}}
\nc{\calt}{{\mathcal T}} \nc{\caltr}{{\mathcal T}^{\,r}}
\nc{\calu}{{\mathcal U}} \nc{\calv}{{\mathcal V}}
\nc{\calw}{{\mathcal W}} \nc{\calx}{{\mathcal X}}
\nc{\CA}{\mathcal{A}}

\nc{\fraka}{{\mathfrak a}} \nc{\frakB}{{\mathfrak B}}
\nc{\frakb}{{\mathfrak b}} \nc{\frakd}{{\mathfrak d}}
\nc{\oD}{\overline{D}}
\nc{\frakF}{{\mathfrak F}} \nc{\frakg}{{\mathfrak g}}
\nc{\frakm}{{\mathfrak m}} \nc{\frakM}{{\mathfrak M}}
\nc{\frakMo}{{\mathfrak M}^0} \nc{\frakp}{{\mathfrak p}}
\nc{\frakS}{{\mathfrak S}} \nc{\frakSo}{{\mathfrak S}^0}
\nc{\fraks}{{\mathfrak s}} \nc{\os}{\overline{\fraks}}
\nc{\frakT}{{\mathfrak T}}
\nc{\oT}{\overline{T}}
\nc{\frakX}{{\mathfrak X}} \nc{\frakXo}{{\mathfrak X}^0}
\nc{\frakx}{{\mathbf x}}
\nc{\frakTx}{\frakT}      
\nc{\frakTa}{\frakT^a}        
\nc{\frakTxo}{\frakTx^0}   
\nc{\caltao}{\calt^{a,0}}   
\nc{\ox}{\overline{\frakx}} \nc{\fraky}{{\mathfrak y}}
\nc{\frakz}{{\mathfrak z}} \nc{\oX}{\overline{X}}

\font\cyr=wncyr10


\title[]{Classification of compatible left-symmetric conformal algebras on solvable Lie conformal algebras of rank two}

\author{Yuanmeng Zhai}
\address{School of Mathematics, Hangzhou Normal University,
Hangzhou 311121, PR China}
\email{2025111029019@stu.hznu.edu.cn}

\author{Yanyong Hong}
\address{School of Mathematics, Hangzhou Normal University,
Hangzhou 311121, PR China}
\email{yyhong@hznu.edu.cn}

\subjclass[2010]{17A30, 17D25, 17A60, 17B69} \keywords{left-symmetric conformal algebras, solvable Lie conformal algebras, isomorphism classes, Novikov conformal algebras}

\begin{abstract}
In this paper, we give a complete classification of compatible left-symmetric conformal algebra structures on solvable Lie conformal algebras which are free of rank two as $\mathbb{C}[\partial]$-modules up to isomorphism. We obtain ten mutually non-isomorphic normal-form types and determine the precise equivalence relations among the parameters within each family. Moreover, we also characterize the Novikov conformal algebras among the classified left-symmetric conformal algebras, thereby obtaining a complete classification of compatible Novikov conformal algebra structures in the same setting.

\end{abstract}

\maketitle




\allowdisplaybreaks

\section{Introduction}
Lie conformal algebras were introduced by V. Kac \cite{K1} as an axiomatic
framework for the singular part of the operator product expansion in
conformal field theory \cite{BPZ}. Also known as vertex Lie algebras
\cite{P}, they are closely connected with vertex algebras \cite{K1},
infinite-dimensional Lie algebras satisfying the locality property
\cite{KacL}, and the Hamiltonian formalism for nonlinear evolution
equations \cite{BDK}. A Lie conformal algebra is said to be finite if it
is finitely generated as a $\mathbb{C}[\partial]$-module. The structure theory,
representation theory, and cohomology of finite Lie conformal algebras have been
extensively developed (see \cite{DK, CK, BKV, DK1}). Moreover, Lie conformal algebras which are free of rank one and rank two as $\mathbb{C}[\partial]$-modules
have been classified up to isomorphism in \cite{DK} and \cite{K2, BCH,
HL3}, respectively.

The notion of left-symmetric conformal algebras was introduced in \cite{HL1} to investigate the compatible left-symmetric algebra structures on formal distribution Lie algebras and to construct vertex algebras. Left-symmetric conformal algebras can also be viewed as left-symmetric algebras in a pseudo-tensor category \cite{Wu}. It was shown in \cite{HB}
that left-symmetric conformal algebras are closely related to the
conformal classical  Yang--Baxter equation and $\mathcal{O}$-operators on
Lie conformal algebras. Several works on the structure theory of
left-symmetric conformal algebras have appeared (see \cite{HL2,
X1, XH}).  Furthermore, as an important class of left-symmetric conformal algebras, Novikov conformal algebras have also been studied in \cite{KL, KN}. An important property of left-symmetric conformal algebras is
that the commutator of a left-symmetric conformal algebra defines a Lie
conformal algebra. Therefore, a natural approach to studying left-symmetric
conformal algebras is to investigate the compatible left-symmetric
conformal algebra structures on known Lie conformal algebras. In
\cite{HL2}, the left-symmetric conformal algebras which are free of rank one as $\mathbb{C}[\partial]$-modules were
classified, and the compatible left-symmetric conformal algebra structures
on finite semisimple Lie conformal algebras were characterized. As for
ranks two and three, the compatible left-symmetric conformal algebra
structures on the Lie conformal algebra $W(a,b)$ and on the
Schr\"odinger--Virasoro conformal algebra and its subalgebra, under
certain natural conditions, have been classified in \cite{L1, W2, LSW} respectively.
These classification results indicate that left-symmetric conformal
algebras which are free of rank two as $\mathbb{C}[\partial]$-modules are extremely diverse. Therefore, it seems difficult to give a complete classification of left-symmetric conformal algebras which are free of rank two as $\mathbb{C}[\partial]$-modules up to isomorphism.

In this paper, we classify compatible left-symmetric conformal algebra structures on  solvable Lie conformal algebras which are free of rank two as $\mathbb{C}[\partial]$-modules. Note
that a compatible left-symmetric conformal algebra is naturally
a conformal module of its sub-adjacent Lie conformal algebra. Using the conformal analogue of Lie's Theorem established in \cite{DK}, 
 we reduce the classification problem to solving a system of nine polynomial equations. This yields twelve explicit presentation families (see Theorem \ref{thm-1}), which are subsequently reduced to ten mutually non-isomorphic normal-form types. Then we determine the precise isomorphism relations among the parameters within each parametrized type see Theorem \ref{pp4.5}). Finally, we characterize the Novikov conformal algebras in those normal forms and thereby obtain a complete classification, up to isomorphism, of compatible Novikov conformal algebra structures in the same setting (see Corollary \ref{cor-Nov}).

This paper is organized as follows. In Section 2, we recall the basic
definitions and facts about left-symmetric conformal algebras and Lie
conformal algebras. Section 3 is devoted to the classification of
left-symmetric conformal algebra structures on solvable Lie conformal
algebras which are free of rank two as $\mathbb{C}[\partial]$-modules. In Section 4, we determine the isomorphism classes
of the left-symmetric conformal algebras classified in Section 3 and characterize which left-symmetric conformal algebras are Novikov conformal algebras.

 \vspace{0.1cm}
\noindent {\bf Notations.} Throughout this paper, we denote by $\mathbb{C}$ the set of complex numbers. Moreover, if $A$ is a vector space over $\mathbb{C}$, then we denote the vector space of polynomials in $\lambda$ with coefficients in $A$  by $A[\lambda]$ respectively. For a nonzero polynomial $f(\lambda)\in \mathbb{C}[\lambda]$, we write $\text{deg} f(\lambda)$ for the degree of $f(\lambda)$ in $\lambda$. Similarly, for a nonzero polynomial $f(\lambda, \partial)\in \mathbb{C}[ \lambda, \partial]$, we denote the degrees of $f(\lambda, \partial)$ in $\lambda$  and $\partial$ by $\text{deg}_\lambda f(\lambda, \partial)$ and $\text{deg}_\partial f(\lambda, \partial)$ respectively.

\section{Preliminaries on left-symmetric conformal algebras and Lie conformal algebras}

In this section, we recall the definitions and basic facts about left-symmetric conformal algebras and Lie conformal algebras.

\begin{defi}\cite{K1, HL1}
A {\bf conformal algebra} $R$ is a  $\mathbb{C}[\partial]$-module equipped with a $\lambda$-product $\cdot_\lambda\cdot$, which is a $\mathbb{C}$-bilinear map from $R \times R \to R[\lambda]$  satisfying
\begin{align*}
&\partial a_\lambda b = -\lambda a_\lambda b, \quad a_\lambda \partial b = (\partial + \lambda)a_\lambda b,\;\;a, b\in R.
  &&\text{(conformal sesquilinearity)}
\end{align*}

A \textbf{left-symmetric conformal algebra} $R$ is a conformal algebra with a $\mathbb{C}$-bilinear map $\cdot_\lambda\cdot: R\times R\rightarrow R[\lambda]$ satisfying the following axiom
\begin{align*}
&(a_\lambda b)_{\lambda+\mu}c - a_\lambda(b_\mu c) = (b_\mu a)_{\lambda+\mu}c - b_\mu(a_\lambda c),\;\;a, b, c\in R.
  &&\text{(left-symmetric identity)}
\end{align*}
We denote it by $(R,\cdot_\lambda\cdot)$.

A \textbf{Novikov conformal algebra} \(R\) is a left-symmetric conformal algebra satisfying the additional axiom
\begin{align*}
&(a_\lambda b)_{\lambda+\mu}c = (a_\lambda c)_{-\mu-\partial}b,\;\;a, b, c\in R.
\end{align*}

A \textbf{Lie conformal algebra} $R$ is a conformal algebra with a $\mathbb{C}$-bilinear map $[\cdot_\lambda\cdot]: R\times R\rightarrow R[\lambda]$, called the {\bf $\lambda$-bracket}, satisfying for all $a,b,c\in R$:
\begin{align*}
[a_\lambda b] &= -[b_{-\lambda-\partial}a], \qquad &&\text{(skew-symmetry)} \\
[a_\lambda [b_\mu c]] &= [[a_\lambda b]_{\lambda+\mu}c] + [b_\mu [a_\lambda c]]. \qquad &&\text{(Jacobi identity)}
\end{align*}

A conformal algebra is said to be \textbf{finite} if it is finitely generated as a $\mathbb{C}[\partial]$-module. The {\bf rank} of a conformal algebra is its rank as a $\mathbb{C}[\partial]$-module.

\end{defi}

Let $(R, \cdot_\lambda \cdot)$ be a left-symmetric conformal algebra and $I$ be a $\mathbb{C}[\partial]$-submodule of $R$. If $I_\lambda R \subseteq I[\lambda]$ and $R_\lambda I \subseteq I[\lambda]$, then $I$ is called an {\bf ideal} of $(R, \cdot_\lambda \cdot)$. The left-symmetric conformal algebra $(R, \cdot_\lambda \cdot)$ is called {\bf simple} if it is non-trivial and has no proper ideals. Analogously, the notions of ideals and simplicity can be defined for Lie conformal algebras.

Let $(R, \cdot_\lambda \cdot)$ and $(S, \cdot_\lambda \cdot)$ be two (either left-symmetric or Lie) conformal algebras . Recall that a {\bf homomorphism} from $(R, \cdot_\lambda \cdot)$ to $(S, \cdot_\lambda \cdot)$ is a $\mathbb{C}[\partial]$-module homomorphism $\varphi: R\rightarrow S$ such that $\varphi(a_\lambda b)=\varphi(a)_\lambda \varphi(b)$ for all $a$, $b\in R$. Moreover, if $\varphi$ is bijective, then $\varphi$ is an {\bf isomorphism} from $(R, \cdot_\lambda \cdot)$ to $(S, \cdot_\lambda \cdot)$, and we say that $(R, \cdot_\lambda \cdot)$ and $(S, \cdot_\lambda \cdot)$ are {\bf isomorphic}.

\delete{A $\mathbb{C}[\partial]$-submodule $I$ of a left-symmetric conformal algebra $R$ is called an \textbf{ideal} if $I_\lambda R \subseteq I[\lambda]$ and $R_\lambda I \subseteq I[\lambda]$.
A left-symmetric conformal algebra $R$ is called \textbf{simple} if it is non-trivial and has no proper ideals.

A homomorphism between left-symmetric conformal algebras $(R,\cdot_\lambda\cdot)$ and $(S,\cdot_\lambda\cdot)$ is a $\mathbb{C}[\partial]$-module homomorphism $\varphi: R \to S$ such that
\[
\varphi(a_\lambda b) = \varphi(a)_\lambda \varphi(b) \quad \text{for all } a,b\in R.
\]
An \textbf{isomorphism} is a bijective homomorphism; if one exists, $R$ and $S$ are \textbf{isomorphic}.}

\delete{\begin{defi}\cite{K2, H1}
A \textbf{Lie conformal algebra} $R$ is a $\mathbb{C}[\partial]$-module  together with a $\mathbb{C}$-bilinear map
$[\cdot_\lambda\cdot] : R \times R \to R[\lambda]$, called the $\lambda$-bracket, satisfying for all $a,b,c\in R$:
\begin{align*}
[\partial a_\lambda b] &= -\lambda [a_\lambda b], \quad [a_\lambda \partial b] = (\partial + \lambda)[a_\lambda b],
  &&\text{(conformal sesquilinearity)} \\
[a_\lambda b] &= -[b_{-\lambda-\partial}a], \qquad &&\text{(skew-symmetry)} \\
[a_\lambda [b_\mu c]] &= [[a_\lambda b]_{\lambda+\mu}c] + [b_\mu [a_\lambda c]]. \qquad &&\text{(Jacobi identity)}
\end{align*}}

\begin{defi}\cite{K1}
A {\bf conformal module} $V$ of a Lie conformal algebra $(R, [\cdot_\lambda \cdot])$
is a $\mathbb{C}[\partial]$-module equipped with a $\mathbb{C}$-bilinear map
$R\times V\rightarrow
V[\lambda]$, $a\times v\mapsto a_\lambda v$, satisfying the following relations for all $a,b\in R$ and $v\in V$,
\begin{eqnarray}
&&\label{conf-1}(\partial a)_\lambda v=-\lambda a_\lambda v,\ a_\lambda(\partial
v)=(\partial+\lambda)a_\lambda v,\\
&&\label{conf-2}a_\lambda(b_\mu v)-b_\mu(a_\lambda v)=[a_\lambda b]_{\lambda+\mu}v.
\end{eqnarray}
\end{defi}

\begin{pro}\label{pro-cp}\cite[Proposition 2.5]{HL1}, \cite[Corollary 3.9]{HL2} Suppose that $(R,\cdot_\lambda\cdot)$ is a left-symmetric conformal algebra. Then the $\lambda$-bracket
\[
[a_\lambda b] = a_\lambda b - b_{-\lambda-\partial}a, \qquad a,b\in R,
\]
defines a Lie conformal algebra structure on $R$. We denote this Lie conformal algebra by $\mathfrak{g}(R)$. Moreover, $R$ is a conformal module of $\mathfrak{g}(R)$ with the $\lambda$-product of $(R,\cdot_\lambda\cdot)$.
\end{pro}
\begin{rmk}
In the above proposition, $\mathfrak{g}(R)$ is called the \textbf{sub-adjacent Lie conformal algebra} of $(R,\cdot_\lambda\cdot)$, and $(R,\cdot_\lambda\cdot)$ is also called a \textbf{compatible left-symmetric conformal algebra structure} on the Lie conformal algebra $\mathfrak{g}(R)$.
\end{rmk}

\begin{defi}\cite{DK}
Let $(R, [\cdot_\lambda \cdot])$ be a Lie conformal algebra.
Write the $\lambda$-bracket as a formal power series
\[
[a_\lambda b] = \sum_{n\ge 0} \frac{\lambda^n}{n!} (a_{(n)}b),
\]
where $a_{(n)}b\in R$ is called the \textbf{$n$-th product} of $a$ and $b$.
The \textbf{derived algebra} $R'$ is the $\mathbb{C}[\partial]$-span of all elements $a_{(n)}b$ with $a,b\in R$, $n\ge 0$.
Inductively set $R^{(1)} = R'$ and $R^{(k+1)} = (R^{(k)})'$ for $k\ge 1$.
$R$ is a \textbf{solvable} Lie conformal algebra if $R^{(n)} = 0$ for some $n\ge 1$.
An ideal $I$ of $R$ is {\bf solvable} if it is solvable as a Lie conformal algebra.
\end{defi}

\begin{pro}\label{pro-LT}\cite[Corollary 8.3]{DK} Let $(R, [\cdot_\lambda \cdot])$ be a finite solvable Lie conformal algebra and $V$ be a finite conformal module of $R$. Then there exists a family of submodules
\begin{eqnarray*}
\text{Tor}(V)=V_0\subset V_1\subset\cdots\subset V_N=V
\end{eqnarray*}
such that $V_{i+1}/V_i$ is free of rank one for $i\geq 0$, where $\text{Tor}(V)=\{v\in V~|~\exists f(\partial)\in \mathbb{C}[\partial]\backslash \{0\} ~\text{such that} $ $f(\partial)v=0\}$. Moreover, the action of $R$ on each $V_{i+1}/V_i=\mathbb{C}[\partial]\overline{v_{i+1}}$ is of the following form
\begin{eqnarray*}
a_\lambda \overline{v_{i+1}}=\phi_a(\lambda)\overline{v_{i+1}},\;\;a\in R,
\end{eqnarray*}
where $\phi_a(\lambda)\in \mathbb{C}[\lambda]$ for each $a\in R$.

\end{pro}

\delete{Let $R$ be a (Lie or left-symmetric) conformal algebra. Define $\operatorname{Coeff}(R)$ to be the quotient of the vector space with basis $\{a_n \mid a\in R,\ n\in\mathbb{Z}\}$ by the subspace spanned over $\mathbb{C}$ by the following elements:
\[
(\alpha a)_n - \alpha a_n,\quad (a+b)_n - a_n - b_n,\quad (\partial a)_n + n a_{n-1},
\]
for all $a,b\in R$, $\alpha\in\mathbb{C}$, $n\in\mathbb{Z}$.
On $\operatorname{Coeff}(R)$ we define an operation by
\[
a_m \cdot b_n = \sum_{j\in\mathbb{N}} \binom{m}{j} (a_{(j)}b)_{m+n-j},
\]
where $a_{(j)}b$ are the coefficients in the formal power series expansion $[a_\lambda b] = \sum_{j\ge0} \frac{\lambda^j}{j!} (a_{(j)}b)$ (or $a_\lambda b = \sum_{j\ge0} \frac{\lambda^j}{j!} (a_{(j)}b)$ in the left-symmetric case). Then $\operatorname{Coeff}(R)$ becomes a (Lie or left-symmetric) algebra (see \cite{H1}).
\begin{pro}
\cite[Theorem 4.7]{H1}
If $A$ is a compatible left-symmetric conformal algebra on a Lie conformal algebra $R$, then $\operatorname{Coeff}(A)$ is a compatible left-symmetric algebra on the Lie algebra $\operatorname{Coeff}(R)$.
\end{pro}}


\section{Compatible left-symmetric conformal algebra structures on solvable Lie conformal algebras of rank two }\mbox{}

In this section, we give a classification of compatible left-symmetric conformal algebra structures on solvable Lie conformal algebras which are free of rank two as $\mathbb{C}[\partial]$-modules.

\begin{pro}
Let $\mathfrak{g}(R)$ be a solvable Lie conformal algebra, which is free of rank two as a $\mathbb{C}[\partial]$-module. Suppose that $(R, \cdot_\lambda \cdot)$ is a compatible left-symmetric conformal algebra on $\mathfrak{g}(R)$. Then there exists a $\mathbb{C}[\partial]$-basis $x$ and $y$ such that $R=\mathbb{C}[\partial] x \oplus \mathbb{C}[\partial] y$ with the following $\lambda$-products:
\begin{eqnarray}
\label{29-thq}&&x_{\lambda}x=f_{1}(\lambda)x,\\
\label{30-thq}&&x_{\lambda}y=A(\lambda,\partial)x+f_{2}(\lambda)y,\\
\label{31-thq}&&y_{\lambda}x=g_{1}(\lambda)x,\\
\label{32-thq}&&y_{\lambda}y=B(\lambda,\partial)x+g_{2}(\lambda)y,
 \end{eqnarray}
 where $f_{1}(\lambda)$, $f_{2}(\lambda)$, $g_{1}(\lambda)$, $g_{2}(\lambda)\in\mathbb{C} [\lambda]$ and $A(\lambda,\partial)$, $B(\lambda,\partial)\in\mathbb{C} [\lambda,\partial]$.
\end{pro}
\begin{proof}
By Proposition \ref{pro-cp}, $R$ is a conformal module of  $\mathfrak{g}(R)$. Since $R$ is free as a $\mathbb{C}[\partial]$-module, it has no torsion, i.e. $\text{Tor}(R)=0$. Then by Proposition \ref{pro-LT}, there exists a submodule $V_1$  such that
$0\subset V_1\subset R$ and $V_1$, $R/V_1$ are free of rank one as $\mathbb{C}[\partial]$-modules. Therefore, we assume that $V_1=\mathbb{C}[\partial]x$ and $R/V_1=\mathbb{C}[\partial]\overline{y}$ for some $x$, $y\in R$. Then we have $R=\mathbb{C}[\partial]x\oplus\mathbb{C}[\partial]y$.
By Proposition \ref{pro-LT}, we can assume
\begin{eqnarray*}
x_\lambda x=f_{1}(\lambda)x,\;\;y_{\lambda}x=g_{1}(\lambda)x,\;\;x_\lambda \overline{y}=f_{2}(\lambda)\overline{y},\;\;y_\lambda \overline{y}=g_{2}(\lambda)\overline{y},
\end{eqnarray*}
for some $f_{1}(\lambda)$, $f_{2}(\lambda)$, $g_{1}(\lambda)$ and $g_{2}(\lambda)\in\mathbb{C} [\lambda]$. Hence there exist $A(\lambda,\partial)$, $B(\lambda,\partial)\in\mathbb{C} [\lambda,\partial]$ such that $x_\lambda y=A(\lambda, \partial)x+f_2(\lambda)y$ and $y_{\lambda}y=B(\lambda,\partial)x+g_{2}(\lambda)y$. Then this conclusion follows directly.
\end{proof}

By the definition of left-symmetric conformal algebras, $(R, \cdot_\lambda \cdot)$ with the $\lambda$-product defined by Eqs.~(\ref{29-thq})--(\ref{32-thq}) is a left-symmetric conformal algebra if and only if the following identities hold:
\begin{eqnarray}
    \label{33-thq}&&f_{1}(\lambda)f_1(\lambda+\mu)=f_{1}(\mu)f_1(\lambda+\mu),\\
    \label{34-thq}&&f_{1}(\lambda)f_2(\lambda+\mu)=f_{1}(\mu)f_2(\lambda+\mu),\\
    \label{35-thq}&&f_1(\lambda+\mu)g_{1}(\lambda)=f_1(\lambda+\mu)A(\mu,-\lambda-\mu)+f_{2}(\mu)g_1(\lambda+\mu),\\
    \label{38-thq}&&f_2(\lambda+\mu)g_{1}(\mu)=f_2(\lambda+\mu)A(\lambda,-\lambda-\mu)+f_{2}(\lambda)g_2(\lambda+\mu),\\
    \label{39-thq}&&f_1(\lambda+\mu)B(\lambda,-\lambda-\mu)+g_1(\lambda+\mu)g_{2}(\lambda)=f_1(\lambda+\mu)B(\mu,-\lambda-\mu)+g_1(\lambda+\mu)g_{2}(\mu),\\
    \label{40-thq}&&f_2(\lambda+\mu)B(\lambda,-\lambda-\mu)+g_2(\lambda)g_{2}(\lambda+\mu)=f_2(\lambda+\mu)B(\mu,-\lambda-\mu)+g_2(\lambda+\mu)g_{2}(\mu),\\
   \label{41-thq}&&f_{1}(\lambda)A(\lambda+\mu,\partial)-f_{1}(\lambda)A(\mu,\lambda+\partial)-f_{2}(\mu)A(\lambda,\partial)\notag\\
   &&\quad=f_{1}(\mu)A(\lambda+\mu,\partial)-f_{1}(\mu)A(\lambda,\mu+\partial)-f_{2}(\lambda)A(\mu,\partial),\\
    \label{43-thq}&&g_{1}(\mu)A(\lambda+\mu,\partial)-g_{1}(\mu)A(\lambda,\mu+\partial)-f_{2}(\lambda)B(\mu,\partial)\notag\\
    &&\quad=A(\lambda,-\lambda-\mu)A(\lambda+\mu,\partial)+f_{2}(\lambda)B(\lambda+\mu,\partial)-f_{1}(\lambda)B(\mu,\lambda+\partial)-g_{2}(\mu)A(\lambda,\partial),\\
    \label{44-thq}&&A(\lambda+\mu,\partial)B(\lambda,-\lambda-\mu)+g_{2}(\lambda)B(\lambda+\mu,\partial)-g_{1}(\lambda)B(\mu,\lambda+\partial)-g_{2}(\mu)B(\lambda,\partial)\notag\\
    &&\quad=A(\lambda+\mu,\partial)B(\mu,-\lambda-\mu)+g_{2}(\mu)B(\lambda+\mu,\partial)-g_{1}(\mu)B(\lambda,\mu+\partial)-g_{2}(\lambda)B(\mu,\partial).\end{eqnarray}
Therefore, the classification of compatible left-symmetric conformal algebra structures on solvable Lie conformal algebras which are free of rank two as $\mathbb{C}[\partial]$-modules is equivalent to finding all solutions of Eqs. (\ref{33-thq})-(\ref{44-thq}).

\begin{lem}
For $f_{1}(\lambda)$, we have $f_{1}(\lambda)=0$ or $f_{1}(\lambda)=a \in \mathbb{C} \setminus \{0\}$.
\delete{For $f_{1}(\lambda)$, there are the following two cases:\\
\noindent{Case 1. $f_{1}(\lambda)=0$;\\
Case 2. $f_{1}(\lambda)=a \in \mathbb{C}^\times$.}}
\end{lem}
\begin{proof}
  It is straightforward by Eq.~(\ref{33-thq}). \delete{ From Eq.~(\ref{33-thq}) we have $[f_{1}(\lambda)-f_{1}(\mu )]f_{1}(\lambda+\mu)=0$, which implies that either $f_{1}(\lambda)=0$ or $f_{1}(\lambda)=f_{1}(\mu)$ for all $\lambda,\mu$, i.e. $f_{1}$ is a constant function. Therefore, we get this lemma. Without loss of generality, let $f_{1}(\lambda)=a \in \mathbb{C} \setminus \{0\}$ when $f_{1}(\lambda) \neq 0$.}
\end{proof}
\begin{lem}\label{lem-l2}
If $f_1(\lambda)=0$, then we have the following cases:

\noindent \text{(A)} $g_{1}(\lambda)=0$;\\
\text{(B)} $g_{1}(\lambda) \neq 0$, \ $g_{2}(\lambda)=b \in \mathbb{C}$, \ $f_{2}(\lambda)=0$.
\end{lem}
\begin{proof}
If $f_{1}(\lambda)=0$, then Eqs.~(\ref{35-thq}) and (\ref{39-thq}) can be simplified into
\begin{eqnarray}
    \label{45-thq}&& f_{2}(\mu)g_{1}(\lambda+\mu)=0,\\
    \label{47-thq}&& g_{1}(\lambda+\mu)g_{2}(\lambda)=g_{1}(\lambda+\mu)g_{2}(\mu).
\end{eqnarray}
By Eq.~(\ref{47-thq}), we have $(g_{2}(\lambda)-g_{2}(\mu)) g_{1}(\lambda+\mu)=0$,
which implies that either $g_{1}(\lambda)=0$ or
$g_{2}(\lambda)=g_{2}(\mu)$,
i.e. $g_{2}(\lambda)=b\in \mathbb{C}$.
By Eq.~(\ref{45-thq}), we obtain $g_{1}(\lambda)=0$ or $f_{2}(\lambda)=0$. Then this conclusion holds.
\end{proof}
Now, we split Case (A) of Lemma \ref{lem-l2} into two subcases:
$f_{2}(\lambda)=0$ and $f_{2}(\lambda) \neq 0$.
\begin{pro}\label{pro-p1}
In Case (A) of Lemma \ref{lem-l2}, if $f_{2}(\lambda)=0$, then we have the following cases:\\
\noindent \textup{(A1)} $g_{2}(\lambda)=A(\lambda,\partial)=0$, $B(\lambda,\partial)$ can be an arbitrary polynomial in $\mathbb{C}[\lambda, \partial]$;\\
\textup{(A2)} $g_{2}(\lambda)=b\in\mathbb{C}\setminus \{0\}$,  $A(\lambda,\partial)=0$, $B(\lambda,\partial)=V(\partial) \in \mathbb{C}[\partial]$;
\delete{where $V(\partial)$ is an arbitrary polynomial in $\partial$;}\\
\textup{(A3)} $g_{2}(\lambda)=A(\lambda,\partial)=b \in \mathbb{C}\setminus\{0\}$, $B(\lambda,\partial)=U(\lambda)+V(\partial)$,
where $U(\lambda) \in \mathbb{C}[\lambda]$ and $V(\partial) \in \mathbb{C}[\partial]$.
\end{pro}
\begin{proof}
In Case (A) of Lemma \ref{lem-l2}, when $f_{2}(\lambda)=0$,
we have $f_{1}(\lambda)=g_{1}(\lambda)=f_{2}(\lambda)=0$. Then
Eqs.~(\ref{33-thq})--(\ref{44-thq}) can be simplified into
\begin{eqnarray}
    \label{48-thq}&&g_{2}(\lambda)g_{2}(\lambda+\mu)=g_{2}(\mu)g_{2}(\lambda+\mu),\\
    \label{50-thq}&&g_{2}(\mu)A(\lambda,\partial)=A(\lambda,-\lambda-\mu)A(\lambda+\mu,\partial ),\\
    \label{51-thq}&&A(\lambda+\mu,\partial)B(\lambda,-\lambda-\mu)+g_{2}(\lambda)B(\lambda+\mu,\partial)-g_{2}(\mu)B(\lambda,\partial)\notag\\
    &&\quad=A(\lambda+\mu,\partial)B(\mu,-\lambda-\mu)+g_{2}(\mu)B(\lambda+\mu,\partial)-g_{2}(\lambda)B(\mu,\partial).
\end{eqnarray}
By Eq.~(\ref{48-thq}) we have $(g_{2}(\lambda)-g_{2}(\mu))g_{2}(\lambda+\mu)=0$, which implies that either $g_{2}(\lambda)=0$ or $g_{2}(\lambda)=g_{2}(\mu)$, i.e. $g_{2}(\lambda)=b \in \mathbb{C}\setminus\{0\}$.

If $g_{2}(\lambda)=0$, Eq.~(\ref{50-thq}) becomes
\begin{eqnarray}
    \label{53-thq}&&A(\lambda,-\lambda-\mu)A(\lambda+\mu,\partial )=0.
\end{eqnarray}
By Eq. (\ref{53-thq}), we have $A(\lambda,\partial)=0$. Then Eq. (\ref{51-thq}) naturally holds. Hence $B(\lambda,\partial)$ can be an arbitrary polynomial in $\mathbb{C}[\lambda,\partial]$. Therefore, we obtain Case (A1).

\delete{We claim $A(\lambda,\partial)=0$.
Suppose, to the contrary, that $A(\lambda,\partial)\neq 0$.
Since the change of variables $(\lambda,\mu,\partial)\mapsto(\lambda+\mu,\partial)$ is an automorphism of
$\mathbb{C}[\lambda,\mu,\partial]$, the polynomial $A(\lambda+\mu,\partial)$ is also nonzero.
From $A(\lambda,-\lambda-\mu)A(\lambda+\mu,\partial)=0$ and the fact that a polynomial ring over a field
is an integral domain, we must have $A(\lambda,-\lambda-\mu)=0$.
Write $A(\lambda,\partial)=\sum_i A_i(\lambda)\,\partial^i$ with $A_i(\lambda)\in\mathbb{C}[\lambda]$.
Then $0=A(\lambda,-\lambda-\mu)=\sum_i A_i(\lambda)(-\lambda-\mu)^i.$ For any fixed $\lambda$, the right side is a polynomial in $\mu$. Because it vanishes identically, all its coefficients must be zero. Expanding $(-\lambda-\mu)^i$ and comparing coefficients shows that $A_i(\lambda)=0$ for every $i$ and every $\lambda$. Hence $A(\lambda,\partial)=0$, contradicting our assumption.
Therefore $A(\lambda,\partial)=0$. In the case $g_{2}(\lambda)=0$ and $A(\lambda,\partial)=0$, Eq.~(\ref{51-thq}) holds identically, so there is no restriction on $B(\lambda,\partial)$. Hence $B(\lambda,\partial)$ can be an arbitrary polynomial in $\mathbb{C}[\lambda,\partial]$. Therefore, according to the above discussion, we have Case(A1)  in this theorem.}

If $g_{2}(\lambda)=b \in \mathbb{C}\setminus\{0\}$, Eq.~(\ref{50-thq}) can be simplified into
\begin{eqnarray}
    \label{55-thq}&&A(\lambda,-\lambda-\mu)A(\lambda+\mu,\partial )=bA(\lambda,\partial).
\end{eqnarray}
Obviously, if $A(\lambda,\partial)=0$, then Eq. (\ref{55-thq}) holds.
If $A(\lambda,\partial)\neq 0$, by comparing the degree of $\mu$, we have $\text{deg}_\lambda(A(\lambda, \partial))+\text{deg}_\partial(A(\lambda, \partial))=0$. Then we have $\text{deg}_\lambda(A(\lambda, \partial))=\text{deg}_\partial(A(\lambda, \partial))=0$. Therefore, $A(\lambda, \partial)=c\in \mathbb{C}$. Then by  Eq. (\ref{55-thq}), we obtain $c=b$.  Therefore, we have $A(\lambda,\partial)=0$ or $A(\lambda,\partial)=b$.

\delete{We claim that $A(\lambda,\partial)=0$ or $A(\lambda,\partial)=b$.
If $A(\lambda,\partial)=0$, the statement holds.
Now we discuss the case $A(\lambda,\partial)\neq 0$.
Let $A(\lambda,\partial)$ as a polynomial in $\partial$:
$A(\lambda,\partial)=\sum_{i=0}^{m}A_i(\lambda)\partial^i$,
where $A_m(\lambda)\neq0$ (since $A(\lambda,\partial)\neq 0$).
We assert that $m=0$.
Suppose, to the contrary, that $m\geq1$.
Substituting $A(\lambda,\partial)=\sum_{i=0}^{m}A_i(\lambda)\partial^i$
into Eq.~(\ref{55-thq}) and comparing the coefficients of $\partial^m$ yields
\begin{eqnarray}
    \label{56-thq}&&A(\lambda,-\lambda-\mu)\,A_m(\lambda+\mu)=b\,A_m(\lambda).
\end{eqnarray}
Comparing the degrees of $\mu$ on both sides shows that
$A_m(\lambda+\mu)$ must be a constant, and $A(\lambda,-\lambda-\mu)$
cannot involve $\mu$, i.e.\
$A(\lambda,-\lambda-\mu)=A_0(\lambda)$ with $A_0(\lambda)\in\mathbb{C}[\lambda]$.
This contradicts $m\geq1$, hence $m=0$.
Thus, in this case, $A(\lambda,\partial)=A_0(\lambda)$
(we denote it by $A_0(\lambda)$ for clarity),
$A_0(\lambda)\in\mathbb{C}[\lambda]$ and $A_0(\lambda)\neq0$.
Plugging this into Eq.~(\ref{55-thq}) gives
$A_0(\lambda)\,A_0(\lambda+\mu)=b\,A_0(\lambda)$.
Since $A_0(\lambda)\neq0$, we can cancel it to obtain
$A_0(\lambda+\mu)=b$, which means
$A(\lambda,\partial)=A_0(\lambda)=b$.}

When $f_1(\lambda)=f_2(\lambda)=g_1(\lambda)=A(\lambda, \partial)=0$ and $g_{2}(\lambda)=b\neq 0 $,  Eqs.~(\ref{33-thq})--(\ref{44-thq}) can be simplified into
\begin{eqnarray}
    \label{57-thq}&&bB(\lambda,\partial)=bB(\mu,\partial).
\end{eqnarray}
Since $b \neq 0$, we have $B(\lambda,\partial)=B(\mu,\partial)$,
which means that $B(\lambda,\partial)=V(\partial)$,
where $V(\partial)$ is an arbitrary polynomial in $\mathbb{C}[\partial]$. Therefore, we have Case (A2).

When $f_1(\lambda)=f_2(\lambda)=g_1(\lambda)=0$ and $g_{2}(\lambda)=A(\lambda,\partial)=b \neq0$, Eqs.~(\ref{33-thq})--(\ref{44-thq}) can be simplified into
\delete{\begin{eqnarray*}
    &&B(\lambda,-\lambda-\mu)-B(\lambda,\partial)=B(\mu,-\lambda-\mu)-B(\mu,\partial),
\end{eqnarray*}
i.e.}
\begin{eqnarray}
    \label{58-thq}&&B(\lambda,-\lambda-\mu)-B(\mu,-\lambda-\mu)=B(\lambda,\partial)-B(\mu,\partial).
\end{eqnarray}
If $B(\lambda,\partial)=0$, then Eq.~(\ref{58-thq}) naturally holds. We now consider the case $B(\lambda,\partial) \neq 0$.
Let $\mu=0$ in Eq. (\ref{58-thq}). Then we obtain
\begin{eqnarray}\label{eq-p1}
B(\lambda,\partial)=B(\lambda,-\lambda)-B(0,-\lambda)+B(0,\partial).
\end{eqnarray}
Setting $U(\lambda)=B(\lambda,-\lambda)-B(0,-\lambda)$ and $V(\partial)=B(0,\partial)$, we have $B(\lambda,\partial)=U(\lambda)+V(\partial)$. Then one can check that Eq. (\ref{58-thq}) holds. Therefore, $U(\lambda)\in\mathbb{C}[\lambda]$ and $V(\partial)\in\mathbb{C}[\partial]$ can be arbitrary polynomials. Then we have Case (A3).

The proof is completed.

\delete{Let $B(\lambda,\partial)=\sum_{i=0}^{p}B_i(\lambda)\partial^i$,
where $B_p(\lambda)\neq0$ (since $B(\lambda,\partial)\neq 0$). If \(p = 0\), Eq.~(\ref{58-thq}) is automatically satisfied.  We now consider the case \(p \ge 1\). Substituting $B(\lambda,\partial)=\sum_{i=0}^{p}B_i(\lambda)\partial^i$ into Eq.~(\ref{58-thq})  yields
\begin{eqnarray}
    \label{59-thq}&&\sum_{i=0}^{p}\bigl[B_i(\lambda)-B_i(\mu)\bigr](-\lambda-\mu)^i
= \sum_{i=0}^{p}\bigl[B_i(\lambda)-B_i(\mu)\bigr]\partial^i.
\end{eqnarray}
By comparing the coefficients of $\partial^i$, we see that for all $i\ge 1, B_i(\lambda)-B_i(\mu)=0$ for all $\lambda,\mu$,
 each $B_i(\lambda)$ ($i\ge 1$) is a constant; let $B_i(\lambda)=b_i\in\mathbb{C}$. Hence $ B(\lambda,\partial)=B_0(\lambda)+\sum_{i=1}^{p} b_i\,\partial^i.$
Notice that the constant term of $B_0(\lambda)$ can be absorbed into the polynomial in $\partial$.}

\delete{
Therefore, $B(\lambda,\partial)$ takes the form
\begin{eqnarray}
    \label{60-thq}&&B(\lambda,\partial)=U(\lambda)+V(\partial),
\end{eqnarray}
where $U(\lambda)\in\mathbb{C}[\lambda]$ and $V(\partial)\in\mathbb{C}[\partial]$ are arbitrary polynomials. Therefore, we have Case (A3) in this theorem.}
\end{proof}
\begin{lem}\label{lem-l3}
Suppose that $f_2(\lambda)\neq 0$ and $A(\lambda,\partial)\in\mathbb{C}[\lambda,\partial]$ satisfies
\begin{eqnarray}\label{eq-q2}f_2(\lambda)A(\mu,\partial)=f_2(\mu)A(\lambda,\partial).
\end{eqnarray}
Then there exists a polynomial $D(\partial)\in\mathbb{C}[\partial]$ such that $A(\lambda,\partial)=f_2(\lambda)D(\partial)$.
\end{lem}
\begin{proof}
If $A(\lambda,\partial)=0$, then we can take $D(\partial)=0$. Now assume \(A(\lambda,\partial)\neq 0\). Then $f_2(\lambda)$ divides $A(\lambda, \partial)$.
Then we obtain
\[
\frac{A(\lambda,\partial)}{f_2(\lambda)}=\frac{A(\mu,\partial)}{f_2(\mu)}.
\]
The left-hand side does not depend on \(\mu\), and the right-hand side does not depend on \(\lambda\). Hence their common value depends only on \(\partial\). Denote this common value by \(D(\partial)\). Then $A(\lambda,\partial)=f_2(\lambda)D(\partial)$.  Obviously, $A(\lambda,\partial)=f_2(\lambda)D(\partial)$ for any $D(\partial)\in \mathbb{C}[\partial]$ satisfies Eq. (\ref{eq-q2}). Then this conclusion holds.
\delete{Write $f_2(\lambda)=\sum_{i=0}^{s}c_i\lambda^i,c_s\neq0.$
Since \(D(\partial)\neq0\) and does not depend on \(\lambda\), we have $\deg_\lambda A(\lambda,\partial)=\deg_\lambda f_2(\lambda)=s.$
Comparing the coefficients of \(\lambda^s\) on both sides yields $A_s(\partial)=c_sD(\partial),$
where \(A_s(\partial)\) is the coefficient of \(\lambda^s\) in \(A(\lambda,\partial)\). Therefore
$D(\partial)=\frac{A_s(\partial)}{c_s}\in\mathbb{C}[\partial].$
Thus $A(\lambda,\partial)=f_2(\lambda)D(\partial),$
as required.}
\end{proof}
\delete{\begin{proof}
If $A(\lambda,\partial)=0$, the statement holds with $D(\partial)=0$. Now assume $A(\lambda,\partial)\neq 0$.
Let $f_2(\lambda)=\sum_{i=0}^{s} c_i \lambda^i$ with $c_s\neq 0$, and $A(\lambda,\partial)=\sum_{j=0}^{m} A_j(\partial)\lambda^j$ with $A_m(\partial)\neq 0$. For convenience, set $c_i=0$ for $i>s$ and $A_j(\partial)=0$ for $j>m$.
Substituting these expressions into the given identity gives
\begin{eqnarray}
    \label{61-thq}&&\Bigl(\sum_{i=0}^{s} c_i \lambda^i\Bigr)\Bigl(\sum_{j=0}^{m} A_j(\partial)\mu^j\Bigr)
= \Bigl(\sum_{i=0}^{s} c_i \mu^i\Bigr)\Bigl(\sum_{j=0}^{m} A_j(\partial)\lambda^j\Bigr).
\end{eqnarray}
Expanding and comparing the coefficients of $\lambda^i\mu^j$ yields
\begin{eqnarray}
    \label{62-thq}
c_i A_j(\partial) = c_j A_i(\partial) \qquad (\forall i,j).
\end{eqnarray}
We first show that $m=s$.
If $m>s$, then taking $i=m$, $j=s$ in Eq.~(\ref{62-thq}) gives
$c_m A_s(\partial) = c_s A_m(\partial)=0$, a contradiction.
Similarly, if $m<s$, taking $i=m$, $j=s$ gives $c_m A_s(\partial) = c_s A_m(\partial)=0$,  again a contradiction.
Hence $m=s$.
Because $c_m\neq 0$, Eq.~(\ref{62-thq}) with $j=m$ gives
$c_m A_i(\partial) = c_i A_m(\partial)$, i.e.
, $A_i(\partial) = c_i A_m(\partial) / c_m$ for all i.
Define $D(\partial) := A_m(\partial)/c_m\in\mathbb{C}[\partial]$. Then $A_i(\partial) = c_i D(\partial)$ for every $i$, and we obtain $ A(\lambda,\partial) = \sum_{i=0}^{m} c_i D(\partial) \lambda^i
= D(\partial) \sum_{i=0}^{m} c_i \lambda^i = f_2(\lambda) D(\partial).$
The definition of $D(\partial)$ does not depend on the particular choice of the nonzero coefficient:
if $c_k\neq0$, then $c_m A_k(\partial) = c_k A_m(\partial)$, so $A_k(\partial)/c_k = A_m(\partial)/c_m = D(\partial)$.
Thus, the lemma is proved.
\end{proof}}

\begin{pro}\label{pp-3.6}
   In Case (A) in Lemma \ref{lem-l2}, if $f_2(\lambda)\neq 0$, then we have
        $A(\lambda,\partial)=f_{2}(\lambda)D(\partial)$,
        $g_{2}(\lambda)=-f_{2}(\lambda)D(-\lambda)$,
        and $B(\lambda,\partial)=-f_{2}(\lambda)D(-\lambda)D(\partial)$ for some $D(\partial)\in \mathbb{C}[\partial]$.
\end{pro}
\begin{proof}
In Case (A) in Lemma \ref{lem-l2} with $f_2(\lambda)\neq 0$, we have $f_{1}(\lambda)=g_{1}(\lambda)=0$. Then Eqs.~(\ref{33-thq})--(\ref{44-thq}) can be simplified into
\begin{eqnarray}
    \label{63-thq}&&f_{2}(\lambda)A(\mu,\partial)=f_{2}(\mu)A(\lambda,\partial),\\
    \label{64-thq}&&f_{2}(\lambda)g_{2}(\lambda+\mu)+f_{2}(\lambda+\mu)A(\lambda,-\lambda-\mu)=0,\\
    \label{66-thq}&&f_{2}(\lambda)B(\mu,\partial)+f_{2}(\lambda)B(\lambda+\mu,\partial)+A(\lambda,-\lambda-\mu)A(\lambda+\mu,\partial)=g_{2}(\mu)A(\lambda,\partial),\\
    \label{68-thq}&&g_{2}(\lambda)g_{2}(\lambda+\mu)+f_{2}(\lambda+\mu)B(\lambda,-\lambda-\mu)=g_{2}(\mu)g_{2}(\lambda+\mu)+f_{2}(\lambda+\mu)B(\mu,-\lambda-\mu),\\
    \label{69-thq}&&g_{2}(\lambda)B(\lambda+\mu,\partial)+A(\lambda+\mu,\partial)B(\lambda,-\lambda-\mu)-g_{2}(\mu)B(\lambda,\partial)\notag\\
    &&\quad=g_{2}(\mu)B(\lambda+\mu,\partial)+A(\lambda+\mu,\partial)B(\mu,-\lambda-\mu)-g_{2}(\lambda)B(\mu,\partial).
\end{eqnarray}

By Lemma \ref{lem-l3}, Eq.~(\ref{63-thq}) implies $A(\lambda,\partial)=f_{2}(\lambda)D(\partial)$ for some
$D(\partial)\in\mathbb{C}[\partial]$. Substituting this into Eq.~(\ref{64-thq}) yields
$f_{2}(\lambda)\bigl(g_{2}(\lambda+\mu)+f_{2}(\lambda+\mu)D(-\lambda-\mu)\bigr)=0$.
Since $f_{2}(\lambda)\neq 0$, we obtain
\begin{eqnarray}
    \label{70-thq}g_{2}(\lambda)=-f_{2}(\lambda)D(-\lambda).
\end{eqnarray}
Substituting $A(\lambda,\partial)$ and $g_{2}(\lambda)$ into Eq.~(\ref{66-thq}), we obtain
\begin{eqnarray}
\label{71-thq}f_{2}(\lambda)\bigl(B(\lambda+\mu,\partial)+f_{2}(\lambda+\mu)D(-\lambda-\mu)D(\partial)
+B(\mu,\partial)+f_{2}(\mu)D(-\mu)D(\partial) \bigr)=0.
\end{eqnarray}
Cancelling the nonzero factor $f_{2}(\lambda)$ gives
\begin{eqnarray}
\label{72-thq}B(\lambda+\mu,\partial)+f_{2}(\lambda+\mu)D(-\lambda-\mu)D(\partial)+B(\mu,\partial)+f_{2}(\mu)D(-\mu)D(\partial)=0.
\end{eqnarray}
Define $H(\lambda,\partial)=B(\lambda,\partial)+f_{2}(\lambda)D(-\lambda)D(\partial)$. Then the above identity becomes
$H(\lambda+\mu,\partial)+H(\mu,\partial)=0$.
Setting $\lambda=0$ yields $2H(\mu,\partial)=0$. Hence $H(\lambda,\partial)=0$ . Consequently,
\begin{eqnarray}
    \label{73-thq}B(\lambda,\partial)=-f_{2}(\lambda)D(-\lambda)D(\partial).
\end{eqnarray}
Finally, a direct verification shows that the obtained expressions for $A(\lambda,\partial)$, $g_{2}(\lambda)$ and $B(\lambda,\partial)$ satisfy Eqs.~(\ref{68-thq})--(\ref{69-thq}). Then the conclusion holds.
\end{proof}

Now, we have completed the discussion of Case (A) in Lemma \ref{lem-l2}.
Next, we turn to Case (B) of Lemma \ref{lem-l2}. In Case (B) of Lemma \ref{lem-l2}, Eqs.~(\ref{33-thq})--(\ref{44-thq}) can be simplified into
\begin{eqnarray}
    \label{75-thq}&&g_{1}(\mu)A(\lambda+\mu,\partial)-g_{1}(\mu)A(\lambda,\mu+\partial)=A(\lambda,-\lambda-\mu)A(\lambda+\mu,\partial)-bA(\lambda,\partial),\\
    \label{76-thq}&&A(\lambda+\mu,\partial)B(\lambda,-\lambda-\mu)-g_{1}(\lambda)B(\mu,\lambda+\partial)-bB(\lambda,\partial)\notag\\
    &&\quad=A(\lambda+\mu,\partial)B(\mu,-\lambda-\mu)-g_{1}(\mu)B(\lambda,\mu+\partial)-bB(\mu,\partial).
\end{eqnarray}
\begin{lem}\label{lem-l4}For Eq.~(\ref{75-thq}), when $g_{1}(\lambda)\neq0$, we obtain that $A(\lambda,\partial)=0$ if $b=0$, and
$A(\lambda,\partial)=0$ or $A(\lambda,\partial)=b$ if $b\neq 0$.
\end{lem}
\begin{proof}
When $b=0$, Eq.~(\ref{75-thq}) reduces to
\begin{eqnarray}
\label{78-thq}&&g_{1}(\mu)A(\lambda+\mu,\partial)-g_{1}(\mu)A(\lambda,\mu+\partial)
= A(\lambda,-\lambda-\mu)A(\lambda+\mu,\partial).
\end{eqnarray}
If $A(\lambda,\partial)=0$, there is nothing to prove. Assume $A(\lambda,\partial)\neq0$ and write $A(\lambda,\partial)=\sum_{i=0}^{m}A_i(\lambda)\partial^{\,i}$ with $A_m(\lambda)\neq0$.   We claim that $m=0$.
On the contrary, suppose that $m\ge1$.
Comparing the coefficients of $\partial^{m}$ in Eq.~(\ref{78-thq}) gives
\delete{\[
A(\lambda,-\lambda-\mu)A_m(\lambda+\mu)=g_{1}(\mu)\bigl[A_m(\lambda+\mu)-A_m(\lambda)\bigr],
\]
which can be rearranged as}
\begin{equation}\label{80-thq}
A_m(\lambda+\mu)\bigl(A(\lambda,-\lambda-\mu)-g_{1}(\mu)\bigr)=-A_m(\lambda)g_{1}(\mu).
\end{equation}
Let $d=\deg A_m(\lambda) \ge0$ and write $A_m(\lambda)=\sum_{i=0}^{d}a_i\lambda^{i}$ with $a_d\neq0$.  Assume $d\ge1$. Set $p=\deg g_{1}(\lambda)\ge0$ and $g_{1}(\mu)=\sum_{j=0}^{p}l_j\mu^{j}$ with $l_p\neq0$.  View Eq.~(\ref{80-thq}) as a polynomial identity in $\mu$.  The right-hand side has degree $p$ in $\mu$.  On the left-hand side, if $m=p$, then the coefficient of $\mu^{m}$ in $A(\lambda,-\lambda-\mu)-g_{1}(\mu)$ is $(-1)^{m}A_m(\lambda)-l_m$, which is not equal to zero, since $d\ge1$. Therefore, we have $\deg_{\mu}A_m(\lambda+\mu)=d$, and
$\deg_{\mu}\bigl(A(\lambda,-\lambda-\mu)-g_{1}(\mu)\bigr)=\max\{m,p\}$.
  Hence the left-hand side has degree $d+\max\{m,p\}\ge1+p>p$, a contradiction.  Therefore, $d=0$, i.e. $A_m(\lambda)$ is a nonzero constant. Denote it by $A_m\in \mathbb{C}\setminus\{0\}$.
Substituting $A_m(\lambda)=A_m$ into Eq.~(\ref{80-thq}) yields $A_m\,A(\lambda,-\lambda-\mu)=0$. So $A(\lambda,-\lambda-\mu)=0$.  Therefore, we have $A(\lambda, \partial)=0$, a contradiction. Thus, we get $m=0$. Then setting $A(\lambda, \partial)=A_0(\lambda)\in \mathbb{C}[\lambda]$ and taking it into Eq. (\ref{78-thq}), we have
\begin{eqnarray}\label{eeqq1}
g_1(\mu)(A_0(\lambda+\mu)-A_0(\lambda))=A_0(\lambda)A_0(\lambda+\mu).
\end{eqnarray}
Setting $\mu=0$ in Eq.~(\ref{eeqq1}), we have $A_0^2(\lambda)=0$. Therefore, we obtain $A(\lambda, \partial)=A_0(\lambda)=0$.
\delete{By comparing the degree of $\mu$ in Eq. (\ref{eeqq1}), we have $g_1(\mu)\in \mathbb{C}\backslash\{0\}$. Then by comparing the degree of $\mu$ in Eq. (\ref{eeqq1}), we have $A(\lambda)\in \mathbb{C}$. Taking it into Eq. (\ref{eeqq1}), we have $A(\lambda, \partial)=A_0(\lambda)=0$.}
\delete{Expanding gives
\[
\sum_{i=0}^{m}A_i(\lambda)(-\lambda-\mu)^{i}=0.
\]
Since $m\ge1$, the coefficient of $\mu^{m}$ is $(-1)^{m}A_m=0$, contradicting $A_m\neq0$.  Thus $m\ge1$ is impossible, so $m=0$.
Consequently $A(\lambda,\partial)=A_0(\lambda)$.  Substituting this into Eq.~(\ref{78-thq}) and setting $\mu=0$ gives $A_0(\lambda)^{2}=0$, i.e.\ $A(\lambda,\partial)=A_0(\lambda)=0$, which contradicts the assumption $A(\lambda,\partial)\neq0$.  Hence when $b=0$ we must have $A(\lambda,\partial)=0$.}

Next, we discuss the case $b \neq 0$. If $A(\lambda,\partial)=0$, then the statement holds.
Assume $A(\lambda,\partial)\neq0$ and we shall prove $A(\lambda,\partial)=b$.
Write $A(\lambda,\partial)=\sum_{i=0}^{m}A_i(\lambda)\partial^i$ with $A_m(\lambda)\neq0$. We claim that $m=0$.
Suppose $m\ge1$.  Comparing the coefficients of $\partial^{m}$  in Eq.~(\ref{75-thq}) gives
\begin{equation}\label{81-thq}
A_m(\lambda+\mu)\bigl(A(\lambda,-\lambda-\mu)-g_{1}(\mu)\bigr)=A_m(\lambda)\bigl(b-g_{1}(\mu)\bigr).
\end{equation}

When $g_1(\lambda)\neq b$, let $d=\deg A_m(\lambda)\ge0$.  If $d\ge1$, the same argument on discussing the degree of $\mu$ as the case when $b=0$ yields a contradiction.  Hence $d=0$. Set $A_m(\lambda)=A_m\in \mathbb{C}\setminus \{0\}$.  Substituting into Eq.~(\ref{81-thq}) yields $A_m\bigl(A(\lambda,-\lambda-\mu)-b\bigr)=0$. Then we have $A(\lambda,-\lambda-\mu)=b$. Therefore, we obtain $A(\lambda,\partial)=b$, which contradicts with $m\geq 1$. Thus, we have $m=0$. Therefore, $A(\lambda,\partial)=A_0(\lambda)$. Taking it into Eq. (\ref{75-thq}), we have
\begin{eqnarray}\label{eeqq2}
g_1(\mu)(A_0(\lambda+\mu)-A_0(\lambda))=A_0(\lambda)(A_0(\lambda+\mu)-b).
\end{eqnarray}
Setting $\mu=0$ in Eq. (\ref{eeqq2}), we have $A_0(\lambda)(A_0(\lambda)-b)=0$. Since $A_0(\lambda)\neq 0$, we have $A(\lambda,\partial)=A_0(\lambda)=b$.

When $g_1(\lambda)=b$, Eq.~(\ref{81-thq}) becomes $A_m(\lambda+\mu)\bigl(A(\lambda,-\lambda-\mu)-b\bigr)=0$.  Since $A_m(\lambda)\neq0$ , we obtain $A(\lambda,-\lambda-\mu)=b$, i.e. $A(\lambda,\partial)=b$, a contradiction. Therefore, we have $m=0$. Then we get $A(\lambda,\partial)=A_0(\lambda)$. Taking it into Eq. (\ref{75-thq}), we have
\begin{eqnarray}
g_1(\mu)A_0(\lambda+\mu)=A_0(\lambda)A_0(\lambda+\mu).
\end{eqnarray}
Since $A_0(\lambda+\mu)\neq 0$, we have $A(\lambda,\partial)=A_0(\lambda)=g_1(\mu)=b$.

The proof is completed.

\delete{Now from Eq.~(\ref{82-thq}) we have $\sum_{i=0}^{m}A_i(\lambda)(-\lambda-\mu)^{i}=b$.  If $m\ge1$, the coefficient of $\mu^{m}$ is $(-1)^{m}A_m(\lambda)=0$, contradicting $A_m(\lambda)\neq0$.  Hence $m\ge1$ is impossible, so $m=0$ and $A(\lambda,\partial)=A_0(\lambda)\neq 0$.
Now substitute $A(\lambda,\partial)=A_0(\lambda)$ into Eq.~(\ref{75-thq}):
\begin{equation}\label{83-thq}
A_0(\lambda+\mu)\bigl[g_{1}(\mu)-A_0(\lambda)\bigr]=A_0(\lambda)\bigl[g_{1}(\mu)-b\bigr].
\end{equation}

If $g_1(\lambda)\neq b$.
Choose $\mu_{0}$ such that $g_{1}(\mu_{0})\neq b$.  Setting $\mu=\mu_{0}$ in Eq.~(\ref{83-thq}) gives
\begin{eqnarray}
A_0(\lambda+\mu_{0})\bigl[g_{1}(\mu_{0})-A_0(\lambda)\bigr]=A_0(\lambda)\bigl[g_{1}(\mu_{0})-b\bigr].
\end{eqnarray}
Let $r=\deg A_0(\lambda)\ge0$.  If $r\ge1$, the left-hand side has degree $2r$ while the right-hand side has degree $r$, contradiction.  Hence $r=0$, so $A_0(\lambda)$ is a nonzero constant, say $A_0$.  The equation simplifies to $A_0\bigl[g_{1}(\mu_{0})-A_0\bigr]=A_0\bigl[g_{1}(\mu_{0})-b\bigr]$.  Cancelling $A_0 \neq0$ gives $A_0=b$.  Thus $A(\lambda,\partial)=A_0(\lambda)=A_0= b$.

If $g_1(\lambda)=b$, then Eq.~(\ref{83-thq}) becomes $A_0(\lambda+\mu)\bigl[b-A_0(\lambda)\bigr]=0$.  Since $A_0(\lambda+\mu)\neq0$, we  have $A_0(\lambda)= b$.
Hence  $A(\lambda,\partial)\neq0$ implies $A(\lambda,\partial)= b$.  This completes the proof for $b\neq0$.}
\end{proof}
Next, we discuss $g_{1}(\lambda)$ and $B(\lambda,\partial)$ based on Lemma \ref{lem-l4} and Eq.~(\ref{76-thq}).
\begin{pro}\label{lem-l6}
  In Case (B) of Lemma \ref{lem-l2}, assume $A(\lambda,\partial)=b=0$ and $g_{1}(\lambda)\neq0$.
Then there exists a  polynomial $F(\partial)\in\mathbb{C}[\partial]$ such that
$B(\lambda,\partial)=g_{1}(\lambda)\,F(\lambda+\partial)$.
\end{pro}
\begin{proof}
When $A(\lambda,\partial)=b=0$, Eq.~(\ref{76-thq}) reduces to
\begin{equation}\label{84-thq}
g_{1}(\lambda)B(\mu,\lambda+\partial)=g_{1}(\mu)B(\lambda,\mu+\partial).
\end{equation}
Since $g_{1}(\lambda)\neq0$, by Eq. (\ref{84-thq}), $g_{1}(\lambda)$ divides $B(\lambda,\mu+\partial)$. Therefore, we have
\begin{eqnarray}\label{eq-q5}
B(\mu,\lambda+\partial)=g_{1}(\mu)\frac{B(\lambda,\mu+\partial)}{g_{1}(\lambda)}.
\end{eqnarray}
Set $C(\lambda,\mu+\partial)=\frac{B(\lambda,\mu+\partial)}{g_{1}(\lambda)}$. Setting $\lambda=0$ in Eq. (\ref{eq-q5}), we obtain
$B(\mu,\partial)=g_1(\mu)C(0,\mu+\partial)$. Let $F(\mu+\partial)=C(0,\mu+\partial)$. Then we have $B(\lambda,\partial)=g_{1}(\lambda)\,F(\lambda+\partial)$. Note that such $B(\lambda,\partial)$ for any $F(\partial)\in\mathbb{C}[\partial]$ satisfies Eq. (\ref{84-thq}).
Then the conclusion holds.
\delete{If $B(\lambda,\partial)=0$, we simply take $F(x)=0$.  Now assume $B(\lambda,\partial)\neq0$ and write
\[
B(\lambda,\partial)=\sum_{i=0}^{p}B_i(\lambda)\partial^{\,i},\qquad
B_p(\lambda)\neq0\;(p\ge0).
\]
\noindent\textbf{Step 1.  The leading coefficient.}

When $p\ge1$, comparing the coefficients of $\partial^{p}$ in Eq.~(\ref{84-thq}) gives
\begin{equation}\label{85-thq}
g_{1}(\mu)B_{p}(\lambda)=g_{1}(\lambda)B_{p}(\mu).
\end{equation}
By Lemma 4.2 of \cite{Luo2026}, applied to the Eq.~(\ref{85-thq})
there exists a nonzero scalar \(\alpha\in \mathbb{C}\setminus\{0\}\) such that
\begin{eqnarray}
B_{p}(\lambda)=\alpha\,g_{1}(\lambda).
\end{eqnarray}

\delete{Let us write $g_{1}(\lambda)=\sum_{i=0}^{s}l_i\lambda^{i}$ with $l_s\neq0$,
and $B_{p}(\lambda)=\sum_{j=0}^{q}b_j\lambda^{j}$ with $b_q\neq0$.
Substituting these expansions into Eq.~(\ref{85-thq}) and comparing the coefficients of
$\lambda^{i}\mu^{j}$ yields
\begin{equation}\label{86-thq}
l_{i}b_{j}=l_{j}b_{i}\qquad(\forall i,j).
\end{equation}
By the same argument as in Proposition 3.4, we have $s=q$.
Set $\alpha:=b_{s}/l_{s}\neq0$.  Then $b_{i}=\alpha l_{i}$ for all $i$.
Therefore
\[
B_{p}(\lambda)=\sum_{j=0}^{s}b_j\lambda^{j}
=\sum_{j=0}^{s}\alpha l_j\lambda^{j}
=\alpha\sum_{j=0}^{s}l_j\lambda^{j}=\alpha\,g_{1}(\lambda).
\]
Moreover, $\alpha$ does not depend on the particular nonzero coefficient of $g_1(\lambda)$:
if $l_k\neq0$ for some $k$, then $l_k b_s = l_s b_k$ by Eq.~(\ref{86-thq}), whence
$b_k/l_k = b_s/l_s = \alpha$.  Thus
\begin{eqnarray*}
    B_{p}(\lambda)=\alpha\,g_{1}(\lambda),\qquad \alpha\in \mathbb{C}\setminus\{0\}.
\end{eqnarray*}}
\noindent\textbf{Step 2.  Induction on $p=\deg_{\partial}B(\lambda,\partial)$.}

We prove that any nonzero solution $B(\lambda,\partial)$ of Eq.~(\ref{84-thq}) can be written as
$B(\lambda,\partial)=g_{1}(\lambda)F(\lambda+\partial)$ for some $F(x)\in\mathbb{C}[x]$.

Base case $p=0$.
Then $B(\lambda,\partial)=B_0(\lambda)$ and Eq.~(\ref{84-thq}) becomes
$g_{1}(\lambda)B_0(\mu)=g_{1}(\mu)B_0(\lambda)$.  Exactly the same coefficient
comparison as in Step~1 shows that $B_0(\lambda)=\beta g_{1}(\lambda)$ for a
constant $\beta\in \mathbb{C}\setminus\{0\}$.  Hence $B(\lambda,\partial)=g_{1}(\lambda)\cdot\beta$,
which is of the form $g_{1}(\lambda)F(\lambda+\partial)$ with $F(x)=\beta$.

Inductive step.
Assume the statement holds for all nonzero solutions with $\deg_{\partial}B(\lambda,\partial)<p$
($p\ge1$).  Let $B(\lambda,\partial)$ be a nonzero solution with $\deg_{\partial}B(\lambda,\partial)=p$.
From Step~1 we have $B_{p}(\lambda)=\alpha\,g_{1}(\lambda)$, $\alpha\neq0$.
Define an auxiliary polynomial
\[
B^{(p)}(\lambda,\partial):=\alpha\,g_{1}(\lambda)\,(\lambda+\partial)^{p}.
\]
One checks directly that $B^{(p)}(\lambda,\partial)$ satisfies Eq.~(\ref{84-thq}).
Moreover $\deg_{\partial}B^{(p)}(\lambda,\partial)=p$ and the coefficient of $\partial^{p}$ in
$B^{(p)}(\lambda,\partial)$ is $\alpha\,g_{1}(\lambda)=B_{p}(\lambda)$.
Set
\[
\widetilde{B}(\lambda,\partial):=B(\lambda,\partial)-B^{(p)}(\lambda,\partial).
\]
Then $\widetilde{B}(\lambda,\partial)$ satisfies Eq.~(\ref{84-thq}) (by linearity) and
$\deg_{\partial}\widetilde{B}(\lambda,\partial)\le p-1$.
If $\widetilde{B}=0$, then $B(\lambda,\partial)=B^{(p)}(\lambda,\partial)$ already has the desired form with
$F(x)=\alpha x^{p}$.
If $\widetilde{B}(\lambda,\partial)\neq0$, the induction hypothesis applies, so there exists
a polynomial $F_{1}(x)\in\mathbb{C}[x]$ such that
$\widetilde{B}(\lambda,\partial)=g_{1}(\lambda)F_{1}(\lambda+\partial)$.
Consequently,
\[
\begin{aligned}
B(\lambda,\partial) &= \alpha\,g_{1}(\lambda)\,(\lambda+\partial)^{p}
+g_{1}(\lambda)F_{1}(\lambda+\partial)\\
&= g_{1}(\lambda)\bigl[\alpha(\lambda+\partial)^{p}+F_{1}(\lambda+\partial)\bigr].
\end{aligned}
\]
Defining $F(x):=\alpha x^{p}+F_{1}(x)\in\mathbb{C}[x]$ completes the induction.
Hence, every polynomial $B(\lambda,\partial)$ satisfying Eq.~(\ref{84-thq}) can be written as
$B(\lambda,\partial)=g_{1}(\lambda)F(\lambda+\partial)$ for some $F(x)\in\mathbb{C}[x]$
(where $F(x)=0$ precisely when $B(\lambda,\partial)=0$).  This proves the lemma.}
\end{proof}
\begin{pro}\label{pro-p3}
   In Case (B) of Lemma \ref{lem-l2}, assume $A(\lambda,\partial)=0$, $b\neq0$, and $g_{1}(\lambda)\neq0$.
Then we have the following cases:
\begin{enumerate}
    \item[(B1)] If $g_{1}(\lambda)\neq b$, then
    $B(\lambda,\partial)=g_{1}(\lambda)F(\lambda+\partial)-bF(\partial)$,
    where $F(\partial)$ is an arbitrary polynomial in $\mathbb{C}[\partial]$.
    \item[(B2)] If $g_{1}(\lambda)= b$, then
    $B(\lambda,\partial)=U(\lambda)+W(\lambda+\partial)-W(\partial)$,
    where $U(\lambda)$ is an arbitrary polynomial in  $\mathbb{C}[\lambda]$ and  $W(\partial)$ is an arbitrary polynomial in $\mathbb{C}[\partial]$.
\end{enumerate}
\end{pro}
\begin{proof}
This conclusion is a special case of \cite[Lemma 4.7]{Luo2026} with $a(\lambda)=g_1(\lambda),b(\lambda)=b$.
\delete{When \(g_1(\lambda)\neq b\), it corresponds to the case \(a(\lambda)\neq b(\lambda)\) in Lemma 4.7, yielding (B1). When \(g_1(\lambda)=b\), it corresponds to the case \(a(\lambda)=b(\lambda)\neq0\) in Lemma 4.7, yielding (B2).}
\end{proof}
\delete{\begin{proof}
    Since $A(\lambda,\partial)=0$ and $b\neq0$, Eq.~(\ref{76-thq}) simplifies to
\begin{equation}\label{87-thq}
g_{1}(\mu)B(\lambda,\mu+\partial)-g_{1}(\lambda)B(\mu,\lambda+\partial)=b\bigl[B(\lambda,\partial)-B(\mu,\partial)\bigr].
\end{equation}
If $B(\lambda,\partial)=0$, both statements hold trivially (take $F=0$ in (B1) or $U=W=0$ in (B2)).
We now assume $B(\lambda,\partial)\neq0$ and write
\[
B(\lambda,\partial)=\sum_{i=0}^{p}B_i(\lambda)\partial^{\,i},\qquad B_p(\lambda)\neq0\;(p\ge0).
\]
When $p\ge1$,comparing the coefficients of $\partial^{p}$ in Eq.~(\ref{87-thq}) gives
\begin{equation}\label{88-thq}
B_{p}(\lambda)\bigl[g_{1}(\mu)-b\bigr]=B_{p}(\mu)\bigl[g_{1}(\lambda)-b\bigr].
\end{equation}
We distinguish two subcases.

\medskip
\noindent\textbf{Case (i): $g_{1}(\lambda)\neq b$.}
When $p\ge1$, by the same argument as in Lemma 3.7 Step 1, there exists a nonzero constant $\alpha$ such that
\begin{equation}\label{89-thq}
B_{p}(\lambda)=\alpha\bigl[g_{1}(\lambda)-b\bigr].
\end{equation}
We now prove that every nonzero solution $B(\lambda,\partial)$ of Eq.~(\ref{87-thq}) admits the form
$B(\lambda,\partial)=g_{1}(\lambda)F(\lambda+\partial)-bF(\partial)$ for some polynomial $F$
when $g_1(\lambda)\neq b$.  The proof proceeds by induction on $p=\deg_{\partial}B(\lambda,\partial)$.

Base case $p=0$. Then $B(\lambda,\partial)=B_0(\lambda)$ and Eq.~(\ref{87-thq}) becomes
\begin{equation}\label{90-thq}
B_0(\lambda)\bigl[g_{1}(\mu)-b\bigr]=B_0(\mu)\bigl[g_{1}(\lambda)-b\bigr].
\end{equation}
Using the same coefficient comparison as above we obtain $B_0(\lambda)=\beta[g_{1}(\lambda)-b]$ for a nonzero constant $\beta$.  Hence
$B(\lambda,\partial)=g_{1}(\lambda)\beta-b\beta$, which is exactly the desired form with $F(x)=\beta$.

Inductive step.
Assume the claim holds for all nonzero solutions with $\deg_{\partial}B(\lambda,\partial)<p$ ($p\ge1$).  Let $B(\lambda,\partial)$ be a solution with $\deg_{\partial}B(\lambda,\partial)=p$ and $B_p(\lambda)\neq0$.  From Eq.~(\ref{89-thq}) we have
$B_p(\lambda)=\alpha[g_{1}(\lambda)-b]$, $\alpha\neq0$.  Define
\[
B^{(p)}(\lambda,\partial):=\alpha g_{1}(\lambda)(\lambda+\partial)^{p}-\alpha b\,\partial^{\,p}.
\]
One readily verifies that $B^{(p)}(\lambda,\partial)$ satisfies Eq.~(\ref{87-thq}).  Moreover $\deg_{\partial}B^{(p)}(\lambda,\partial)=p$ and the coefficient of $\partial^{p}$ in $B^{(p)}(\lambda,\partial)$ is $\alpha[g_{1}(\lambda)-b]=B_{p}(\lambda)$.  Set
\[
\widetilde{B}(\lambda,\partial):=B(\lambda,\partial)-B^{(p)}(\lambda,\partial).
\]
Then $\widetilde{B}(\lambda,\partial)$ satisfies Eq.~(\ref{87-thq}) (by linearity) and $\deg_{\partial}\widetilde{B}(\lambda,\partial)\le p-1$.
If $\widetilde{B}(\lambda,\partial)=0$, then $B(\lambda,\partial)=B^{(p)}(\lambda,\partial)$ already has the required form with $F(x)=\alpha x^{p}$.
If $\widetilde{B}(\lambda,\partial)\neq0$, the induction hypothesis yields a polynomial $F_{1}(x)$ such that
$\widetilde{B}(\lambda,\partial)=g_{1}(\lambda)F_{1}(\lambda+\partial)-bF_{1}(\partial)$.  Consequently,
\[
\begin{aligned}
B(\lambda,\partial) &= \alpha g_{1}(\lambda)(\lambda+\partial)^{p}-\alpha b\,\partial^{p}
                     +g_{1}(\lambda)F_{1}(\lambda+\partial)-bF_{1}(\partial)\\
&= g_{1}(\lambda)\bigl[\alpha(\lambda+\partial)^{p}+F_{1}(\lambda+\partial)\bigr]
   -b\bigl[\alpha\partial^{p}+F_{1}(\partial)\bigr].
\end{aligned}
\]
Defining $F(x):=\alpha x^{p}+F_{1}(x)$ gives the desired representation.  This completes the induction, establishing (B1).

\medskip
\noindent\textbf{Case (ii): $g_{1}(\lambda)= b$.}
Then Eq.~(\ref{87-thq}) reduces to
\begin{equation}\label{91-thq}
B(\lambda,\mu+\partial)-B(\mu,\lambda+\partial)=B(\lambda,\partial)-B(\mu,\partial).
\end{equation}
When $p\ge 1$, comparing the coefficients of $\partial^{p}$ gives the identity $B_{p}(\lambda)-B_{p}(\mu)=B_{p}(\lambda)-B_{p}(\mu),$
which imposes no restriction.
When $p=1$, we have $B(\lambda,\partial)=B_0(\lambda)+B_1(\lambda)\partial$.
Substituting this into Eq.~(\ref{91-thq}) yields $B_1(\lambda)\mu = B_1(\mu)\lambda$.
When $p\ge2$,then comparing the coefficients of $\partial^{p-1}$ yields
\begin{eqnarray}\label{92-thq}
    \lambda B_{p}(\mu)=\mu B_{p}(\lambda)\qquad(\forall\,\lambda,\mu).
\end{eqnarray}
Consequently, for every $p\ge1$, we have $\lambda B_{p}(\mu)=\mu B_{p}(\lambda)$.
By the same method as in Lemma~3.7 Step~1, we obtain a nonzero constant $k$ such that $B_{p}(\lambda)=k\lambda$.
Now we prove by induction on $p$ that any solution of Eq.~(\ref{91-thq}) when $g_1(\lambda)=b$ can be written as
$B(\lambda,\partial)=U(\lambda)+W(\lambda+\partial)-W(\partial)$ for some polynomials $U,W$.

Base case $p=0$.
Then $B(\lambda,\partial)=B_0(\lambda)$.
Choosing $U(\lambda)=B_0(\lambda)$ and $W=0$ gives the required form.

Inductive step.
Suppose the statement holds for all solutions with $\deg_{\partial}B(\lambda,\partial)<p$ ($p\ge1$), and let $B(\lambda,\partial)$ be a solution with $\deg_{\partial}B(\lambda,\partial)=p$, $B_p(\lambda)=k\lambda$ ($k\neq0$).  Define
\[
B^{(p)}(\lambda,\partial):=\frac{k}{p+1}\bigl[(\lambda+\partial)^{p+1}-\partial^{\,p+1}\bigr].
\]
A direct computation shows that $B^{(p)}(\lambda,\partial)$ satisfies Eq.~(\ref{91-thq}).
Its $\partial^{p}$-coefficient is $\frac{k}{p+1}\cdot\binom{p+1}{p}\lambda=k\lambda=B_{p}(\lambda)$, and $\deg_{\partial}B^{(p)}(\lambda,\partial)=p$.
Set $\widetilde{B}(\lambda,\partial):=B(\lambda,\partial)-B^{(p)}(\lambda,\partial)$.  Then $\widetilde{B}(\lambda,\partial)$ satisfies Eq.~(\ref{91-thq}) and $\deg_{\partial}\widetilde{B}(\lambda,\partial)\le p-1$.  If $\widetilde{B}(\lambda,\partial)=0$, then $B(\lambda,\partial)=B^{(p)}(\lambda,\partial)$ already has the required form with $U=0, W(x)=\frac{k}{p+1}x^{p+1}$.
If $\widetilde{B}(\lambda,\partial)\neq0$, by the induction hypothesis, there exist polynomials $U_{1},W_{1}$ such that
\[
\widetilde{B}(\lambda,\partial)=U_{1}(\lambda)+W_{1}(\lambda+\partial)-W_{1}(\partial).
\]
Therefore
\[
\begin{aligned}
B(\lambda,\partial) &= \frac{k}{p+1}\bigl[(\lambda+\partial)^{p+1}-\partial^{\,p+1}\bigr]
                     +U_{1}(\lambda)+W_{1}(\lambda+\partial)-W_{1}(\partial)\\
&= U_{1}(\lambda)+\Bigl[\frac{k}{p+1}(\lambda+\partial)^{p+1}+W_{1}(\lambda+\partial)\Bigr]
                -\Bigl[\frac{k}{p+1}\partial^{\,p+1}+W_{1}(\partial)\Bigr].
\end{aligned}
\]
Taking $U:=U_{1}$ and $W(x):=\frac{k}{p+1}x^{p+1}+W_{1}(x)$ gives the desired form.
The induction is complete, proving (B2).
\end{proof}}
\begin{pro}\label{pro-p4}
   In Case (B) of Lemma \ref{lem-l2},  assume $A(\lambda,\partial)=b\neq0$ and $g_{1}(\lambda)\neq0$.
Then we have the following cases:
\begin{enumerate}
    \item[(B3)] If $g_{1}(\lambda)\neq b$, then
    $B(\lambda,\partial)=g_{1}(\lambda)F(\lambda+\partial)-bF(\partial)+bF(-\lambda)$,
    where $F(\partial)$ is an arbitrary polynomial in $\mathbb{C}[\partial]$.
    \item[(B4)] If $g_{1}(\lambda)= b$, then
    $B(\lambda,\partial)=F(\lambda+\partial)-F(\partial)+F(-\lambda)+H(\lambda)(\lambda+\partial)$,
    where $H(\lambda)$ and $F(\lambda)$ are arbitrary polynomials in  $\mathbb{C}[\lambda]$.
\end{enumerate}
\end{pro}
\begin{proof}
    When $A(\lambda,\partial)=b\neq0$, Eq.~(\ref{76-thq}) reduces to
\begin{equation}\label{93-thq}
b\bigl(B(\lambda,-\lambda-\mu)-B(\mu,-\lambda-\mu)\bigr)+g_{1}(\mu)B(\lambda,\mu+\partial)-g_{1}(\lambda)B(\mu,\lambda+\partial)=b\bigl(B(\lambda,\partial)-B(\mu,\partial)\bigr).
\end{equation}
If $B(\lambda,\partial)=0$, then the statement holds with $F(\partial)=0$ in (B3) and  $F(\lambda)=H(\lambda)=0$ in (B4).  We now assume $B(\lambda,\partial)\neq0$ and write
$B(\lambda,\partial)=\sum_{i=0}^{p}B_i(\lambda)\partial^{\,i}$ with $B_p(\lambda)\neq0$.
When $p\ge1$, comparing the coefficients of $\partial^{p}$ in Eq.~(\ref{93-thq}) yields
\begin{equation}\label{94-thq}
B_{p}(\lambda)\bigl(g_{1}(\mu)-b\bigr)=B_{p}(\mu)\bigl(g_{1}(\lambda)-b\bigr).
\end{equation}
We discuss it in the following two subcases.

\medskip
\noindent\textbf{Case (i): $g_{1}(\lambda)\neq b$.}
When $p\ge1$,  by the same argument as in Proposition \ref{lem-l6}, there exists a nonzero constant $\alpha$ such that
\begin{equation}\label{95-thq}
B_{p}(\lambda)=\alpha\bigl(g_{1}(\lambda)-b\bigr).
\end{equation}
We prove by induction on $p$ that any nonzero solution $B(\lambda,\partial)$ of Eq.~(\ref{93-thq}) has the form
$B(\lambda,\partial)=g_{1}(\lambda)F(\lambda+\partial)-bF(\partial)+bF(-\lambda)$ for some polynomial $F(\partial)$ when $g_1(\lambda)\neq b$.

First, we consider the case when $p=0$.
Then $B(\lambda,\partial)=B_0(\lambda)$ and Eq.~(\ref{93-thq}) becomes
\begin{equation}\label{96-thq}
B_0(\lambda)g_{1}(\mu)=B_0(\mu)g_{1}(\lambda).
\end{equation}
Then we obtain $B_0(\lambda)=\beta g_{1}(\lambda)$ for a nonzero constant $\beta$.  Choosing $F(\partial)=\beta$ gives the required form.

Assume that the claim holds for all solutions with $\deg_{\partial}B(\lambda,\partial)<p$, where $p\ge1$.  Let $B(\lambda,\partial)$ be a solution of Eq.~(\ref{93-thq}) with $\deg_{\partial}B(\lambda,\partial)=p$.  From Eq.~(\ref{95-thq}), we have $B_p(\lambda)=\alpha(g_{1}(\lambda)-b)$, where $\alpha\in \mathbb{C}\setminus \{0\}$.  Define
\begin{eqnarray}
B^{(p)}(\lambda,\partial):=\alpha g_{1}(\lambda)(\lambda+\partial)^{p}-\alpha b\,\partial^{\,p}+b\alpha(-\lambda)^{p}.
\end{eqnarray}
A direct verification shows that $B^{(p)}(\lambda,\partial)$ satisfies Eq.~(\ref{93-thq}).  Moreover, $\deg_{\partial}B^{(p)}(\lambda,\partial)=p$ and the coefficient of $\partial^{p}$ in $B^{(p)}(\lambda,\partial)$ is $\alpha(g_{1}(\lambda)-b)=B_{p}(\lambda)$.  Set
\begin{eqnarray}
\widetilde{B}(\lambda,\partial):=B(\lambda,\partial)-B^{(p)}(\lambda,\partial).
\end{eqnarray}
Then $\widetilde{B}(\lambda,\partial)$ satisfies  Eq.~(\ref{93-thq}) and $\deg_{\partial}\widetilde{B}(\lambda,\partial)\le p-1$.
If $\widetilde{B}(\lambda,\partial)=0$, then $B(\lambda,\partial)=B^{(p)}(\lambda,\partial)$ already has the required form with $F(\partial)=\alpha \partial^{p}$.
If $\widetilde{B}(\lambda,\partial)\neq0$, the induction hypothesis yields a polynomial $F_{1}(\partial)$ such that
$\widetilde{B}(\lambda,\partial)=g_{1}(\lambda)F_{1}(\lambda+\partial)-bF_{1}(\partial)+bF_{1}(-\lambda)$.  Consequently, we have
\begin{eqnarray*}
\begin{aligned}
B(\lambda,\partial) &= \alpha g_{1}(\lambda)(\lambda+\partial)^{p}-\alpha b\,\partial^{p}+b\alpha(-\lambda)^{p}\\
&\quad +g_{1}(\lambda)F_{1}(\lambda+\partial)-bF_{1}(\partial)+bF_{1}(-\lambda)\\
&= g_{1}(\lambda)\bigl(\alpha(\lambda+\partial)^{p}+F_{1}(\lambda+\partial)\bigr)
   -b\bigl(\alpha\partial^{p}+F_{1}(\partial)\bigr)
   +b\bigl(\alpha(-\lambda)^{p}+F_{1}(-\lambda)\bigr)
   .
\end{aligned}
\end{eqnarray*}
Defining $F(\partial):=\alpha \partial^{p}+F_{1}(\partial)$  completes the induction, which proves (B3).

\medskip
\noindent\textbf{Case (ii): $g_{1}(\lambda)= b$.}
Then Eq.~(\ref{93-thq}) reduces to
\begin{equation}\label{97-thq}
B(\lambda,-\lambda-\mu)-B(\mu,-\lambda-\mu)+B(\lambda,\mu+\partial)-B(\mu,\lambda+\partial)=B(\lambda,\partial)-B(\mu,\partial).
\end{equation}
When $p\ge1$, comparing the coefficients of $\partial^{p}$ gives the identity $B_{p}(\lambda)-B_{p}(\mu)=B_{p}(\lambda)-B_{p}(\mu)$
which imposes no restriction.
When $p=1$, we have $B(\lambda,\partial)=B_0(\lambda)+B_1(\lambda)\partial$.
Substituting this into Eq.~(\ref{97-thq}) yields
$B(\lambda,\partial)=B_1(\lambda)(\lambda+\partial)+t$, where $t\in \mathbb{C}$.
If $p\ge2$, comparing the coefficients of $\partial^{p-1}$ yields $\lambda B_{p}(\mu)=\mu B_{p}(\lambda) $.
By the same method as in Proposition~\ref{lem-l6}, there exists a nonzero constant $k$ such that $B_{p}(\lambda)=k\lambda$.

We now prove by induction on $p$ that any solution of Eq.~(\ref{97-thq}) can be written as
$B(\lambda,\partial)=F(\lambda+\partial)-F(\partial)+F(-\lambda)+H(\lambda)(\lambda+\partial)$ for some polynomials $F(\lambda)$ and  $H(\lambda)$ when $g_1(\lambda)=b$.

When $p=0$, we have $B(\lambda,\partial)=B_0(\lambda)$. Substituting this into Eq.~(\ref{97-thq}) shows that $B_0(\lambda)$ is a nonzero constant, which we denote simply by $B_0$. Taking $F(\lambda)=B_0$ and $H(\lambda)=0$ yields the desired form.
For $p=1$, from the previous discussion we have
\begin{equation}\label{98-thq}
    B(\lambda,\partial)=B_1(\lambda)(\lambda+\partial)+t,\;\; t\in \mathbb{C}.
\end{equation} Taking $F(\lambda)=t$ and $H(\lambda)=B_1(\lambda)$ yields the desired form.

Assume that the claim holds for all solutions with $\deg_{\partial}B(\lambda,\partial)<p$, where $p\ge2$.  Let $B(\lambda,\partial)$ be a solution of Eq.~(\ref{97-thq}) with $\deg_{\partial}B(\lambda,\partial)=p$. Note that $B_p(\lambda)=k\lambda$, where $k\in \mathbb{C}\setminus \{0\}$.  Define
\begin{eqnarray}
B^{(p)}(\lambda,\partial):=\frac{k}{p+1}\bigl((\lambda+\partial)^{p+1}-\partial^{\,p+1}+(-\lambda)^{p+1}\bigr).
\end{eqnarray}
One checks that $B^{(p)}(\lambda,\partial)$ satisfies Eq.~(\ref{97-thq}).  The coefficient of $\partial^{p}$ in $B^{(p)}(\lambda,\partial)$ is $\frac{k}{p+1}\binom{p+1}{p}\lambda=k\lambda=B_{p}(\lambda)$, and $\deg_{\partial}B^{(p)}(\lambda,\partial)=p$.
Set $\widetilde{B}(\lambda,\partial):=B(\lambda,\partial)-B^{(p)}(\lambda,\partial)$.  Then $\widetilde{B}(\lambda,\partial)$ satisfies Eq.~(\ref{97-thq}) and $\deg_{\partial}\widetilde{B}(\lambda,\partial)\le p-1$. If $\widetilde{B}(\lambda,\partial)=0$, then $B(\lambda,\partial)=B^{(p)}(\lambda,\partial)$ already has the required form with $F(\lambda)= \frac{k}{p+1}\lambda^{p+1}$ and $H(\lambda)=0$.
If $\widetilde{B}(\lambda,\partial)\neq0$, by the induction hypothesis, there exist polynomials $F_{1}(\lambda)$ and $H_{1}(\lambda)$ such that
$\widetilde{B}(\lambda,\partial)=F_{1}(\lambda+\partial)-F_{1}(\partial)+F_{1}(-\lambda)+H_{1}(\lambda)(\lambda+\partial)$.
Therefore, we have
\begin{eqnarray*}
\begin{aligned}
B(\lambda,\partial) &= \frac{k}{p+1}\bigl((\lambda+\partial)^{p+1}-\partial^{\,p+1}+(-\lambda)^{p+1}\bigr) \\
&\quad +F_{1}(\lambda+\partial)-F_{1}(\partial)+F_{1}(-\lambda)+H_{1}(\lambda)(\lambda+\partial)\\
&= \Bigl(\frac{k}{p+1}(\lambda+\partial)^{p+1}+F_{1}(\lambda+\partial)\Bigr)
   -\Bigl(\frac{k}{p+1}\partial^{\,p+1}+F_{1}(\partial)\Bigr)\\
&\quad + \Bigl(\frac{k}{p+1}(-\lambda)^{p+1}+F_{1}(-\lambda)\Bigr) + H_{1}(\lambda)(\lambda+\partial).
\end{aligned}
\end{eqnarray*}
Defining $F(\lambda):=\frac{k}{p+1}\lambda^{p+1}+F_{1}(\lambda)$ and $H(\lambda):=H_{1}(\lambda)$ gives the required form.  This completes the induction, proving (B4).
\end{proof}

Finally, we turn to the case when $f_{1}(\lambda)=a\in \mathbb{C}\setminus\{0\}$.
\begin{pro}\label{pro-p5}
If $f_{1}(\lambda)=a\in \mathbb{C}\setminus\{0\}$, then we have the following cases:
\begin{enumerate}
    \item[(i)] If $f_{2}(\lambda)=0$, then
    $A(\lambda,\partial)=g_{1}(-\lambda-\partial)$,
    $g_{2}(\lambda)=d\in\mathbb{C}$,
    $B(\lambda,\partial)=\frac{1}{a}\bigl(g_{1}(\lambda)g_{1}(-\lambda-\partial)-dg_{1}(-\partial)\bigr)$.
    \item[(ii)] If $f_{2}(\lambda)\neq0$, then
    $A(\lambda,\partial)=g_{1}(-\lambda-\partial)-\frac{1}{a}f_{2}(\lambda)g_{1}(-\partial)$,
    $g_{2}(\lambda)=\frac{1}{a}f_{2}(\lambda)g_{1}(\lambda)$,
    $B(\lambda,\partial)=\frac{1}{a}g_{1}(\lambda)\bigl(g_{1}(-\lambda-\partial)-\frac{1}{a}f_{2}(\lambda)g_{1}(-\partial)\bigr)+c$,
    where $c\in\mathbb{C}$. Moreover, if $c \neq 0$, then $f_{2}(\lambda)=\dfrac{a}{2}$.
\end{enumerate}
\end{pro}
\begin{proof}
When $f_{1}(\lambda)=a$, Eqs.~(\ref{33-thq})--(\ref{44-thq}) can be simplified into
\begin{eqnarray}
    \label{99-thq}&&ag_{1}(\lambda)=aA(\mu,-\lambda-\mu)+f_{2}(\mu)g_1(\lambda+\mu),\\
    \label{102-thq}&&f_2(\lambda+\mu)g_{1}(\mu)=f_2(\lambda+\mu)A(\lambda,-\lambda-\mu)+f_{2}(\lambda)g_2(\lambda+\mu),\\
     \label{103-thq}&&aA(\mu,\lambda+\partial)+f_{2}(\mu)A(\lambda,\partial)=aA(\lambda,\mu+\partial)+f_{2}(\lambda)A(\mu,\partial),\\
    \label{104-thq}&&aB(\lambda,-\lambda-\mu)+g_1(\lambda+\mu)g_{2}(\lambda)=aB(\mu,-\lambda-\mu)+g_1(\lambda+\mu)g_{2}(\mu),\\
    \label{105-thq}&&f_2(\lambda+\mu)B(\lambda,-\lambda-\mu)+g_2(\lambda)g_{2}(\lambda+\mu)=f_2(\lambda+\mu)B(\mu,-\lambda-\mu)+g_2(\mu)g_2(\lambda+\mu),\\
    \label{107-thq}&&g_{1}(\mu)A(\lambda+\mu,\partial)-g_{1}(\mu)A(\lambda,\mu+\partial)-f_{2}(\lambda)B(\mu,\partial)\notag\\
    &&\quad=A(\lambda,-\lambda-\mu)A(\lambda+\mu,\partial)+f_{2}(\lambda)B(\lambda+\mu,\partial)-aB(\mu,\lambda+\partial)-g_{2}(\mu)A(\lambda,\partial),\\
    \label{108-thq}&&A(\lambda+\mu,\partial)B(\lambda,-\lambda-\mu)+g_{2}(\lambda)B(\lambda+\mu,\partial)-g_{1}(\lambda)B(\mu,\lambda+\partial)-g_{2}(\mu)B(\lambda,\partial)\notag\\
    &&\quad=A(\lambda+\mu,\partial)B(\mu,-\lambda-\mu)+g_{2}(\mu)B(\lambda+\mu,\partial)-g_{1}(\mu)B(\lambda,\mu+\partial)-g_{2}(\lambda)B(\mu,\partial).
\end{eqnarray}

    From Eq.~(\ref{99-thq}), we obtain
\begin{eqnarray}
A(\mu,-\lambda-\mu)=g_{1}(\lambda)-\frac{1}{a}f_{2}(\mu)g_{1}(\lambda+\mu),
\end{eqnarray}
which implies
\begin{equation}\label{109-thq}
A(\lambda,\partial)=g_{1}(-\lambda-\partial)-\frac{1}{a}f_{2}(\lambda)g_{1}(-\partial).
\end{equation}
Substituting Eq.~(\ref{109-thq}) into Eq.~(\ref{103-thq}) shows that Eq.~(\ref{103-thq}) holds automatically.
Inserting Eq.~(\ref{109-thq}) into  Eq.~(\ref{102-thq}) gives
\begin{align}
&f_{2}(\lambda)\Bigl(g_{2}(\lambda+\mu)-\frac{1}{a}f_{2}(\lambda+\mu)g_{1}(\lambda+\mu)\Bigr)=0.\label{111-thq}
\end{align}
Hence either $f_{2}(\lambda)=0$, or $f_{2}(\lambda)\neq0$ and then
$g_{2}(\lambda)=\dfrac{1}{a}f_{2}(\lambda)g_{1}(\lambda)$.  We discuss these two cases separately.

\medskip
\noindent\textbf{Case (i). $f_{2}(\lambda)=0$.}
Then Eq.~(\ref{109-thq}) reduces to $A(\lambda,\partial)=g_{1}(-\lambda-\partial)$.
Moreover, when $f_{2}(\lambda)=0$, Eq.~(\ref{105-thq}) implies that $g_{2}(\lambda)$ is a constant. Write
$g_{2}(\lambda)=d\in\mathbb{C}$.
Substituting $A(\lambda,\partial)$ and $g_{2}(\lambda)$ into Eq.~(\ref{107-thq}) yields
\begin{eqnarray}
B(\lambda,\partial)=\frac{1}{a}\bigl(g_{1}(\lambda)g_{1}(-\lambda-\partial)-dg_{1}(-\partial)\bigr).
\end{eqnarray}
Then one checks that Eqs.~(\ref{104-thq}) and (\ref{108-thq}) naturally hold.

\medskip
\noindent\textbf{Case (ii). $f_{2}(\lambda)\neq0$, $g_{2}(\lambda)=\dfrac{1}{a}f_{2}(\lambda)g_{1}(\lambda)$.}
Here $A(\lambda,\partial)$ is given by Eq.~(\ref{109-thq}).
Define a particular solution
\[
B_{0}(\lambda,\partial)=\frac{1}{a}g_{1}(\lambda)\Bigl(g_{1}(-\lambda-\partial)-\frac{1}{a}f_{2}(\lambda)g_{1}(-\partial)\Bigr).
\]
A direct verification shows that $B_{0}(\lambda,\partial)$ satisfies all equations.
Now let $B(\lambda,\partial)$ be any solution and set $\widetilde B(\lambda,\partial)=B(\lambda,\partial)-B_{0}(\lambda,\partial)$.
Then $\widetilde B(\lambda,\partial)$ also satisfies the equations.
Plugging $B(\lambda,\partial)=B_0(\lambda,\partial)+\widetilde B(\lambda,\partial)$ into  Eq.~(\ref{107-thq}), we obtain
\begin{align}
a\widetilde B(\mu,\lambda+\partial) &= f_{2}(\lambda)\bigl(\widetilde B(\lambda+\mu,\partial)+\widetilde B(\mu,\partial)\bigr).\label{113-thq}
\end{align}
From Eqs.~(\ref{104-thq})--(\ref{105-thq}), we also get
\begin{equation}\label{114-thq}
\widetilde B(\lambda,-\lambda-\mu)=\widetilde B(\mu,-\lambda-\mu).
\end{equation}
We claim that $\widetilde B(\lambda,\partial)$ satisfying Eqs.~(\ref{113-thq})--(\ref{114-thq}) must be a constant.
If $\widetilde B(\lambda,\partial)=0$, the claim is trivial.  Assume $\widetilde B(\lambda,\partial)\neq0$ and write
$\widetilde B(\lambda,\partial)=\sum_{i=0}^{p} \widetilde{B}_{i}(\lambda)\partial^{\,i}$ with  $\widetilde{B}_{p}(\lambda)\neq0$.
When $p=0$, $\widetilde B(\lambda,\partial)=\widetilde B_0(\lambda)$, and by Eq.~(\ref{114-thq}), $\widetilde B_0(\lambda)$ is a nonzero constant. When $p\ge 1$, comparing the coefficients of $\partial^{p}$ in Eq.~(\ref{113-thq}) gives
\begin{equation}\label{115-thq}
a \widetilde{B}_{p}(\mu)=f_{2}(\lambda)\bigl(\widetilde{B}_{p}(\lambda+\mu)+\widetilde{B}_{p}(\mu)\bigr).
\end{equation}
Write $\widetilde{B}_{p}(\lambda)=\sum_{i=0}^{q}b_{i}\lambda^{i}$ with $b_{q}\neq0$.
Regarding the right-hand side of Eq.~(\ref{115-thq}) as a polynomial in $\lambda$, the term
$\widetilde{B}_{p}(\lambda+\mu)$ contributes degree $q$ while $\widetilde{B}_{p}(\mu)$ is constant in $\lambda$.
If $q\ge1$, the left-hand side is independent of $\lambda$, whereas the right-hand side
contains $b_{q}\lambda^{q}$ which cannot cancel because
$f_{2}(\lambda)\neq0$ and $b_{q}\neq0$.  Hence $q=0$, i.e.\ $\widetilde{B}_{p}(\lambda)=B\in \mathbb{C}\setminus\{0\}$.
Substituting $\widetilde{B}_{p}(\lambda)=B$ into Eq.~(\ref{115-thq}) yields $aB=f_{2}(\lambda)(B+B)$. So
$f_{2}(\lambda)=\dfrac{a}{2}$.  Then Eq.~(\ref{113-thq}) becomes
\begin{equation}\label{116-thq}
2\,\widetilde B(\mu,\lambda+\partial)=\widetilde B(\lambda+\mu,\partial)+\widetilde B(\mu,\partial).
\end{equation}
Setting $\partial=0$ in Eq.~(\ref{116-thq}) gives
$2\,\widetilde B(\mu,\lambda)=\widetilde B(\lambda+\mu,0)+\widetilde B(\mu,0)$.
Define $\Phi(\lambda)=\widetilde B(\lambda,0)$ . Then
\begin{equation}\label{117-thq}
\widetilde B(\lambda,\mu)=\frac12\Phi(\lambda)+\frac12\Phi(\lambda+\mu).
\end{equation}
Note that $\Phi(\lambda)\neq0$ because
$\widetilde B(\lambda,\partial)\neq0$.
Inserting Eq.~(\ref{117-thq}) back into Eq.~(\ref{116-thq}), we obtain
\begin{eqnarray}
\Phi(\lambda+\mu)+\Phi(\mu+\partial)=\Phi(\lambda+\mu+\partial)+\Phi(\mu).
\end{eqnarray}
Taking $\mu=0$ gives
\begin{equation}\label{118-thq}
\Phi(\lambda)+\Phi(\partial)=\Phi(\lambda+\partial)+\Phi(0).
\end{equation}
Let $\Phi(\partial)=\sum_{i=0}^{r}\phi_{i}\partial^{i}$ with $\phi_{r}\neq0$.
If $r\ge2$, substituting into Eq.~(\ref{118-thq}) and comparing the coefficient of
$\lambda^{r-1}\partial$ yields $\phi_{r}\,r=0$, contradicting $\phi_{r}\neq0$.
Thus $r\le1$, i.e.\ $\Phi(\partial)=\phi_{0}+\phi_{1}\partial$.
Substituting it into Eq.~(\ref{117-thq}), we get
\begin{eqnarray}
\widetilde B(\lambda,\partial)=\frac{1}{2}\phi_{1}\partial+\phi_{1}\lambda+\phi_{0}.
\end{eqnarray}
Taking it into Eq.~(\ref{114-thq}), $\phi_{1}$ must vanish, leaving
$\widetilde B(\lambda,\partial)=\phi_{0}$, a constant.
Denoting this constant by $c$, we  obtain
\begin{eqnarray}
B(\lambda,\partial)=B_{0}(\lambda,\partial)+c
=\frac{1}{a}g_{1}(\lambda)\Bigl(g_{1}(-\lambda-\partial)-\frac{1}{a}f_{2}(\lambda)g_{1}(-\partial)\Bigr)+c.
\end{eqnarray}
Moreover, if $c \neq 0$, then $f_{2}(\lambda)=\dfrac{a}{2}$. A direct verification shows that the above solution satisfies Eq.~(\ref{108-thq}).
This completes the proof of Case (ii).
\end{proof}
\begin{thm}\label{thm-1}
    Any compatible left-symmetric conformal algebra structure on solvable Lie conformal algebras which are free of rank two as $\mathbb{C}[\partial]$-modules is $R=\mathbb{C}[\partial]x\oplus \mathbb{C}[\partial]y$ with one of the following $\lambda$-products:
\begin{enumerate}
\item \begin{align*}
x_\lambda x = 0, \quad x_\lambda y = 0, \;\;y_\lambda x= 0,  \quad y_\lambda y = B(\lambda,\partial)x,
\end{align*}
where $B(\lambda,\partial)\in\mathbb{C}[\lambda,\partial]$. Denote this left-symmetric conformal algebra by $\mathcal{L}_1(B(\lambda, \partial))$;

\item \begin{align*}
x_\lambda x = 0, \quad x_\lambda y = 0, \;\;y_\lambda x= 0,  \quad y_\lambda y = V(\partial)x+by,
\end{align*}
where $V(\partial)\in\mathbb{C}[\partial]$ and $b\in \mathbb{C}\setminus\{0\}$. Denote this left-symmetric conformal algebra by $\mathcal{L}_2(V(\partial),b)$;

\item 
 \begin{align*}
x_\lambda x = 0, \quad x_\lambda y = bx, \;\;y_\lambda x= 0,  \quad y_\lambda y = (U(\lambda)+V(\partial))x+by,
\end{align*}
where $U(\lambda)\in \mathbb{C}[\lambda]$, $V(\partial)\in \mathbb{C}[\partial]$, and $b\in \mathbb{C}\setminus\{0\}$. Denote this left-symmetric conformal algebra by $\mathcal{L}_3(U(\lambda),V(\partial),b)$;

\item 
 \begin{align*}
x_\lambda x &= 0, \quad x_\lambda y = f_2(\lambda)D(\partial)x+f_2(\lambda)y, \\
y_\lambda x&= 0,  \quad y_\lambda y = -f_2(\lambda)D(-\lambda)D(\partial)x-f_2(\lambda)D(-\lambda)y,
\end{align*}
where $D(\partial)\in \mathbb{C}[\partial]$ and $f_2(\lambda)\in \mathbb{C}[\lambda]\setminus \{0\}$. Denote this left-symmetric conformal algebra by $\mathcal{L}_4(f_2(\lambda), D(\partial))$;

\item 
 \begin{align*}
x_\lambda x = 0, \quad x_\lambda y = 0, \;\;y_\lambda x= g_1(\lambda)x,  \quad y_\lambda y = g_1(\lambda)F(\lambda+\partial)x,
\end{align*}
where $F(\partial)\in \mathbb{C}[\partial]$ and $g_1(\lambda)\in \mathbb{C}[\lambda]\setminus \{0\}$. Denote this left-symmetric conformal algebra by $\mathcal{L}_5(g_1(\lambda), F(\partial))$;

\item 
 \begin{align*}
x_\lambda x = 0, \quad x_\lambda y = 0, \;\;y_\lambda x= g_1(\lambda)x,  \quad y_\lambda y = (g_1(\lambda)F(\lambda+\partial)-bF(\partial))x+by,
\end{align*}
where $F(\partial)\in \mathbb{C}[\partial]$, $g_{1}(\lambda) \in \mathbb{C}[\lambda]\setminus \{0\}$, $b\in \mathbb{C}\setminus\{0\}$ and $g_1(\lambda)\neq b$.
 Denote this left-symmetric conformal algebra by $\mathcal{L}_6(g_1(\lambda), F(\partial),b)$;

\item 
 \begin{align*}
x_\lambda x = 0, \quad x_\lambda y = 0, \;\;y_\lambda x= bx,  \quad y_\lambda y = (U(\lambda)+W(\lambda+\partial)-W(\partial))x+by,
\end{align*}
where $U(\lambda)\in \mathbb{C}[\lambda]$, $W(\partial)\in \mathbb{C}[\partial]$ and $b\in \mathbb{C}\setminus\{0\}$.  Denote this left-symmetric conformal algebra by $\mathcal{L}_7(U(\lambda), W(\partial),b)$;

\item
 \begin{align*}
x_\lambda x &= 0, \quad x_\lambda y = bx, \\
y_\lambda x&= g_1(\lambda)x,  \quad y_\lambda y = (g_1(\lambda)F(\lambda+\partial)-bF(\partial)
+bF(-\lambda))x+by,
\end{align*}
where $F(\partial)\in \mathbb{C}[\partial]$, $g_{1}(\lambda) \in \mathbb{C}[\lambda]\setminus \{0\}$,  $b\in \mathbb{C}\setminus\{0\}$ and $g_1(\lambda)\neq b$.
Denote this left-symmetric conformal algebra by $\mathcal{L}_8(g_1(\lambda), F(\partial),b)$;

\item
 \begin{align*}
x_\lambda x &= 0, \quad x_\lambda y = bx, \\
y_\lambda x&= bx,  \quad y_\lambda y = (F(-\lambda)+F(\lambda+\partial)-F(\partial)+H(\lambda)(\lambda+\partial))x+by,
\end{align*}
where $F(\lambda)$, $H(\lambda)\in \mathbb{C}[\lambda]$, and $b\in \mathbb{C}\setminus\{0\}$. Denote this left-symmetric conformal algebra by $\mathcal{L}_9(F(\lambda), H(\lambda),b)$;

\item
 \begin{align*}
x_\lambda x &= ax, \quad x_\lambda y = g_1(-\lambda-\partial)x, \\
y_\lambda x&= g_1(\lambda)x,  \quad y_\lambda y = \frac{1}{a}(g_1(\lambda)g_1(-\lambda-\partial)-dg_1(-\partial))x+dy,
\end{align*}
where $g_1(\lambda)\in \mathbb{C}[\lambda]$, $d\in\mathbb{C}$ and $a\in \mathbb{C}\setminus\{0\}$. Denote this left-symmetric conformal algebra by $\mathcal{L}_{10}(g_1(\lambda), d, a)$;

\item
 \begin{align*}
x_\lambda x &= ax, \quad x_\lambda y = (g_1(-\lambda-\partial)-\frac{1}{a}f_2(\lambda)g_1(-\partial))x+f_2(\lambda)y, \\
y_\lambda x&= g_1(\lambda)x,  \quad y_\lambda y = \frac{1}{a}g_1(\lambda)(g_1(-\lambda-\partial)-\frac{1}{a}f_2(\lambda)g_1(-\partial))x+\frac{1}{a}f_2(\lambda)g_1(\lambda)y,
\end{align*}
where $f_2(\lambda)\in \mathbb{C}[\lambda]\setminus\{0\}$, $g_1(\lambda)\in \mathbb{C}[\lambda]$, and $a\in \mathbb{C}\setminus\{0\}$. Denote this left-symmetric conformal algebra by $\mathcal{L}_{11}(f_2(\lambda), g_1(\lambda), a)$;

\item
 \begin{align*}
x_\lambda x &= ax, \quad x_\lambda y = (g_1(-\lambda-\partial)-\frac{1}{2}g_1(-\partial))x+\frac{a}{2}y, \\
y_\lambda x&= g_1(\lambda)x,  \quad y_\lambda y = (\frac{1}{a}g_1(\lambda)g_1(-\lambda-\partial)-\frac{1}{2a}g_1(\lambda)g_1(-\partial)+c)x+\frac{1}{2}g_1(\lambda)y,
\end{align*}
where $g_1(\lambda)\in \mathbb{C}[\lambda]$, and $a$, $c\in \mathbb{C}\setminus\{0\}$. Denote this left-symmetric conformal algebra by $\mathcal{L}_{12}( g_1(\lambda), a, c)$.
\end{enumerate}
\end{thm}
 \begin{proof}
 For each left-symmetric conformal algebra in the list, the derived algebra of its sub-adjacent Lie conformal algebra is either zero or contained in an abelian Lie conformal subalgebra of rank one. Hence $\mathfrak{g}(R)^{(2)}=0$. Therefore, the sub-adjacent Lie conformal algebras of these left-symmetric conformal algebras are solvable.  Then it is straightforward by Propositions \ref{pro-p1},  \ref{pp-3.6},  \ref{lem-l6}, \ref{pro-p3}, \ref{pro-p4} and \ref{pro-p5}.
\end{proof}
\begin{rmk}
The left-symmetric conformal algebra $\tilde{\mathcal{S}}(k, p(\lambda), a)$ given in \cite[Theorem 3.2]{LSW} is just $\mathcal{L}_9(F(\lambda), H(\lambda), b)$ by letting $x=M$, $y=Y$, $F(\lambda)=p(\lambda)+a$, $H(\lambda)=1$ and $b=k$. Similarly, the left-symmetric conformal algebra $\tilde{\mathcal{S}}(\psi(\partial, \lambda))$ given in \cite[Theorem 3.2]{LSW} is $\mathcal{L}_1(B(\lambda,\partial))$ by letting $x=M$, $y=Y$, and $B(\lambda,\partial)=\psi(\partial, \lambda)+\psi(\partial, -\lambda-\partial)+\lambda$.
\end{rmk}
\section{Isomorphism classes of compatible left-symmetric conformal algebras on solvable Lie conformal algebras of rank two}\mbox{}
In this section, we give a characterization of isomorphism classes of those left-symmetric conformal algebras given in Theorem \ref{thm-1} and determine which left-symmetric conformal algebras are Novikov conformal algebras.

\delete{\begin{pro}
    In Theorem \ref{thm-1}, Cases (a)--(k) are all non-simple, and Case (l) is simple, i.e. $\mathcal{B}_3( g_1(\lambda), a, C)$ is a simple left-symmetric conformal algebra for any $g_1(\lambda)\in \mathbb{C}[\lambda]$, and $a$, $C\in \mathbb{C}\setminus\{0\}$.
\end{pro}
\begin{proof}
For Cases~(a)-(c) and (e)-(j), the submodule $\mathbb{C}[\partial]x$ is a proper ideal.
For Case~(d), the submodule $\mathbb{C}[\partial]\bigl(D(\partial)x+y\bigr)$ is a proper ideal.
For Case~(k), the submodule $\mathbb{C}[\partial]\bigl(g_1(-\partial)x-ay\bigr)$ is a proper ideal.
Therefore all these cases are non-simple.
Next, we prove that $\mathcal{B}_3( g_1(\lambda), a, C)$ is simple. \delete{Its $\lambda$-brackets are
\[
\begin{aligned}
x_\lambda x &= a x,\\
x_\lambda y &= \bigl[g_1(-\lambda-\partial)-\frac{1}{2} g_1(-\partial)\bigr]x+\frac{a}{2}y,\\
y_\lambda x &= g_1(\lambda) x,\\
y_\lambda y &= \Bigl[\frac{1}{a} g_1(\lambda)g_1(-\lambda-\partial)-\frac{1}{2a}g_1(\lambda)g_1(-\partial)+C\Bigr]x+\frac{1}{2} g_1(\lambda) y,
\end{aligned}
\]
where $a\in \mathbb{C}\setminus\{0\}$, $C\in \mathbb{C}\setminus\{0\}$, and $g_1(\lambda)$ is an arbitrary polynomial.}
Assume that $I$ is a nonzero ideal of $\mathcal{B}_3( g_1(\lambda), a, C)$. Let
$v=A(\partial)x+B(\partial)y$ be a nonzero element in $I$.

Then there exist $R(\lambda,\partial),S(\lambda,\partial),U(\lambda,\partial),V(\lambda,\partial)\in\mathbb{C}[\lambda,\partial]$ such that
\[
x_\lambda v = R(\lambda,\partial)\,v,\quad y_\lambda v = S(\lambda,\partial)\,v,\quad v_\lambda x = U(\lambda,\partial)\,v,\quad v_\lambda y = V(\lambda,\partial)\,v.
\]

We assert that $B(\partial)\neq0$.
If $B(\partial)=0$, then $v=A(\partial)x,A(\partial)\neq0$.  Then \(I = \mathbb{C}[\partial]x\), and a direct computation shows that \(I\) is not an ideal of \(L\). Therefore $B(\partial)\neq0$.

Now we assert that $B(\partial)$ is a nonzero constant.
Suppose $\deg B(\partial)\ge1$.  The $y$-component of $x_\lambda v$ gives $\dfrac{a}{2}B(\lambda+\partial)=R (\lambda,\partial)B(\partial)$.
If $\beta$ is a root of $B(\partial)$, setting $\partial=\beta$ leads to $\dfrac{a}{2}B(\lambda+\beta)=0$ for all $\lambda$, hence $B(\partial)=0$,
contradiction.  Thus $B(\partial)$ is a nonzero constant; denote it simply by $B$, so we have $v=A(\partial)x+By,B\neq0$.

From $v_\lambda x = U(\lambda,\partial)v$, the $y$-component forces $U(\lambda,\partial)=0$; the $x$-component yields
\begin{align}\label{128-thq}
    a A(-\lambda) + B\,g_1(\lambda) = 0.
\end{align}
From $v_\lambda y = V(\lambda,\partial) v$, compute its $y$-component:
\[
\frac{a}{2}A(-\lambda) + \frac{B}{2}g_1(\lambda) =BV(\lambda,\partial)= 0.
\] Hence $V(\lambda,\partial)=0$.  Now the $x$-component of $v_\lambda y$ must also be zero:
\begin{align}
    A(-\lambda)\bigl[g_1(-\lambda-\partial)-\frac{1}{2} g_1(-\partial)\bigr]
+B\bigl[\frac{1}{a}g_1(\lambda)g_1(-\lambda-\partial)- \frac{1}{2a}g_1(\lambda)g_1(-\partial)+C\bigr]=0.
\end{align}
Using Eq.~(\ref{128-thq}) to replace $A(-\lambda)$ we have $BC=0$, contradiction.
Thus no nonzero proper ideal exists; $L$ is simple.
\end{proof}}
\begin{pro}\label{pro-1}
     Cases (a)--(l) of Theorem~\ref{thm-1} can be reduced to one of the following ten normal forms, by choosing some invertible change of $\mathbb{C}[\partial]$-basis
\begin{eqnarray*}
\begin{pmatrix} X \\ Y \end{pmatrix} = M \begin{pmatrix} x \\ y \end{pmatrix}.
\end{eqnarray*}
Since all transformation matrices belong to $\mathrm{GL}_2(\mathbb{C}[\partial])$, under each change of $\mathbb{C}[\partial]$-basis, the $\lambda$-products take the corresponding normal form listed below.
\small\renewcommand{\arraystretch}{1.3}%
\setlength{\tabcolsep}{4pt}%
\begin{center}
\begin{longtable}{|c|c|c|}
\hline
\textbf{Original case} & \textbf{Transformation matrix $M$} & \textbf{Normal form}\\
\hline
\endfirsthead
\multicolumn{3}{|r|}{\textit{Continued}}\\
\hline
\textbf{Original case} & \textbf{Transformation matrix $M$} & \textbf{Normal form}\\
\hline
\endhead
\hline
\endfoot
\hline
\endlastfoot
$\mathcal{L}_1(B(\lambda, \partial))$ & $\begin{pmatrix}1&0\\0&1\end{pmatrix}\;$ & $\mathcal{A}_1(B(\lambda,\partial))$\\
\hline
$\mathcal{L}_2(V(\partial),b)$ & $\begin{pmatrix}1&0\\ \frac{V(\partial)}{b^2}&\frac1b\end{pmatrix}$ & $\mathcal{A}_2$\\
\hline
$\mathcal{L}_3(U(\lambda),V(\partial),b)$ & $\begin{pmatrix}1&0\\ -\frac{U(-\partial)}{b^2}&\frac1b\end{pmatrix}$ & $\mathcal{A}_3\! \Bigl(\frac{U(-\partial)+V(\partial)}{b^2}\Bigr)$\\
\hline
$\mathcal{L}_4(f_2(\lambda), D(\partial))$ & $\begin{pmatrix}1&0\\ D(\partial)&1\end{pmatrix}$ & $\mathcal{A}_4(f_2(\lambda))$\\
\hline
$\mathcal{L}_5(g_1(\lambda), F(\partial))$ & $\begin{pmatrix}-F(\partial)&1\\1&0\end{pmatrix}$ & $\mathcal{A}_4(g_1(\lambda))$\\
\hline
$\mathcal{L}_6(g_1(\lambda), F(\partial),b)$ & $\begin{pmatrix}-\frac{F(\partial)}{b}&\frac1b\\1&0\end{pmatrix}$ & $\mathcal{B}_2\!\Bigl(\frac{g_1(\lambda)}{b}\Bigr)$\\
\hline
$\mathcal{L}_7(U(\lambda), W(\partial),b)\;(U(\lambda)\not\sim\lambda)$ & $\begin{pmatrix}1&0\\ -\frac{W(\partial)}{b^2}&\frac1b\end{pmatrix}$ & $\mathcal{A}_5\!\Bigl(\frac{U(\lambda)}{b^2}\Bigr)$\\
\hline
$\mathcal{L}_7(c\lambda, W(\partial),b)\;,c \in \mathbb{C}$ & $\begin{pmatrix}-\frac{c}{b^2}\partial-\frac{W(\partial)}{b^2}&\frac1b\\1&0\end{pmatrix}$ & $\mathcal{B}_2(1)$\\
\hline
$\mathcal{L}_8(g_1(\lambda), F(\partial),b)\;$ & $\begin{pmatrix}1&0\\ -\frac{F(\partial)}{b}&\frac1b\end{pmatrix}$ & $\mathcal{A}_6\!\Bigl(\frac{g_1(\lambda)}{b}\Bigr)$\\
\hline
$\mathcal{L}_9(F(\lambda), H(\lambda),b)$ & $\begin{pmatrix}1&0\\ -\frac{F(\partial)}{b^2}&\frac1b\end{pmatrix}$ & $\mathcal{A}_7\!\Bigl(\frac{H(\lambda)}{b^2}\Bigr)$\\
\hline
$\mathcal{L}_{10}(g_1(\lambda), 0, a);$ & $\begin{pmatrix}-\frac{g_1(-\partial)}{a}&1\\\frac1a&0\end{pmatrix}$ & $\mathcal{A}_2$\\
\hline
$\mathcal{L}_{10}(g_1(\lambda), d, a);\;(d\neq0)$ & $\begin{pmatrix}\frac1a&0\\-\frac{g_1(-\partial)}{ad}&\frac1d\end{pmatrix}$ & $\mathcal{B}_1$\\
\hline
$\mathcal{L}_{11}(f_2(\lambda), g_1(\lambda), a)$ & $\begin{pmatrix}\frac1a&0\\ g_1(-\partial)&-a\end{pmatrix}$ & $\mathcal{B}_2\!\Bigl(\frac{f_2(\lambda)}{a}\Bigr)$\\
\hline
$\mathcal{L}_{12}( g_1(\lambda), a, c)$ & $\begin{pmatrix}\frac1a&0\\-\frac{1}{a\sqrt{ac}} g_1(-\partial)&\frac{1}{\sqrt{ac}}\end{pmatrix}$ & $\mathcal{S}$\\
\hline
\end{longtable}
\end{center}
{\normalsize}
Ten normal forms are listed as follows.
\begin{enumerate}
\item[$\mathcal{A}_1(B(\lambda,\partial))$:]
$X_\lambda X=0,\quad X_\lambda Y = 0,\quad Y_\lambda X = 0,\quad Y_\lambda Y = B(\lambda,\partial)X$,
$B(\lambda,\partial)\in\mathbb{C}[\lambda,\partial]$;
\item[$\mathcal{A}_2$:]
$X_\lambda X=0,\quad X_\lambda Y = 0,\quad Y_\lambda X = 0,\quad Y_\lambda Y = Y$;
\item[$\mathcal{A}_3(V(\partial))$:]
$X_\lambda X=0,\quad X_\lambda Y = X,\quad Y_\lambda X = 0,\quad Y_\lambda Y = V(\partial)X + Y$,
  $V(\partial)\in\mathbb{C}[\partial]$;
\item[$\mathcal{A}_4(f(\lambda))$:]
$X_\lambda X=0,\quad X_\lambda Y = f(\lambda)Y,\quad Y_\lambda X = 0,\quad Y_\lambda Y = 0$,
  $0\neq f(\lambda)\in\mathbb{C}[\lambda]$;
\item[$\mathcal{A}_5(U(\lambda))$:]
$X_\lambda X=0,\quad X_\lambda Y = 0,\quad Y_\lambda X = X,\quad Y_\lambda Y = U(\lambda)X + Y$,
  $U(\lambda)\in\mathbb{C}[\lambda]$ and $U(\lambda)\not\sim\lambda$ (i.e. there is no constant $c$ with $U(\lambda)=c\lambda$);
\item[$\mathcal{A}_6(g(\lambda))$:]
$X_\lambda X=0,\quad X_\lambda Y = X,\quad Y_\lambda X= g(\lambda)X,\quad Y_\lambda Y= Y$,
  $0\neq g(\lambda)\in\mathbb{C}[\lambda]$, $g(\lambda)\neq 1$. When $g(\lambda)=1$, this coincides with $\mathcal{A}_7(0)$;
\item[$\mathcal{A}_7(H(\lambda))$:]
$X_\lambda X=0,\quad X_\lambda Y = X,\quad Y_\lambda X = X,\quad Y_\lambda Y= H(\lambda)(\lambda+\partial)X+ Y$,
  $H(\lambda)\in\mathbb{C}[\lambda]$;
\item[$\mathcal{B}_1$:]
$X_\lambda X=X,\quad X_\lambda Y= 0,\quad Y_\lambda X = 0,\quad Y_\lambda Y = Y$;
\item[$\mathcal{B}_2(f(\lambda))$:]
$X_\lambda X=X,\quad X_\lambda Y = f(\lambda)Y,\; Y_\lambda X = 0,\; Y_\lambda Y = 0$,
  $0\neq f(\lambda)\in\mathbb{C}[\lambda]$;
\item[$\mathcal{S}$:]
$X_\lambda X=X,\quad X_\lambda Y= \frac{1}{2}Y,\;Y_\lambda X= 0,\;Y_\lambda Y = X$.
\end{enumerate}
\end{pro}
\begin{proof}
It is straightforward by direct computations.
\end{proof}
{\begin{pro}\label{pp1}
    In Proposition \ref{pro-1}, $\mathcal{S}$ is simple and other normal forms are non-simple. Therefore, $\mathcal{S}$ is not isomorphic to other normal forms.
\end{pro}
\begin{proof}
For $\mathcal{A}_1(B(\lambda,\partial))$, $\mathcal{A}_2$, $\mathcal{A}_3(V(\partial))$, $\mathcal{A}_5(U(\lambda))$, $\mathcal{A}_6(g(\lambda))$, $\mathcal{A}_7(H(\lambda))$ and $\mathcal{B}_1$, the submodule $\mathbb{C}[\partial]X$ is a proper ideal.
For $\mathcal{A}_4(f(\lambda))$ and $\mathcal{B}_2(f(\lambda))$ , the submodule $\mathbb{C}[\partial]Y$ is a proper ideal.
Therefore all these cases are non-simple.  Let $I$ be a nonzero ideal of $\mathcal{S}$ and $f(\partial)X+g(\partial)Y$ be a nonzero element in $I$. If $f(\partial)\neq 0$, by $(f(\partial)X+g(\partial)Y)_\lambda X=f(-\lambda)X$ and considering the coefficient of the highest degree in $\lambda$, we have $X\in I$.  Similarly, if $f(\partial)=0$, then $g(\partial)\neq 0$. By $Y_\lambda(f(\partial)X+g(\partial)Y)=g(\lambda+\partial)X$ and considering the coefficient of the highest degree in $\lambda$, we also have $X\in I$. Then by $X_\lambda Y= \frac{1}{2}Y$, we have $Y\in I$. Since $I$ is a $\mathbb{C}[\partial]$-submodule of $\mathcal{S}$, we have $I=\mathcal{S}$. Therefore, $\mathcal{S}$ is simple.
\delete{Next, we prove that $\mathcal{S}$ is simple.
Assume that $I$ is a nonzero ideal of $\mathcal{S}(c)$. Let
$v=A(\partial)x+B(\partial)y$ be a nonzero element in $I$.
Then there exist $R(\lambda,\partial),S(\lambda,\partial),U(\lambda,\partial),V(\lambda,\partial)\in\mathbb{C}[\lambda,\partial]$ such that
\[
x_\lambda v = R(\lambda,\partial)\,v,\quad y_\lambda v = S(\lambda,\partial)\,v,\quad v_\lambda x = U(\lambda,\partial)\,v,\quad v_\lambda y = V(\lambda,\partial)\,v.
\]
If $A(\partial)=0$, then $v=B(\partial)y,B(\partial)\neq0$.   From $x_\lambda v=\frac{1}{2}B(\lambda+\partial)y\in I$  and $y_\lambda v=cB(\lambda+\partial)x\in I$ we have $x,y\in I$ i.e. $I=\mathcal{S}(c)$.
If $A(\partial)\neq0$, by $v_\lambda x=A(-\lambda)x\in I$ and $x_\lambda y=\frac{1}{2}y$ we have $x,y\in I$ i.e. $I=\mathcal{S}(c)$. Thus no nonzero proper ideal exists; $\mathcal{S}(c)$ is simple.}
\end{proof}}
\delete{\begin{thm}
Let $L=\mathbb{C}[\partial]x\oplus\mathbb{C}[\partial]y$ be a non-abelian two-dimensional left-symmetric conformal algebra over $\mathbb{C}$ whose $\lambda$-brackets are taken from \textbf{Cases 1Â¨C12 in Theorem 3.11}. Then $L$ is isomorphic to exactly one of the following normal forms; these normal forms are pairwise non-isomorphic. The abelian algebra (all $\lambda$-brackets zero) constitutes a separate trivial class.
\noindent\textbf{Class A: there exists a basis with $x_\lambda x = 0$}
\begin{enumerate}
\item[$\mathcal{A}_1(B)$:]
  $x_\lambda y = 0,\quad y_\lambda x = 0,\quad y_\lambda y = B(\lambda,\partial)x$,
  $0\neq B(\lambda,\partial)\in\mathbb{C}[\lambda,\partial]$.
\item[$\mathcal{A}_2$:]
  $x_\lambda y = 0,\quad y_\lambda x = 0,\quad y_\lambda y = y$.
\item[$\mathcal{A}_3(V)$:]
  $x_\lambda y = x,\quad y_\lambda x = 0,\quad y_\lambda y = V(\partial)x + y$,
  $V(\partial)\in\mathbb{C}[\partial]$.
\item[$\mathcal{A}_4(f)$:]
  $x_\lambda y = f(\lambda)y,\quad y_\lambda x = 0,\quad y_\lambda y = 0$,
  $0\neq f(\lambda)\in\mathbb{C}[\lambda]$.
\item[$\mathcal{A}_5(U)$:]
  $x_\lambda y = 0,\quad y_\lambda x = x,\quad y_\lambda y = U(\lambda)x + y$,
  $0\neq U(\lambda)\in\mathbb{C}[\lambda]$ and $U(\lambda)\not\sim\lambda$ (i.e. there is no constant $c$ with $U(\lambda)=c\lambda$).
\item[$\mathcal{A}_6(g)$:]
  $x_\lambda y = x,\quad y_\lambda x = g(\lambda)x,\quad y_\lambda y = y$,
  $0\neq g(\lambda)\in\mathbb{C}[\lambda]$, $g(\lambda)\neq 1$. When $g(\lambda)=1$ this coincides with $\mathcal{A}_7(0)$.
\item[$\mathcal{A}_7(H)$:]
  $x_\lambda y = x,\quad y_\lambda x = x,\quad y_\lambda y = H(\lambda)(\lambda+\partial)x + y$,
  $H(\lambda)\in\mathbb{C}[\lambda]$.
\end{enumerate}
\noindent\textbf{Class B: there exists a basis with $x_\lambda x = x$}
\begin{enumerate}
\item[$\mathcal{B}_1$:]
  $x_\lambda y = 0,\quad y_\lambda x = 0,\quad y_\lambda y = y$.
\item[$\mathcal{B}_2(f)$:]
  $ x_\lambda y = f(\lambda)y,\; y_\lambda x = 0,\; y_\lambda y = 0$,
  $0\neq f(\lambda)\in\mathbb{C}[\lambda]$.
  \item[$\mathcal{S}(g,C)$:]
  \[
  \begin{aligned}
  x_\lambda y &= \bigl[g(-\lambda-\partial)-\frac{1}{2} g(-\partial)\bigr]x +\frac{1}{2}  y,\\
  y_\lambda x &= g(\lambda)x,\\
  y_\lambda y &= \bigl[g(\lambda)g(-\lambda-\partial)-\frac{1}{2}  g(\lambda)g(-\partial)+C\bigr]x + \frac{1}{2}  g(\lambda)y,
  \end{aligned}
  \]
  where $g(\lambda)\in\mathbb{C}[\lambda]$, $C\in \mathbb{C}\setminus\{0\}$.
\end{enumerate}
\end{thm}
\begin{proof}
    The proof of this theorem is given by the following content.
\end{proof}

\begin{lem}
Every non-abelian algebra among Cases~1Â¨C12 of Theorem~3.11 can be transformed, by an invertible $\mathbb{C}[\partial]$-module basis change
\[
\begin{pmatrix} X \\ Y \end{pmatrix} = M \begin{pmatrix} x \\ y \end{pmatrix},
\]
into a normal form appearing in Theorem~3.13.   The correspondence is given by the following table (all transformation matrices belong to $\mathrm{GL}_2(\mathbb{C}[\partial])$, hence each transformation is an isomorphism of $\mathbb{C}[\partial]$-modules):

\small\renewcommand{\arraystretch}{1.3}%
\setlength{\tabcolsep}{4pt}%
\begin{center}
\begin{longtable}{|c|c|c|}
\hline
\textbf{Original case} & \textbf{Transformation matrix $M$} & \textbf{Normal form}\\
\hline
\endfirsthead
\multicolumn{3}{|r|}{\textit{Continued}}\\
\hline
\textbf{Original case} & \textbf{Transformation matrix $M$} & \textbf{Normal form}\\
\hline
\endhead
\hline
\endfoot
\hline
\endlastfoot
$1\;(B(\lambda,\partial)\neq0)$ & $\begin{pmatrix}1&0\\0&1\end{pmatrix}\;$ & $\mathcal{A}_1(B(\lambda,\partial))$\\
\hline
$2$ & $\begin{pmatrix}1&0\\ \frac{V(\partial)}{b^2}&\frac1b\end{pmatrix}$ & $\mathcal{A}_2$\\
\hline
$3$ & $\begin{pmatrix}1&0\\ -\frac{U(-\partial)}{b^2}&\frac1b\end{pmatrix}$ & $\mathcal{A}_3\! \Bigl(\frac{U(-\partial)+V(\partial)}{b^2}\Bigr)$\\
\hline
$4$ & $\begin{pmatrix}1&0\\ D(\partial)&1\end{pmatrix}$ & $\mathcal{A}_4(f_2(\lambda))$\\
\hline
$5$ & $\begin{pmatrix}-F(\partial)&1\\1&0\end{pmatrix}$ & $\mathcal{A}_4(g_1(\lambda))$\\
\hline
$6$ & $\begin{pmatrix}-\frac{F(\partial)}{b}&\frac1b\\1&0\end{pmatrix}$ & $\mathcal{B}_2\!\Bigl(\frac{g_1(\lambda)}{b}\Bigr)$\\
\hline
$7\;(U(\lambda)\not\sim\lambda,U(\lambda)\neq0)$ & $\begin{pmatrix}1&0\\ -\frac{W(\partial)}{b^2}&\frac1b\end{pmatrix}$ & $\mathcal{A}_5\!\Bigl(\frac{U(\lambda)}{b^2}\Bigr)$\\
\hline
$7\;(U=c\lambda), c \in \mathbb{C}$ & $\begin{pmatrix}-\frac{c}{b^2}\partial-\frac{W(\partial)}{b^2}&\frac1b\\1&0\end{pmatrix}$ & $\mathcal{B}_2(1)$\\
\hline
$8\;$ & $\begin{pmatrix}1&0\\ -\frac{F(\partial)}{b}&\frac1b\end{pmatrix}$ & $\mathcal{A}_6\!\Bigl(\frac{g_1(\lambda)}{b}\Bigr)$\\
\hline
$9$ & $\begin{pmatrix}1&0\\ -\frac{F(\partial)}{b^2}&\frac1b\end{pmatrix}$ & $\mathcal{A}_7\!\Bigl(\frac{H(\lambda)}{b^2}\Bigr)$\\
\hline
$10\;(G=0)$ & $\begin{pmatrix}-\frac{g_1(-\partial)}{a}&1\\\frac1a&0\end{pmatrix}$ & $\mathcal{A}_2$\\
\hline
$10\;(G\neq0)$ & $\begin{pmatrix}\frac1a&0\\-\frac{g_1(-\partial)}{aG}&\frac1G\end{pmatrix}$ & $\mathcal{B}_1$\\
\hline
$11$ & $\begin{pmatrix}\frac1a&0\\ g_1(-\partial)&-a\end{pmatrix}$ & $\mathcal{B}_2\!\Bigl(\frac{f_2(\lambda)}{a}\Bigr)$\\
\hline
$12$ & $\begin{pmatrix}\frac1a&0\\0&1\end{pmatrix}$ & $\mathcal{S}(g_1,aC)$\\
\hline
\end{longtable}
\end{center}
{\normalsize}
\end{lem}
\begin{proof}
    For each row the determinant of the given matrix is a nonzero constant, hence the transformation is invertible.  Expanding the $\lambda$-brackets in the new basis $\{X,Y\}$ and using the original bracket expressions, one verifies directly that they coincide with the normal form listed in the right column.
\end{proof}
\begin{lem}
    For each non-abelian normal form in Theorem 3.13 we determine all nonzero proper ideals.  Let $I = \mathbb{C}[\partial]v$ with $v = E(\partial)x + F(\partial)y \neq 0$ .  Solving the ideal conditions $x_\lambda v = R(\lambda,\partial) v$, $y_\lambda v = S(\lambda,\partial) v$ yields the following data.
\vspace{1mm}
\begin{center}
\begin{longtable}{|c|c|c|}
\hline
Normal form & Number of ideals &  Ideals\\
\hline
\endfirsthead
\multicolumn{3}{|r|}{Continued.}\\
\hline
Normal form & Number of ideals & Generators of ideals\\
\hline
\endhead
\hline
\endfoot
\hline
\endlastfoot
$\mathcal{A}_1(B)$ & 1 & $I_x=\mathbb{C}[\partial]x$\\
\hline
$\mathcal{A}_2$ & 2 & $I_x=\mathbb{C}[\partial]x,I_y=\mathbb{C}[\partial]y$\\
\hline
$\mathcal{A}_3(V)$ & 1 & $I_x=\mathbb{C}[\partial]x$\\
\hline
$\mathcal{A}_4(f)$ & 1 & $I_y=\mathbb{C}[\partial]y$\\
\hline
$\mathcal{A}_5(U)$ & 1 & $I_x=\mathbb{C}[\partial]x$\\
\hline
$\mathcal{A}_6(g)$ & 1 & $I_x=\mathbb{C}[\partial]x$\\
\hline
$\mathcal{A}_7(H)$ & 1 & $I_x=\mathbb{C}[\partial]x$\\
\hline
$\mathcal{B}_1$ & 2 & $I_x=\mathbb{C}[\partial]x,I_y=\mathbb{C}[\partial]y$\\
\hline
$\mathcal{B}_2(f)$ & 1 & $I_y=\mathbb{C}[\partial]y$\\
\hline
$\mathcal{S}(g,C)$ & 0 & -\\
\hline
\end{longtable}
\end{center}
\end{lem}
\begin{proof}
We take the computation of ideals for \(\mathcal{A}_2\) and \(\mathcal{B}_2(f)\) as examples; the method for the other cases is the same.

\textbf{Computation of the ideals for \(\mathcal{A}_2\):}
Let \(I = \mathbb{C}[\partial]v\) with \(v = E(\partial)x + F(\partial)y \neq 0\) .
There exist \(R(\lambda,\partial),S(\lambda,\partial) \in \mathbb{C}[\lambda,\partial]\) such that
\[
x_\lambda v = R(\lambda,\partial)\,v,\qquad y_\lambda v = S(\lambda,\partial)\,v.
\]
 By computing  \(y_\lambda v\) we have
\[
y_\lambda v = F(\partial+\lambda)\,y=S(\lambda,\partial)E(\partial)x+S(\lambda,\partial)F(\partial)y.
\]
Comparing coefficients of \(x\) and \(y\) gives
\begin{eqnarray}
   \label{130-thq}&&S(\lambda,\partial)E(\partial)=0,\\
    \label{131-thq}&&F(\lambda+\partial)=S(\lambda,\partial)F(\partial).
\end{eqnarray}
If $E(\partial)\neq0$, then Eq.~(\ref{130-thq}) forces $S(\lambda,\partial)=0$, and then by Eq.~(\ref{131-thq}) we have $F(\lambda+\partial)=0$ i.e. $F(\partial)=0$. Then $I = \mathbb{C}[\partial]v=\mathbb{C}[\partial]x=I_x$. A direct computation shows that \(I_x\) is indeed a proper ideal of \(\mathcal{A}_2\).
If $E(\partial)=0$,  then $I = \mathbb{C}[\partial] y = I_y$.  One checks directly that $I_y$ is a proper ideal of $\mathcal{A}_2$.

\textbf{Computation of the ideals for  \(\mathcal{B}_2(f)\):}
Let \(I = \mathbb{C}[\partial]v\) with \(v = E(\partial)x + F(\partial)y \neq 0\) .
There exist \(R(\lambda,\partial),S(\lambda,\partial)\in\mathbb{C}[\lambda,\partial]\) such that
\[
x_\lambda v = R(\lambda,\partial) v,\qquad y_\lambda v = S(\lambda,\partial) v.
\]
By computing \(x_\lambda v\) we have
\[
x_\lambda v = E(\partial+\lambda)x + F(\partial+\lambda)f(\lambda)y = R(\lambda,\partial) E(\partial) x + R(\lambda,\partial) F(\partial) y.
\]
Comparing coefficients yields
\begin{eqnarray}
     \label{132-thq}&&E(\partial+\lambda) = R(\lambda,\partial) E(\partial),\\
    \label{133-thq}&&F(\partial+\lambda)f(\lambda) = R(\lambda,\partial) F(\partial).
\end{eqnarray}
If $F(\partial)=0$, then $I=\mathbb{C}[\partial]x$.  A direct verification shows that $I$ is not an ideal of $\mathcal{B}_2(f)$; hence $F(\partial)\neq0$.  From Eq.~(\ref{133-thq}) we obtain that $F(\partial)$ is a nonzero constant, which we denote by $F\neq0$.  Then from Eqs.~(\ref{132-thq})--(\ref{133-thq}) we have $R(\lambda,\partial)=f(\lambda)$ and $E(\lambda+\partial)=f(\lambda)E(\partial)$.  If $E(\partial)=0$, then $I=\mathbb{C}[\partial]y$, and a direct verification shows that $I$ is an ideal of $\mathcal{B}_2(f)$.  If $E(\partial)\neq0$, then from $E(\lambda+\partial)=f(\lambda)E(\partial)$ we obtain that $E(\partial)$ is a nonzero constant, which we denote by $E\neq0$.  Thus $v=Ex+Fy$ with $EF\neq0$, so $I=\mathbb{C}[\partial]v$, and a direct verification shows that $I$ is not an ideal of $\mathcal{B}_2(f)$.

From the above discussion, we obtain:
\begin{itemize}
\item $\mathcal{A}_1(B)$:  $y_\lambda v = F(\partial+\lambda)B(\lambda,\partial) x=S(\lambda,\partial)E(\partial)x+S(\lambda,\partial)F(\partial)y$ $\Rightarrow$ $F(\partial)=0$, only $I_x$.
\item $\mathcal{A}_3(V)$,$\mathcal{A}_6(g)$,$\mathcal{A}_7(H)$: $x_\lambda v = F(\lambda+\partial) x=R(\lambda,\partial)E(\partial)x+R(\lambda,\partial)F(\partial)y$ $\Rightarrow$ $F(\partial)=0$, only $I_x$.
\item $\mathcal{A}_4(f)$: $x_\lambda v = F(\lambda+\partial) f (\lambda)y=R(\lambda,\partial)E(\partial)x+R(\lambda,\partial)F(\partial)y$ $\Rightarrow$ $R(\lambda,\partial)\neq0,E(\partial)=0$, only $I_y$.
\item $\mathcal{A}_5(U)$: $y_\lambda v = [E(\lambda+\partial)+F(\lambda+\partial)U(\lambda)] x + F(\lambda+\partial) y=S(\lambda,\partial)E(\partial)x+S(\lambda,\partial)F(\partial)y.$ If $F(\partial) \neq 0$, then $F(\partial)$ must be a nonzero constant, say $F$, then $S(\lambda,\partial)=1$, $E(\lambda+\partial)-E(\partial)=-FU(\lambda)$, which contradicts $U(\lambda)\neq 0$ and $U(\lambda)\not\sim\lambda$ .  $\Rightarrow$ $F(\partial)=0$, only $I_x$.
\item $\mathcal{B}_1$: from $x_\lambda v$ and $y_\lambda v$, comparing the coefficients of $x$ and $y$ yields either $E(\partial)=0,\;F(\partial)\neq0$ or $F(\partial)=0,\;E(\partial)\neq0$. $\Rightarrow$ $I_x$ and $I_y$.
\item $\mathcal{S}(g,C)$: already proved to have no proper ideals.
\end{itemize}
\end{proof}
\begin{lem}
    $\mathcal{A}_2$ and $\mathcal{B}_1$ in Theorem~3.13 are not isomorphic.
\end{lem}
\begin{proof}
    Suppose there exists an isomorphism $\varphi: \mathcal{A}_2 \to \mathcal{B}_1$ of left-symmetric conformal algebras, and write
\[
\varphi(x) = P(\partial)x' + Q(\partial)y',\qquad \varphi(y) = R(\partial)x' + S(\partial)y',
\]
where $P(\partial),Q(\partial),R(\partial),S(\partial) \in \mathbb{C}[\partial]$ and the determinant $\Delta = P(\partial)S(\partial) - Q(\partial)R(\partial)$ is a nonzero constant (hence $\Delta \in \mathbb{C}^\times$). Hereafter we use \(\{x,y\}\) for a basis of the domain algebra and \(\{x',y'\}\) for a basis of the codomain.
In $\mathcal{A}_2$ we have $x_\lambda x = 0$, therefore $\varphi(x)_\lambda \varphi(x) = 0$.  Expanding this in $\mathcal{B}_1$ we have
\[
\begin{aligned}
\varphi(x)_\lambda \varphi(x) &= (P(\partial) x' + Q(\partial) y')_\lambda (P(\partial) x' + Q(\partial) y') \\
&= P(-\lambda)P(\partial+\lambda)x' + Q(-\lambda)Q(\partial+\lambda)y'=0.
\end{aligned}
\]
Then we have
\[
P(-\lambda)P(\partial+\lambda) = 0,\qquad Q(-\lambda)Q(\partial+\lambda) = 0
\]
i.e. $P(\partial)=Q(\partial)=0$. Consequently
\[
\Delta = P(\partial)S(\partial) - Q(\partial)R(\partial)= 0,
\]
which contradicts $\Delta \neq 0$.  Therefore $\mathcal{A}_2 \not\cong \mathcal{B}_1$.
\end{proof}

\delete{\begin{lem}
    In Theorem 3.13, the algebras $\mathcal{A}_1(B)$, $\mathcal{A}_3(V)$, $\mathcal{A}_4(f)$, $\mathcal{A}_5(U)$, $\mathcal{A}_6(g)$, $\mathcal{A}_7(H)$, and $\mathcal{B}_2(f)$ are mutually non-isomorphic.
\end{lem}
\begin{proof}
    We illustrate the method with the proofs that $\mathcal{B}_2(f)\not\cong\mathcal{A}_5(U)$ and \(\mathcal{A}_6(g) \not\cong \mathcal{A}_7(H)\); the other cases are analogous.

$\mathcal{B}_2(f)\not\cong\mathcal{A}_5(U)$:

Assume there exists an isomorphism \(\varphi: \mathcal{A}_5(U) \to \mathcal{B}_2(f)\) of left-symmetric conformal algebras,  and write
\[
\varphi(x) = P(\partial)x' + Q(\partial)y',\qquad \varphi(y) = R(\partial)x' + S(\partial)y',
\]
with \(P(\partial),Q(\partial),R(\partial),S(\partial) \in \mathbb{C}[\partial]\) and \(\Delta = P(\partial)S(\partial) - Q(\partial)R(\partial) \neq 0\).

In \(\mathcal{A}_5(U)\) we have \(x_\lambda x = 0\).  Computing in \(\mathcal{B}_2(f)\) gives
\[
\varphi(x)_\lambda\varphi(x) = P(-\lambda)P(\partial+\lambda)x' + P(-\lambda)Q(\partial+\lambda)f(\lambda)y' = 0,
\]
hence \(P(\partial) = 0\).  Thus \(\varphi(x) = Q(\partial)y'\) with \(Q(\partial) \neq 0\).

In \(\mathcal{A}_5(U)\) we have \(y_\lambda x = x\).  Computing \(\varphi(y_\lambda x)\) and \(\varphi(y)_\lambda\varphi(x)\):
\[
\varphi(y_\lambda x) =\varphi(x)=   Q(\partial)y',
\]
\[
\varphi(y)_\lambda\varphi(x) =   R(-\lambda)Q(\partial+\lambda)f(\lambda)y'.
\]
Comparing the coefficients of  \(y'\) yields
\begin{eqnarray}
 \label{134-thq}&&Q(\partial)=R(-\lambda)Q(\partial+\lambda)f(\lambda) .
\end{eqnarray}
Assume \(\deg Q(\partial) \ge 1\).  Since \(Q(\partial) \neq 0\), there exists \(\alpha \in \mathbb{C}\) such that \(Q(\alpha) = 0\).  Setting \(\partial = \alpha\) in Eq.~(\ref{134-thq})  we obtain
\[
0 = Q(\alpha) = R(-\lambda)\,Q(\alpha+\lambda)\,f(\lambda) \qquad \text{for all } \lambda \in \mathbb{C}.
\]
 \(R(-\lambda)\)  is a nonzero polynomial (otherwise the right-hand side of Eq.~(\ref{134-thq})  would be identically zero while the left-hand side is not), so the product \(R(-\lambda)f(\lambda)\) is a nonzero polynomial in \(\lambda\); it has only finitely many zeros.  Consequently there exists \(\lambda_0 \in \mathbb{C}\) such that \(R(-\lambda_0)f(\lambda_0) \neq 0\).  Substituting this value into the identity above gives
\[
0 = R(-\lambda_0)\,Q(\alpha+\lambda_0)\,f(\lambda_0) \;\Longrightarrow\; Q(\alpha+\lambda_0) = 0.
\]
Thus \(\alpha+\lambda_0\) is another root of \(Q(\partial)\).  Because there are infinitely many \(\lambda\) with \(R(-\lambda)f(\lambda)\neq0\), we obtain infinitely many distinct roots of \(Q(\partial)\), forcing \(Q(\partial) = 0\), which contradicts \(Q(\partial)\neq0\).  Hence \(\deg Q(\partial) = 0\), i.e. \(Q(\partial)=Q\in \mathbb{C}\setminus\{0\}\) is a nonzero constant.

Substituting the constant \(Q(\partial)=Q\) into Eq.~(\ref{134-thq}) and cancelling \(Q\) we get
\[
1 = R(-\lambda)\,f(\lambda).
\]
The left-hand side is the constant \(1\); therefore both \(R(-\lambda)\) and \(f(\lambda)\) must be nonzero constants.  Write \(R(\partial)=R\in \mathbb{C}\setminus\{0\}) and \(f(\lambda)=f_0\in \mathbb{C}\setminus\{0\}); then \(R f_0 = 1\).

In \(\mathcal{A}_5(U)\) we have \(y_\lambda y = U(\lambda)x + y\).  Applying \(\varphi\) to both sides gives
\[
\varphi(y_\lambda y)
 = R\,x' +  [U(\lambda)\,Q+S(\partial)]\,y'.
\]
On the other hand, compute \(\varphi(y)_\lambda\varphi(y)\) in \(\mathcal{B}_2(f)\) :
\[
\begin{aligned}
\varphi(y)_\lambda\varphi(y)
&= R^2\,x' + R\,S(\partial+\lambda)\,f_0\,y'.
\end{aligned}
\]
Comparing the coefficients of \(y'\) gives
\begin{align*}
     S(\partial+\lambda) = U(\lambda)\,Q + S(\partial).
\end{align*}
i.e.
\begin{align}
    \label{135-thq}S(\partial+\lambda) - S(\partial) = U(\lambda)\,Q.
\end{align}
Since $U(\lambda)Q \neq 0$, $S(\partial)$ is not a constant.  Eq.~(\ref{135-thq}) gives $\deg S(\partial) \le 1$, so $\deg S(\partial)=1$.
Write \(S(\partial) = \alpha\partial + \beta\).  Then \(S(\partial+\lambda)-S(\partial) = \alpha\lambda\), so \(\alpha\lambda = U(\lambda)Q\), i.e. \(U(\lambda) = \dfrac{\alpha}{Q}\,\lambda\), contradicting the condition \(U\not\sim\lambda\) in \(\mathcal{A}_5(U)\).
Consequently \(\mathcal{B}_2(f) \not\cong \mathcal{A}_5(U)\).  \\
\(\mathcal{A}_6(g) \not\cong \mathcal{A}_7(H)\) (\(\varphi: \mathcal{A}_6(g) \to \mathcal{A}_7(H))\):\\
From \(x_\lambda y = x\) we have
\[
\varphi(x) = \varphi(x)_\lambda \varphi(y) = \bigl[P(-\lambda)S(\partial+\lambda) + Q(-\lambda)R(\partial+\lambda) + Q(-\lambda)S(\partial+\lambda)H(\lambda)(\lambda+\partial)\bigr]x' + Q(-\lambda)S(\partial+\lambda)y'.
\]
Comparing this with \(\varphi(x)=P(\partial)x'+Q(\partial)y'\) gives the two equations
\begin{eqnarray}
 \label{136}&&Q(-\lambda)S(\partial+\lambda) = Q(\partial), \\
 \label{137}&&P(-\lambda)S(\partial+\lambda) + Q(-\lambda)R(\partial+\lambda) + Q(-\lambda)S(\partial+\lambda)H(\lambda)(\lambda+\partial) = P(\partial).
\end{eqnarray}

Suppose first that \(Q(\partial)=0\).  Then Eq.~(\ref{137}) reduces to \(P(-\lambda)S(\partial+\lambda)=P(\partial)\).  Setting \(\lambda=0\) yields \(P(0)S(\partial)=P(\partial)\).  Because \(P(\partial)\neq0\) (otherwise \(\varphi(x)=0\)), we must have \(P(0)\neq0\) and consequently \(S(\partial)=P(\partial)/P(0)\).  Substituting this back into \(P(-\lambda)S(\partial+\lambda)=P(\partial)\) gives \(P(-\lambda)P(\partial+\lambda)=P(0)P(\partial)\), which forces \(P(\partial)\) to be a nonzero constant \(P\) and thus \(S(\partial)=1\).
Suppose now that \(Q(\partial)\neq0\).  In Eq.~(\ref{136}) put \(\lambda=0\) to obtain \(Q(0)S(\partial)=Q(\partial)\).  Since \(Q(\partial)\neq0\), \(Q(0)\neq0\) and we can write \(S(\partial)=Q(\partial)/Q(0)\).  Substituting this expression into Eq.~(\ref{136}) gives \(Q(-\lambda)Q(\partial+\lambda)=Q(0)Q(\partial)\).  Taking \(\partial=0\) yields \(Q(-\lambda)Q(\lambda)=Q(0)^2\), which forces \(\deg Q(\partial)=0\); thus \(Q(\partial)\) is a nonzero constant, \(S(\partial)=1\).
Therefore, $Q(\partial)$ is a constant, denote it by $Q$, and $S(\partial)=1$.  Now \(\varphi(x)=P(\partial)x'+Qy'\) and \(\varphi(y)=R(\partial)x'+y'\) .  Applying \(y_\lambda x = g(\lambda)x\) we obtain
\[
\varphi(y_\lambda x) =g(\lambda)[P(\partial)x'+Qy']=\varphi(y)_\lambda\varphi(x)=[R(-\lambda)Q+P(\lambda+\partial)+QH(\lambda)(\lambda+\partial)]x'+Qy',
\]
which again forces \(g(\lambda)=1\), contradicting \(g(\lambda)\neq1\).

In the following, all contradictions are derived from the condition \(\Delta = P(\partial)S(\partial) - Q(\partial)R(\partial) \neq 0\).

$\mathcal{B}_2(f) \not\cong \mathcal{A}_1(B), \mathcal{A}_3(V), \mathcal{A}_4(f), \mathcal{A}_6(g), \mathcal{A}_7(H)$:

Let \(\varphi: A \to \mathcal{B}_2(f)\) where \(A\) is any of the A-class algebras \(\mathcal{A}_1(B)\), \(\mathcal{A}_3(V)\), \(\mathcal{A}_4(f)\), \(\mathcal{A}_6(g)\), \(\mathcal{A}_7(H)\).  All these algebras satisfy \(x_\lambda x = 0\).  In \(\mathcal{B}_2(f)\) we obtain
\[
0=\varphi(x)_\lambda\varphi(x)=P(-\lambda)P(\partial+\lambda)x'+P(-\lambda)Q(\partial+\lambda)f(\lambda) y'=0
\]
i.e.
$P(\partial)=0,\ \varphi(x)=Q(\partial)y',\ Q(\partial)\neq0.$

\begin{itemize}
\item \(A=\mathcal{A}_1\): $\varphi(y_\lambda y)=B(\lambda,\partial)Q(\partial)y'=
\varphi(y)_\lambda\varphi(y)=R(-\lambda)R(\partial+\lambda)x'+R(-\lambda)S(\partial+\lambda)f(\lambda) y'$ yields \(R(\partial)=0\); then \(Q(\partial)=0\), a contradiction.

\item \(A=\mathcal{A}_3(V),\mathcal{A}_6(g),\mathcal{A}_7(H)\):   \(\varphi(x_\lambda y)=\varphi(x)=Q(\partial)y'=\varphi(x)_\lambda\varphi(y)=0\) , a contradiction.

\item \(A=\mathcal{A}_4(f)\):   \(\varphi(x_\lambda y)=f(\lambda)[(R(\partial)x'+S(\partial)y']=\varphi(x)_\lambda\varphi(y)=0\), a contradiction.
 \end{itemize}

$\mathcal{A}_1(B) \not\cong \mathcal{A}_3(V), \mathcal{A}_4(f), \mathcal{A}_5(U), \mathcal{A}_6(g), \mathcal{A}_7(H)$. Here $\varphi$ is assumed to be an isomorphism from $\mathcal{A}_1(B)$ to each of the listed algebras.  Then from $x_\lambda y = 0$ we obtain a contradiction.
\begin{itemize}
\item \(\mathcal{A}_1(B)\) vs \(\mathcal{A}_3(V)\):  $0=\varphi(x_\lambda y)=\varphi(x)_\lambda\varphi(y)=[P(-\lambda)S(\lambda+\partial)+Q(-\lambda)S(\lambda+\partial)V(\partial)]x'+Q(-\lambda)S(\lambda+\partial)y'$.\\
$S(\partial)=0 , \varphi(y)_\lambda\varphi(y)=\varphi(y_\lambda y)=0 \;\Longrightarrow\; P(\partial)=Q(\partial)=0$, a contradiction. \\
$S(\partial)\neq0 \;\Longrightarrow\; Q(-\lambda)=0 \;\Longrightarrow\; P(\partial)=0$, a contradiction.

\item \(\mathcal{A}_1(B)\) vs \(\mathcal{A}_4(f)\):
 $0=\varphi(x_\lambda y)=\varphi(x)_\lambda\varphi(y)=P(-\lambda)S(\partial+\lambda)f (\lambda)y'$.\\
$S(\partial)=0 ,\varphi(y)_\lambda\varphi(y)=\varphi(y_\lambda y)=0 \;\Longrightarrow\; P(\partial)=Q(\partial)=0$, a contradiction. \\
$S(\partial)\neq0 \;\Longrightarrow\; P(-\lambda)=0 ,\varphi(y)_\lambda\varphi(x)=\varphi(y_\lambda x)=0\;\Longrightarrow\; R(\partial)=0$ or $Q(\partial)=0$, a contradiction.
\item \(\mathcal{A}_1(B)\) vs \(\mathcal{A}_5(U)\):
$0=\varphi(x_\lambda y)=\varphi(x)_\lambda\varphi(y)=[Q(-\lambda)R(\lambda+\partial)+Q(-\lambda)S(\lambda+\partial)U(\lambda)]x'+Q(-\lambda)S(\lambda+\partial)y'$.\\
$Q(\partial)\neq0 ,\Longrightarrow\; R(\partial)=0$ and $S(\partial)=0$, a contradiction. \\
$Q(\partial)=0,\varphi(y)_\lambda\varphi(x)=\varphi(y_\lambda x)=0 \;\Longrightarrow\; P(\partial)=0$ or $S(\partial)=0 $, a contradiction.

\item \(\mathcal{A}_1(B)\) vs \(\mathcal{A}_6(g)\):  $0=\varphi(x_\lambda y)=\varphi(x)_\lambda\varphi(y)=[P(-\lambda)S(\lambda+\partial)+Q(-\lambda)R(\lambda+\partial)g(\lambda)]x'+Q(-\lambda)S(\lambda+\partial)y'$.\\
$S(\partial)\neq0 ,\Longrightarrow\; Q(-\lambda)=0,\Longrightarrow\; P(\partial)=0$, a contradiction. \\
$S(\partial)=0,\varphi(y)_\lambda\varphi(y)=\varphi(y_\lambda y)=0 \;\Longrightarrow\; P(\partial)=0$ and $Q(\partial)=0 $, a contradiction.

\item \(\mathcal{A}_1(B)\) vs \(\mathcal{A}_7(H)\):  identical to \(\mathcal{A}_1(B)\) vs \(\mathcal{A}_6(g)\) with \(g(\lambda)=1\).
  \end{itemize}

\(\mathcal{A}_3(V) \not\cong \mathcal{A}_4(f), \mathcal{A}_5(U), \mathcal{A}_6(g), \mathcal{A}_7(H)\). Here \(\varphi\) is assumed to be an isomorphism from \(\mathcal{A}_3(V)\) to each of the algebras \(\mathcal{A}_4(f), \mathcal{A}_5(U), \mathcal{A}_6(g), \mathcal{A}_7(H)\).

\begin{itemize}
\item \(\mathcal{A}_3(V)\) vs \(\mathcal{A}_4(f)\):
  \(\varphi(x)=\varphi(x_\lambda y)=\varphi(x)_\lambda\varphi(y)=P(-\lambda)S(\partial+\lambda)f(\lambda)y'=P(\partial)x'+Q(\partial)y'\)\\
  \(\;\Longrightarrow\;  P(\partial)=0\) and $Q(\partial)=0$, a contradiction.

\item \(\mathcal{A}_3(V)\) vs \(\mathcal{A}_5(U)\):
 \(\varphi(x)=\varphi(x_\lambda y)=\varphi(x)_\lambda\varphi(y)=[Q(-\lambda)R(\partial+\lambda)+Q(-\lambda)S(\partial+\lambda)U(\lambda)]x'+Q(-\lambda)S(\partial+\lambda)y'=P(\partial)x'+Q(\partial)y'\). \\
 \(Q(\partial)=0 \;\Longrightarrow\; P(\partial)=0\), a contradiction. \\
 \(Q(\partial)\neq0 \;\Longrightarrow\; S(\partial)=1, Q(\partial)=Q\in \mathbb{C}\setminus\{0\}).
 \(0=\varphi(y_\lambda  x)=\varphi(y)_\lambda\varphi(x)=[P(\partial+\lambda)+QU(\lambda)]x'+Qy'\) \(\;\Longrightarrow\; Q(\partial)=0\), a contradiction.

\item \(\mathcal{A}_3(V)\) vs \(\mathcal{A}_6(g)\):
  \(0=\varphi(y_\lambda x)=\varphi(y)_\lambda\varphi(x)=[R(-\lambda)Q(\partial+\lambda)+S(-\lambda)P(\partial+\lambda)g(\lambda)]x'+S(-\lambda)Q(\partial+\lambda)y'\).\\
 \(S(\partial)=0 \;\Longrightarrow\; \varphi(y)_\lambda\varphi(y)=0\):
\(V(\partial)\neq0 \;\Rightarrow\; Q(\partial)=0\);
\(V(\partial)=0 \;\Rightarrow\; R(\partial)=0\), a contradiction.\\
\(S(\partial)\neq0 \;\Longrightarrow\; Q(\partial)=0,\ \varphi(x_\lambda y)=\varphi(x)=\varphi(x)_\lambda\varphi(y) \;\Longrightarrow\; P(\partial)=P\in \mathbb{C}^\times, S(\partial)=1, g(\lambda)=0\), a contradiction.

\item \(\mathcal{A}_3(V)\) vs \(\mathcal{A}_7(H)\):  identical to the \(\mathcal{A}_3(V)\) vs \(\mathcal{A}_6(g)\) with \(g(\lambda)=1\).
\end{itemize}

\(\mathcal{A}_4(f) \not\cong \mathcal{A}_5(U), \mathcal{A}_6(g), \mathcal{A}_7(H)\).  Here \(\varphi\) is assumed to be an isomorphism from \(\mathcal{A}_4(f)\) to each of the listed algebras.

\begin{itemize}
\item \(\mathcal{A}_4(f)\) vs \(\mathcal{A}_5(U)\), \(\mathcal{A}_6(g),\mathcal{A}_7(H)\):
  \(0=\varphi(y_\lambda y)=\varphi(y)_\lambda\varphi(y)\) ,\(0=\varphi(x_\lambda x)=\varphi(x)_\lambda\varphi(x)\) \(\;\Longrightarrow\; S(\partial)=Q(\partial)=0\), a contradiction.
\end{itemize}

\(\mathcal{A}_5(U) \not\cong \mathcal{A}_6(g), \mathcal{A}_7(H)\). Here \(\varphi\) is assumed to be an isomorphism from \(\mathcal{A}_5(U)\) to each of the listed algebras.

\begin{itemize}
\item \(\mathcal{A}_5(U)\) vs \(\mathcal{A}_6(g)\):
  \(0=\varphi(x_\lambda y)=\varphi(x)_\lambda\varphi(y)=[P(-\lambda)S(\partial+\lambda)+Q(-\lambda)R(\partial+\lambda)g(\lambda)]x'+Q(-\lambda)S(\partial+\lambda)y'\).\\
 \(S(\partial)=0 ,0=\varphi(y_\lambda y)=\varphi(y)_\lambda\varphi(y)=[U(\lambda)P(\partial)+R(\partial)]x'+U(\lambda)Q(\partial)y' \;\Longrightarrow\; Q(\partial)=0\), a contradiction. \\
\(S(\partial)\neq0 \;\Longrightarrow\; Q(\partial)=0 \;\Longrightarrow\; P(\partial)=0\), a contradiction.
\item \(\mathcal{A}_5(U)\) vs \(\mathcal{A}_7(H)\):  identical to the  \(\mathcal{A}_5(U)\) vs \(\mathcal{A}_6(g)\) with \(g(\lambda)=1\).
\end{itemize}

Thus all single?ideal classes are pairwise non?isomorphic.

\end{proof}}}
\begin{rmk}
Since $(I_2=\mathbb{C}e_1\oplus \mathbb{C}e_2, \cdot)$ with $e_1\cdot e_1=2e_1$, $e_1\cdot e_2=e_2$, $e_2\cdot e_1=0$ and $e_2\cdot e_2=e_1$ is a simple left-symmetric algebra \cite[Corollary 4.1]{B7}, the change of basis
\[
e_1=2X,\qquad e_2=\sqrt{2}\,Y,
\]
gives an isomorphism $\mathcal{S}\cong \text{Cur}(I_2)$, where $\text{Cur}(I_2)$ is the current left-symmetric conformal algebra associated with $(I_2, \cdot)$ (see \cite{HL1}).
\end{rmk}
Next,
we will prove that the left-symmetric conformal algebras
$\mathcal{A}_1(B(\lambda,\partial))$-$\mathcal{B}_2(f(\lambda))$
are mutually non-isomorphic.

First, we use their sub-adjacent Lie conformal algebras to make a rough classification. Let $(R_1, \cdot_\lambda \cdot)$, $(R_2, \cdot_\lambda \cdot)$  be two left-symmetric conformal algebras and $\mathfrak{g}(R_1)$, $\mathfrak{g}(R_2)$  be their sub-adjacent Lie conformal algebras.
If $\varphi:(R_1, \cdot_\lambda \cdot)\to (R_2, \cdot_\lambda \cdot)$ is an isomorphism of left-symmetric conformal algebras, then
\begin{eqnarray*}
\varphi([a_\lambda b])
= \varphi(a_\lambda b - b_{-\lambda-\partial}a)
= \varphi(a)_\lambda \varphi(b) - \varphi(b)_{-\lambda-\partial}\varphi(a)
= [\varphi(a)_\lambda \varphi(b)],\;\;a, b\in R_1.
\end{eqnarray*}
Thus $\varphi: \mathfrak{g}(R_1)\rightarrow \mathfrak{g}(R_2)$ is  an isomorphism of Lie conformal algebras.
Therefore, if $\mathfrak{g}(R_1) \not\cong \mathfrak{g}(R_2)$, then $(R_1, \cdot_\lambda \cdot) \not\cong (R_2, \cdot_\lambda \cdot)$.
By a direct computation, the sub-adjacent  Lie conformal algebras of these left-symmetric conformal algebras are as follows.
\vspace{1mm}
\setlength{\LTpre}{0pt}
\setlength{\LTpost}{0pt}

\begingroup
\centering
\small
\begin{longtable}{|c|>{\centering\arraybackslash}m{5.2cm}|c|}
\hline
Left-symmetric conformal algebra & sub-adjacent Lie conformal algebra & Standard basis \\
\hline
\endfirsthead
\multicolumn{3}{|r|}{Continued.} \\
\hline
Left-symmetric algebra & Induced Lie conformal algebra & Standard basis \\
\hline
\endhead
\hline
\endfoot
\hline
\endlastfoot

$\mathcal{A}_1(B(\lambda,\partial))$ & $R_{\mathrm{nil}}(Q(\lambda,\partial))$, $Q(\lambda,\partial)=B(\lambda,\partial)-B(-\lambda-\partial,\partial)$ & $A=X,\ B=Y$ \\
\hline
$\mathcal{A}_2$ & abelian & $A=X,\ B=Y$ \\
\hline
$\mathcal{A}_7(H(\lambda))$ & $R_{\mathrm{nil}}(Q(\lambda,\partial))$, $Q(\lambda,\partial)=H(\lambda)(\lambda+\partial)+\lambda H(-\lambda-\partial)$ & $A=X,\ B=Y$ \\
\hline
$\mathcal{B}_1$ & abelian & $A=X,\ B=Y$ \\
\hline
$\mathcal{A}_3(V(\partial))$ & $R_{\mathrm{sol}}(-1)$ & $A=X,\ B=Y$ \\
\hline
$\mathcal{A}_4(f(\lambda))$ & $R_{\mathrm{sol}}(f(\lambda))$ & $A=Y,\ B=X$ \\
\hline
$\mathcal{A}_5(U(\lambda))$ & $R_{\mathrm{sol}}(1)$ & $A=X,\ B=Y+U(-\partial)X$ \\
\hline
$\mathcal{A}_6(g(\lambda))$ & $R_{\mathrm{sol}}(g(\lambda)-1)$ & $A=X,\ B=Y$ \\
\hline
$\mathcal{B}_2(f(\lambda))$ & $R_{\mathrm{sol}}(f(\lambda))$ & $A=Y,\ B=X$ \\
\hline

\end{longtable}
\endgroup
Here, we adopt the notations $R_{\mathrm{nil}}(Q(\lambda,\partial))$ and  $R_{\mathrm{sol}}(f(\lambda))$ which can be found in \cite[Theorem 2.21]{BCH}.
Note that the nilpotent Lie conformal algebras and the solvable but not nilpotent Lie conformal algebras are not isomorphic to each other. Therefore, we have the following result.
\begin{pro}\label{pro-p2}
$\mathcal{A}_1(B(\lambda,\partial))$, $\mathcal{A}_2$, $\mathcal{A}_7(H(\lambda))$ and $\mathcal{B}_1$ are not isomorphic to $\mathcal{A}_3(V(\partial))$, $\mathcal{A}_4(f(\lambda))$, $\mathcal{A}_5(U(\lambda))$, $\mathcal{A}_6(g(\lambda))$ and $\mathcal{B}_2(f(\lambda))$.
\end{pro}
\delete{\begin{lem}
    $\mathcal{A}_1(B)$ and $\mathcal{A}_7(H)$ in Theorem~3.13 are not isomorphic.
\end{lem}
\begin{proof}
Here we adopt the same notation as in Lemma~3.6, and \(\phi\) is an isomorphism from \(\mathcal{A}_1(B)\) to \(\mathcal{A}_7(H)\).
    $0=\varphi(x_\lambda y)=\varphi(x)_\lambda\varphi(y)
=\bigl[P(-\lambda)S(\lambda+\partial)+Q(-\lambda)R(\lambda+\partial)
+Q(-\lambda)S(\lambda+\partial)H(\lambda)(\lambda+\partial)\bigr]x'
+Q(-\lambda)S(\lambda+\partial)y'$.\\
$S(\partial)\neq0 \implies Q(-\lambda)=0 \implies P(\partial)=0$, a contradiction.\\
$S(\partial)=0$,  $\varphi(y)_\lambda\varphi(y)=0=\varphi(y_\lambda y)=B(\lambda,\partial)(P(\partial)x'+Q(\partial)y')\implies P(\partial)=Q(\partial)=0$, a contradiction.
\end{proof}
We now introduce some isomorphism invariants to prove that the algebras
\(\mathcal{A}_3(V)\), \(\mathcal{A}_4(f)\), \(\mathcal{A}_5(U)\), \(\mathcal{A}_6(g)\), and \(\mathcal{B}_2(f)\)
are mutually non-isomorphic. Throughout this subsection, \(L\) denotes a two-dimensional left-symmetric conformal algebra
with a unique nonzero proper ideal \(I\), and \(I\) is a rank-one free \(\mathbb C[\partial]\)-submodule. We fix an ideal generator \(e\) and a complement-space generator \(q\) such that
\(L=\mathbb C[\partial]q\oplus I\). More precisely, for \(\mathcal{A}_3(V),\mathcal{A}_5(U),\mathcal{A}_6(g)\) we take
\((I,e,q)=(\mathbb C[\partial]x,\ x,\ y)\), and for \(\mathcal{A}_4(f),\mathcal{B}_2(f)\) we take
\((I,e,q)=(\mathbb C[\partial]y,\ y,\ x)\).

For each such \(L\), the quotient \(Q=L/I\) is a free \(\mathbb C[\partial]\)-module of rank one,
and the class \(\bar{q}=q+I\) is a basis of \(Q\).
The \(\lambda\)-product on \(Q\) is defined by $\bar{u}_\lambda\bar{v}=\overline{u_\lambda v}, u,v\in L.$
This definition is independent of the choice of representatives, and hence is well defined.
With this product, \(Q\) becomes a one-dimensional left-symmetric conformal algebra,
since the conformal sesquilinearity and left-symmetry identities are inherited from \(L\).
Moreover, any two bases of \(Q\) differ by a nonzero constant.

\begin{defi}
Let \(Q=L/I\) be as above. Define the \emph{quotient type} \(\tau(L)\) by
\[
\tau(L)=
\begin{cases}
1, & \text{if there exists a basis } \bar{v} \text{ of } Q \text{ such that } \bar{v}_\lambda\bar{v}=\bar{v},\\
0, & \text{otherwise}.
\end{cases}
\]
If such a basis exists, we call it a \emph{quotient idempotent basis}
\end{defi}
\begin{rmk}
Assume \(\tau(L)=1\), we show that the quotient idempotent basis is unique. Let \(\bar{v}\) and \(\bar{w}\) be two idempotent bases of \(Q\), so there exists \(c\in \mathbb{C}\setminus\{0\}) such that \(\bar{w}=c\bar{v}\). Then
\[
\bar{w}_\lambda\bar{w}=(c\bar{v})_\lambda(c\bar{v})
=c^2(\bar{v}_\lambda\bar{v})
=c^2\bar{v}
=c\bar{w}=\bar{w}.
\]
We obtain \(c=1\). Therefore \(\bar{w}=\bar{v}\), and the quotient idempotent basis is unique.
\end{rmk}
\begin{pro}
Let \(\bar{w}\) be any basis of \(Q\). Since \(Q\) is a one-dimensional left-symmetric conformal algebra, there exists a polynomial \(\alpha(\lambda,\partial)\in\mathbb{C}[\lambda,\partial]\) such that $\bar{w}_\lambda\bar{w}=\alpha(\lambda,\partial)\bar{w}.$
We claim that \(\tau(L)=1\) if and only if \(\alpha(\lambda,\partial)\) is a nonzero constant.
\end{pro}
\begin{proof}
If \(\alpha(\lambda,\partial)=c\in \mathbb{C}\setminus\{0\}\), then taking \(\bar{v}=c^{-1}\bar{w}\) we obtain
\[
\bar{v}_\lambda\bar{v}=c^{-2}(\bar{w}_\lambda\bar{w})=c^{-2}(c\bar{w})=c^{-1}\bar{w}=\bar{v},
\]
so  \(\tau(L)=1\).
Conversely, suppose there exists an idempotent basis \(\bar{v}\). Since any two bases differ by a nonzero scalar, write \(\bar{v}=k\bar{w}\) for some \(k\in \mathbb{C}\setminus\{0\}\). Then
\[
\bar{v}_\lambda\bar{v}=k^2(\bar{w}_\lambda\bar{w})
=k^2\alpha(\lambda,\partial)\bar{w}
=k\alpha(\lambda,\partial)\bar{v}.
\]
The idempotence condition gives \(k\alpha(\lambda,\partial)=1\), hence \(\alpha(\lambda,\partial)=k^{-1}\in \mathbb{C}\setminus\{0\}\).
Finally, if another basis \(\bar{w}'=d\bar{w}(d\in \mathbb{C}\setminus\{0\})\) is chosen, then the new coefficient \(\alpha'(\lambda,\partial)\) satisfies
\[
\bar{w}'_\lambda\bar{w}'=d^2\alpha(\lambda,\partial)\bar{w}
=d\alpha(\lambda,\partial)\bar{w}',
\]
so \(\alpha'(\lambda,\partial)=d\alpha(\lambda,\partial)\). Thus \(\alpha'(\lambda,\partial)\) is a nonzero constant if and only if \(\alpha(\lambda,\partial)\) is a nonzero constant.
\end{proof}
\begin{defi}
Let \(e\) be a generator of the ideal \(I\) and \(q\) a generator of the complementary subspace, so that \(L=\mathbb{C}[\partial]q\oplus I\). Define
\[
\ell(L)=
\begin{cases}
0, & \text{if } q_\lambda e=0,\\
1, & \text{if } q_\lambda e\neq 0,
\end{cases}
\qquad
r(L)=
\begin{cases}
0, & \text{if } e_\lambda q=0,\\
1, & \text{if } e_\lambda q\neq 0.
\end{cases}
\]
\end{defi}
\begin{rmk}
The definitions of \(\ell(L)\) and \(r(L)\) are independent of the choices of \(e\) and \(q\). Indeed, a general generator of the complementary subspace has the form $q'=c\,q+P(\partial)e, c\in \mathbb{C}\setminus\{0\},\ P\in\mathbb{C}[\partial],$
and a general generator of the ideal has the form $e'=d\,e, d\in \mathbb{C}\setminus\{0\}.$
Using \(e_\lambda e=0\) and conformal sesquilinearity, we obtain
\[
q'_\lambda e'
=(c\,q+P(\partial)e)_\lambda(d\,e)
=c d\,(q_\lambda e),
\qquad
e'_\lambda q'
=(d\,e)_\lambda(c\,q+P(\partial)e)
=c d\,(e_\lambda q).
\]
Since \(cd\neq 0\), hence \(\ell(L)\) and \(r(L)\) do not depend on the choice of \(e\) and \(q\), so they are well-defined.
\end{rmk}
\begin{defi}
Assume \(\tau(L)=1\), and let \(\bar{q}\) be the unique idempotent basis of \(Q\).
Let \(\pi:L\to Q\) be the natural projection. Denote by
$S_{\bar{q}}=\pi^{-1}(\bar{q})$ the set of all lifts of \(\bar{q}\) in \(L\).
Define
\[
\operatorname{Split}(L)=
\begin{cases}
1, & \text{if there exists } s\in S_{\bar{q}} \text{ such that } s_\lambda s=s,\\
0, & \text{otherwise}.
\end{cases}
\]
\end{defi}
\begin{rmk}
The set \(S_{\bar{q}}\) is independent of the choice of representative \(q\). If \(q'\) is another element with \(\bar{q'}=\bar{q}\), then \(q-q'\in I=\mathbb{C}[\partial]e\), and
\[
S_{\bar{q}}=\{q+P(\partial)e\mid P(\partial)\in\mathbb{C}[\partial]\}
=\{q'+P'(\partial)e\mid P'(\partial)\in\mathbb{C}[\partial]\}=S_{\bar{q'}}.
\]
Hence \(\operatorname{Split}(L)\) is well-defined.
\end{rmk}
\begin{pro}
Let \(L\) and \(L'\) be two-dimensional left-symmetric conformal algebras, each having a unique nonzero proper ideal, denoted by \(I\) and \(I'\), respectively. If \(\varphi:L\to L'\) is an isomorphism of left-symmetric conformal algebras, then
\[
\tau(L)=\tau(L'),\quad \ell(L)=\ell(L'),\quad r(L)=r(L'),\quad \operatorname{Split}(L)=\operatorname{Split}(L').
\]
\end{pro}
\begin{proof}
Since \(\varphi\) maps ideals to ideals and the ideals are unique, we have \(\varphi(I)=I'\). Let \(e\) be a generator of \(I\), and let \(q\) be a complementary generator of \(L\), so that \(L=\mathbb{C}[\partial]q\oplus I\). Let \(e'\) be a generator of \(I'\), and there exists \(c\in \mathbb{C}\setminus\{0\}\) such that $\varphi(e)=c e'.$
Since \(\{q,e\}\) is a basis of \(L\), the set \(\{\varphi(q),\varphi(e)\}\) is a basis of \(L'\). Hence $L'=\mathbb{C}[\partial]\varphi(q)\oplus I',$
so \(\varphi(q)\) is a complementary generator of \(L'\) with respect to \(I'\).

Define the quotient map $\bar\varphi: Q=L/I \longrightarrow Q'=L'/I',
\bar\varphi(\bar u)=\overline{\varphi(u)}.$
This map is well-defined: if \(\bar u=\bar v\), then \(u-v\in I\), hence \(\varphi(u)-\varphi(v)\in I'\), so \(\overline{\varphi(u)}=\overline{\varphi(v)}\). Moreover, \(\bar\varphi\) is a \(\mathbb{C}[\partial]\)-module isomorphism and preserves the \(\lambda\)-product:
\[
\bar\varphi(\bar u_\lambda\bar v)
=\overline{\varphi(u_\lambda v)}
=\overline{\varphi(u)_\lambda\varphi(v)}
=\bar\varphi(\bar u)_\lambda\bar\varphi(\bar v).
\]
Thus \(\bar\varphi\) is an isomorphism of quotient algebras.

\textbf{(1) Invariance of \(\tau\).}
Assume \(\tau(L)=1\). Then \(Q\) has an idempotent basis \(\bar v\). Since \(\bar\varphi\) is an isomorphism of quotient algebras, \(\bar v'=\bar\varphi(\bar v)\) is a basis of \(Q'\), and
\[
\bar v'_\lambda\bar v'
=\bar\varphi(\bar v_\lambda\bar v)
=\bar\varphi(\bar v)
=\bar v',
\]
so \(Q'\) also has an idempotent basis; hence \(\tau(L')=1\). Conversely, if \(\tau(L')=1\), applying \(\bar\varphi^{-1}\) yields \(\tau(L)=1\). Therefore \(\tau(L)=\tau(L')\).

\textbf{(2) Invariance of \(\ell\) and \(r\).}
We compute \(\ell(L')\) and \(r(L')\) using the complementary generator \(\varphi(q)\) and the ideal generator \(e'\). For the left primitive product,
\[
\varphi(q_\lambda e)
=\varphi(q)_\lambda\varphi(e)
=\varphi(q)_\lambda(c e')
=c\bigl(\varphi(q)_\lambda e'\bigr).
\]
Since \(c\neq0\) and \(\varphi\) is bijective,
\[
q_\lambda e=0 \iff \varphi(q)_\lambda e'=0.
\]
By the well-definedness of \(\ell\), the vanishing of \(\varphi(q)_\lambda e'\) is exactly what determines \(\ell(L')\). Hence \(\ell(L)=\ell(L')\).
Similarly,\(r(L)=r(L')\).

\textbf{(3) Invariance of \(\operatorname{Split}\).}
Assume \(\tau(L)=1\); then \(\tau(L')=1\) by (1). Let \(\bar q\) and \(\bar q'\) be the unique idempotent bases of \(Q\) and \(Q'\), respectively. Since \(\bar\varphi\) is an isomorphism of quotient algebras and idempotent bases are unique, we have \(\bar q'=\bar\varphi(\bar q)\).

If \(\operatorname{Split}(L)=1\), there exists \(s\in S_{\bar q}\) such that \(s_\lambda s=s\). Then $\pi'(\varphi(s))=\overline{\varphi(s)}=\bar\varphi(\pi(s))=\bar\varphi(\bar q)=\bar q',$
so \(\varphi(s)\in S_{\bar q'}\), and $\varphi(s)_\lambda\varphi(s)=\varphi(s_\lambda s)=\varphi(s).$
Thus \(\varphi(s)\) is an idempotent element in \(S_{\bar q'}\), and hence \(\operatorname{Split}(L')=1\). Conversely, if \(\operatorname{Split}(L')=1\), applying \(\varphi^{-1}\) gives \(\operatorname{Split}(L)=1\). Therefore \(\operatorname{Split}(L)=1 \iff \operatorname{Split}(L')=1\), so \(\operatorname{Split}\) is invariant.

Combining the above arguments, all four invariants are invariant under isomorphisms of left-symmetric conformal algebras.
\end{proof}
\begin{thm}
The left-symmetric conformal algebras $\mathcal{A}_3(V),\mathcal{A}_4(f), \mathcal{A}_5(U),\mathcal{A}_6(g),\mathcal{B}_2(f)$
are mutually non-isomorphic.
\end{thm}
\begin{proof}
    The values of the isomorphism invariants are given in the following table:
\setlength{\LTpre}{0pt}
\setlength{\LTpost}{0pt}

\begingroup
\centering
\begin{longtable}{|c|c|c|c|c|}
\hline
Algebra & $\tau$ & $\ell$ & $r$ & $\operatorname{Split}$\\
\hline
\endfirsthead
\multicolumn{5}{|r|}{Continued.}\\
\hline
Algebra & $\tau$ & $\ell$ & $r$ & $\operatorname{Split}$\\
\hline
\endhead
\hline
\endfoot
\hline
\endlastfoot
$\mathcal{A}_3(V)$ & $1$ & $0$ & $1$ & $0\ (V(\partial)\neq0),\ 1\ (V(\partial)=0)$\\
\hline
$\mathcal{A}_4(f)$ & $0$ & $1$ & $0$ & $-$\\
\hline
$\mathcal{A}_5(U)$ & $1$ & $1$ & $0$ & $0$\\
\hline
$\mathcal{A}_6(g)$ & $1$ & $1$ & $1$ & $1$\\
\hline
$\mathcal{B}_2(f)$ & $1$ & $1$ & $0$ & $1$\\
\hline
\end{longtable}
\endgroup
Hence the five algebras are pairwise non-isomorphic.

We present a simple computation showing that \(\operatorname{Split}(\mathcal{A}_5(U))=0\).
In \(\mathcal{A}_5(U)\), the ideal generator is \(e=x\) and the complementary generator is \(q=y\). The quotient idempotent basis is \(\bar{y}=y+I\). Every lift has the form $s=y+P(\partial)x, P(\partial)\in\mathbb{C}[\partial].$
we compute $s_\lambda s=\bigl(U(\lambda)+P(\lambda+\partial)\bigr)x+y.$
If \(s_\lambda s=s\), comparing the coefficient of \(x\) gives $U(\lambda)=P(\partial)-P(\partial+\lambda).$
The left-hand side is independent of \(\partial\). Since \(U(\lambda)\neq 0\), the right-hand side must be a nonzero polynomial in \(\lambda\) independent of \(\partial\). This forces \(\deg P(\partial)\le 1\). Indeed, if \(\deg P(\partial)\ge 2\), then \(P(\partial)-P(\partial+\lambda)\) would contain positive powers of \(\partial\), contradicting the independence of \(\partial\). Thus \(P(t)=at+b\) for some \(a,b\in\mathbb{C}\). Substitution gives
$U(\lambda)=P(\partial)-P(\partial+\lambda)=-a\lambda$, contradicting the assumptions \(U(\lambda)\not\sim\lambda\) and $U(\lambda)\neq0$. Therefore no such lift exists, and \(\operatorname{Split}(\mathcal{A}_5(U))=0\).
\end{proof}}

\begin{defi}
Let $(R, \cdot_\lambda \cdot)$ be a left-symmetric conformal algebra. A  nonzero element $a\in R$ is called a {\bf square-zero element}, if $a_\lambda a=0$. A nonzero element $a\in R$ is called a {\bf left annihilator} ({\bf right annihilator}) , if $a_\lambda b=0$ for all $b\in R$ ($b_\lambda a=0$ for all $b\in R$).
A nonzero element $a\in R$ is called an {\bf idempotent}, if $a_\lambda a=a$.
\end{defi}
\begin{rmk}
It is easy to see that the existence of each of the above four types of elements is preserved under isomorphisms of left-symmetric conformal algebras. Therefore,  the existence of each of the above four types of elements are isomorphism invariants.
\end{rmk}
\begin{pro}\label{pp4.3}
    The left-symmetric conformal algebras \(\mathcal{A}_1(B(\lambda,\partial))\), \(\mathcal{A}_2\), \(\mathcal{A}_7(H(\lambda))\), and \(\mathcal{B}_1\) are mutually non-isomorphic.
\end{pro}
\begin{proof}
\delete{We first introduce some isomorphism invariants for left-symmetric conformal algebras. Let \(L\) be a left-symmetric conformal algebra.
Nonzero square-zero element: there exists \(0\neq u\in L\) such that \(u_\lambda u=0\).\\Nonzero left annihilator: there exists \(0\neq u\in L\) such that \(u_\lambda v=0\) for all \(v\in L\).\\
Nonzero idempotent: there exists \(0\neq u\in L\) such that \(u_\lambda u=u\).\\
If \(\varphi:L\to M\) is an isomorphism of left-symmetric conformal algebras, then \(\varphi(u_\lambda v)=\varphi(u)_\lambda\varphi(v)\) for all \(u,v\in L\). It follows directly that the existence of each of the above three types of elements is preserved under isomorphism. Hence they are isomorphism invariants.}
A direct check of the existence of nonzero square-zero elements, nonzero left annihilators, and nonzero idempotents gives the following table.
\begin{eqnarray*}
\begin{array}{|c|c|c|c|}
\hline
\text{Algebra} & \text{nonzero square-zero} & \text{nonzero left annihilator} & \text{nonzero idempotent}\\
\hline
\mathcal{A}_1(B(\lambda,\partial)) & \text{yes} & \text{yes} & \text{no}\\
\hline
\mathcal{A}_2 & \text{yes} & \text{yes} & \text{yes}\\
\hline
\mathcal{A}_7(H(\lambda)) & \text{yes} & \text{no} & \text{yes iff } H(\lambda)=c\lambda ~~\text{for some $c\in \mathbb{C}$}\\
\hline
\mathcal{B}_1 & \text{no} & \text{no} & \text{yes}\\
\hline
\end{array}
\end{eqnarray*}
Since the triples of invariants are mutually different,  the four left-symmetric conformal algebras are pairwise non-isomorphic.
We illustrate the computation by determining the existence of a nonzero idempotent for $\mathcal{A}_7(H(\lambda))$. The remaining computations are similar.
Let $u=P(\partial)X+Q(\partial)Y\neq 0$ satisfy $u_\lambda u=u$. Expanding in $\mathcal{A}_7(H(\lambda))$ and comparing the coefficients of $Y$ yields $Q(-\lambda)Q(\partial+\lambda)=Q(\partial)$. Therefore, $Q(\partial)$ is a constant. If $Q(\partial)=0$, then $u=P(\partial)X$. Since $X_\lambda X=0$, $u_\lambda u=0\neq u$, a contradiction. Hence $Q(\partial)\neq0$ i.e. $Q(\partial)=1$.
Comparing the coefficients of $X$ gives
\begin{align}\label{72-}
    P(\partial)-P(\partial+\lambda)-P(-\lambda)=H(\lambda)(\lambda+\partial).
\end{align}
Write $P(\partial)=\sum_{k=0}^n p_k \partial^k$. If $n\ge3$, the left-hand side of Eq.~(\ref{72-}) has degree $n-1$ in $\partial$, while the right-hand side has degree at most $1$, a contradiction. So $n\le2$. Checking $n=0,1,2$ shows that Eq.~(\ref{72-}) has a solution if and only if $H(\lambda)=c\lambda$ for some $c\in\mathbb{C}$, in which case $u=-\frac{c}{2}\partial^2X+Y$ is a nonzero idempotent.
Hence $\mathcal{A}_7(H(\lambda))$ contains a nonzero idempotent if and only if $H(\lambda)=c\lambda$ for some $c\in\mathbb{C}$.
\end{proof}
\begin{pro}\label{pp4.4}
    The left-symmetric conformal algebras \(\mathcal{A}_3(V(\partial))\), \(\mathcal{A}_4(f(\lambda))\), \(\mathcal{A}_5(U(\lambda))\), \(\mathcal{A}_6(g(\lambda))\), and \(\mathcal{B}_2(f(\lambda))\) are mutually non-isomorphic.
\end{pro}
\begin{proof}
A direct check of the existence of nonzero left annihilators, nonzero right annihilators and nonzero idempotents of these left-symmetric conformal algebras gives the following table.
\begin{center}
\begin{tabular}{|c|c|c|c|}
\hline
Algebra & nonzero left annihilator & nonzero right annihilator & nonzero idempotent \\
\hline
\(\mathcal{A}_3(V(\partial))\) & no & yes & yes iff \(V(\partial)=0\) \\
\hline
\(\mathcal{A}_4(f(\lambda))\) & yes & yes & no \\
\hline
\(\mathcal{A}_5(U(\lambda))\) & yes & no & no \\
\hline
\(\mathcal{A}_6(g(\lambda))\) & no & no & yes \\
\hline
\(\mathcal{B}_2(f(\lambda))\) & yes & no & yes \\
\hline
\end{tabular}
\end{center}
Since the triples of invariants are mutually different, the five left-symmetric conformal algebras are pairwise non-isomorphic.

\delete{We illustrate the computation by showing that $\mathcal{A}_5(U(\lambda))$ has no nonzero idempotents. Let $u=P(\partial)X+Q(\partial)Y\neq0$ satisfy $u_\lambda u=u$.
Then
\[
\begin{aligned}
u_\lambda u
&=(P(\partial)x+Q(\partial)y)_\lambda(P(\partial)x+Q(\partial)y)\\
&=\bigl[Q(-\lambda)P(\partial+\lambda)+Q(-\lambda)Q(\partial+\lambda)U(\lambda)\bigr]x+Q(-\lambda)Q(\partial+\lambda)y.
\end{aligned}
\]
Comparing this with \(u=P(\partial)x+Q(\partial)y\), the coefficient of \(y\) gives $Q(-\lambda)Q(\partial+\lambda)=Q(\partial).$
By an argument analogous to the proof of Proposition ~\ref{pp4.3}, we obtain \(Q(\partial)=1\).

Now comparing the coefficients of \(x\) yields $P(\partial)-P(\partial+\lambda)=U(\lambda)$.
Write \(P(t)=\sum_{k=0}^n p_k t^k\), then we have $n\leq1$ and $-p_1\lambda=U(\lambda).$
This contradicts the hypothesis \(U(\lambda)\not\sim\lambda\). Therefore \(\mathcal{A}_5(U(\lambda))\) has no nonzero idempotent.}
\end{proof}

\begin{thm}\label{pp4.5}
All compatible left-symmetric conformal algebra structures on solvable Lie
conformal algebras which are free of rank two as $\mathbb{C}[\partial]$-modules
are isomorphic to one of the following ten types of normal forms:
\[
\mathcal{A}_1(B(\lambda,\partial)),\quad
\mathcal{A}_2,\quad
\mathcal{A}_3(V(\partial)),\quad
\mathcal{A}_4(f(\lambda)),\quad
\mathcal{A}_5(U(\lambda)),
\]
\[
\mathcal{A}_6(g(\lambda)),\quad
\mathcal{A}_7(H(\lambda)),\quad
\mathcal{B}_1,\quad
\mathcal{B}_2(f(\lambda)),\quad
\mathcal{S}.
\]
Normal forms belonging to two different types are not isomorphic. Within each
parametrized type, the isomorphism relation is given by
\begin{eqnarray*}
\begin{array}{|c|c|}
\hline
\text{Algebra type} & \text{Isomorphism condition} \\
\hline
\mathcal{A}_1(B_1(\lambda,\partial)) \cong \mathcal{A}_1(B_2(\lambda,\partial)) & B_1(\lambda,\partial)=c\,B_2(\lambda,\partial),\quad c\in \mathbb{C}\setminus\{0\}\\
\hline
\mathcal{A}_3(V_1(\partial)) \cong \mathcal{A}_3(V_2(\partial)) & V_1(\partial)=c\,V_2(\partial),\quad c\in \mathbb{C}\setminus\{0\}\\
\hline
\mathcal{A}_4(f_1(\lambda)) \cong \mathcal{A}_4(f_2(\lambda)) & f_1(\lambda)=c\,f_2(\lambda),\quad c\in \mathbb{C}\setminus\{0\}\\
\hline
\mathcal{A}_5(U_1(\lambda)) \cong \mathcal{A}_5(U_2(\lambda)) & U_1(\lambda)=a\,U_2(\lambda)+c\,\lambda,\quad a\in \mathbb{C}\setminus\{0\},\ c\in\mathbb{C} \\
\hline
\mathcal{A}_6(g_1(\lambda)) \cong \mathcal{A}_6(g_2(\lambda)) & g_1(\lambda)=g_2(\lambda) \\
\hline
\mathcal{A}_7(H_1(\lambda)) \cong \mathcal{A}_7(H_2(\lambda)) & H_1(\lambda)=a\,H_2(\lambda)+c\,\lambda,\quad a\in \mathbb{C}\setminus\{0\},\ c\in\mathbb{C} \\
\hline
\mathcal{B}_2(f_1(\lambda)) \cong \mathcal{B}_2(f_2(\lambda)) & f_1(\lambda)=f_2(\lambda) \\
\hline
\end{array}
\end{eqnarray*}
\end{thm}

\begin{proof}
By Proposition~\ref{pro-1}, every left-symmetric conformal algebra appearing
in Theorem~\ref{thm-1} is isomorphic to one of the ten normal-form types
listed above. By Propositions~\ref{pp1}, \ref{pro-p2}, \ref{pp4.3} and
\ref{pp4.4}, normal forms belonging to different types are mutually
non-isomorphic. It remains to determine the isomorphisms within each
parametrized type.

Let $\varphi$ be an isomorphism between two normal forms of the same type,
with $\mathbb{C}[\partial]$-bases $\{X_1,Y_1\}$ and $\{X_2,Y_2\}$ respectively. Write
\begin{eqnarray}
\varphi(X_1)=A(\partial)X_2+B(\partial)Y_2,\qquad
\varphi(Y_1)=C(\partial)X_2+D(\partial)Y_2,
\end{eqnarray}
where $A(\partial)$, $B(\partial)$, $C(\partial)$, $D(\partial)\in\mathbb C[\partial]$ and $A(\partial)D(\partial)-B(\partial)C(\partial)\in\mathbb C\setminus \{0\}$.

\smallskip
\noindent
\textit{Type $\mathcal A_1(B(\lambda, \partial))$.} It is easy to see that if $\mathcal A_1(B_1(\lambda, \partial))\cong \mathcal A_1(0)$, then $B_1(\lambda, \partial)=0$. Suppose that $B_2(\lambda,\partial)\neq 0$.
Then $\varphi(X_1)_\lambda\varphi(X_1)=0$ gives $B(-\lambda)B(\partial+\lambda)=0$. Thus, $B(\partial)=0$. Then invertibility  implies
$A(\partial)=a$ and $D(\partial)=d$, for some $a$, $d\in\mathbb C\setminus \{0\}$.
Using $\varphi(Y_{1\lambda}Y_1)=\varphi(B_1(\lambda,\partial)X_1)=\varphi(Y_1)_\lambda\varphi(Y_1)$, we obtain $aB_1(\lambda,\partial)
=
d^2B_2(\lambda,\partial)$.
Thus, we have $B_1(\lambda,\partial)=\frac{d^2}{a}B_2(\lambda,\partial)$.
Obviously, $B_1(\lambda,\partial)=0$ if and only if $B_2(\lambda,\partial)=0$. Therefore, $B_1(\lambda,\partial)=cB_2(\lambda,\partial)$
for some $c\in\mathbb C\setminus \{0\}$.

\smallskip
\noindent
\textit{Type $\mathcal A_3(V(\partial))$.}
As in the case of type $\mathcal A_1(B(\lambda, \partial))$, we have
$B(\partial)=0$. Hence we have
\begin{eqnarray*}
\varphi(X_1)=aX_2,\qquad
\varphi(Y_1)=C(\partial)X_2+dY_2,
\qquad a,d\in\mathbb C\setminus \{0\}.
\end{eqnarray*}
$\varphi(X_{1})_\lambda\varphi(Y_1)=\varphi(X_1)$ gives $d=1$.
$\varphi(Y_{1})_\lambda\varphi(Y_1)=\varphi(V_1(\partial)X_1+Y_1)$ yields
$aV_1(\partial)+C(\partial)
=
C(-\lambda)+V_2(\partial)$.
Hence we have $V_2(\partial)-aV_1(\partial)
=
C(\partial)-C(-\lambda)$.
Since the left-hand side is independent of $\lambda$, $C(-\lambda)$ must be a constant.
Consequently, we have $V_2(\partial)=aV_1(\partial)$. Therefore, $A_3(V_1(\partial))\cong A_3(V_2(\partial))$ if and only if
$V_1(\partial)=cV_2(\partial)$ for some $c\in\mathbb C\setminus \{0\}$.

\smallskip
\noindent
\textit{Type $\mathcal A_4(f(\lambda))$.}
From $
\varphi({X_{1}}_\lambda Y_1)
=
\varphi(X_1)_\lambda\varphi(Y_1)
$
and $f_1(\lambda)\neq0$, comparing the coefficient of  $X_2$ gives
$C(\partial)=0$. Then invertibility implies $A(\partial)=a$ and $D(\partial)=d$ for some $a$, $d\in\mathbb C\setminus \{0\}$.
 $\varphi(X_{1})_\lambda \varphi(X_1)=0$ further gives $B(\partial)=0$.
Comparing the coefficient of $Y_2$ in
$\varphi(X_{1\lambda}Y_1)=
\varphi(X_1)_\lambda\varphi(Y_1)$ yields $f_1(\lambda)=a f_2(\lambda)$.
Thus we have $f_1(\lambda)=cf_2(\lambda)$ for some $c\in\mathbb C\setminus \{0\}$.

\smallskip
\noindent
\textit{Type $\mathcal A_5(U(\lambda))$.}
As in the case of type $\mathcal A_1(B(\lambda, \partial))$, we have
$B(\partial)=0$. Thus, we have
\begin{eqnarray*}
\varphi(X_1)=aX_2,\qquad
\varphi(Y_1)=C(\partial)X_2+dY_2,
\end{eqnarray*}
for some $a$, $d\in \mathbb C\setminus \{0\}$.
$\varphi(Y_{1})_\lambda\varphi(X_1)=\varphi(X_1)$ gives $d=1$. Comparing the coefficient of
$X_2$ in
$\varphi({Y_1}_\lambda Y_1)=
\varphi(Y_1)_\lambda\varphi(Y_1)$ gives
\begin{eqnarray*}
U_2(\lambda)-aU_1(\lambda)
=
C(\partial)-C(\partial+\lambda).
\end{eqnarray*}
Since the left-hand side depends only on $\lambda$,
$\deg C(\partial)\le1$. Write $C(\partial)=c_0+c_1\partial$.
Then we have $C(\partial)-C(\partial+\lambda)=-c_1\lambda$,
and hence $U_2(\lambda)=aU_1(\lambda)-c_1\lambda$. Therefore,  we have
$U_1(\lambda)=\alpha U_2(\lambda)+c\lambda$
for some $\alpha\in\mathbb C\setminus \{0\}$ and $c\in\mathbb C$.

\smallskip
\noindent
\textit{Type $\mathcal A_6(g(\lambda))$.}
As in the case of type $\mathcal A_1(B(\lambda, \partial))$, we have
$B(\partial)=0$. Therefore, we have
\begin{eqnarray*}
\varphi(X_1)=aX_2,\qquad
\varphi(Y_1)=C(\partial)X_2+dY_2,
\end{eqnarray*}
for some $a$, $d\in \mathbb C\setminus \{0\}$.
 $\varphi(X_{1})_\lambda \varphi(Y_1)=\varphi(X_1)$ gives $d=1$. Using
$\varphi(Y_{1})_\lambda\varphi(X_1)=\varphi(g_1(\lambda)X_1)$, we obtain
$ag_1(\lambda)X_2=ag_2(\lambda)X_2$.
Hence $g_1(\lambda)=g_2(\lambda)$.

\smallskip
\noindent
\textit{Type $\mathcal A_7(H(\lambda))$.}
As in the case of type $\mathcal A_1(B(\lambda, \partial))$, we have
$B(\partial)=0$. Then we have
\begin{eqnarray*}
\varphi(X_1)=aX_2,\qquad
\varphi(Y_1)=C(\partial)X_2+dY_2,
\end{eqnarray*}
for some $a$, $d\in \mathbb C\setminus \{0\}$. $\varphi(X_{1})_\lambda \varphi(Y_1)=\varphi(X_1)$ gives $d=1$.
Using $\varphi(Y_{1})_\lambda\varphi(Y_1)=
\varphi(H_1(\lambda)(\lambda+\partial)X_1+Y_1)$, we obtain
\begin{equation}\label{eq:A7iso}
C(\partial+\lambda)-C(\partial)+C(-\lambda)
=
\bigl(aH_1(\lambda)-H_2(\lambda)\bigr)(\lambda+\partial).
\end{equation}
Let $C(\partial)=\sum_{i=0}^n c_i\partial^i$.
If $n\ge3$, then the degree of $\partial$ in the left-hand side of Eq. \eqref{eq:A7iso} is $n-1\ge2$, whereas the degree of $\partial$ in the right-hand side is at most one, a contradiction. Hence $n\le2$.
Writing $C(\partial)=c_0+c_1\partial+c_2\partial^2$,
the left-hand side becomes
$c_0+2c_2\lambda(\lambda+\partial)$.
Since the right-hand side is divisible by $\lambda+\partial$, we obtain
$c_0=0$. Therefore, we have $H_2(\lambda)=aH_1(\lambda)-2c_2\lambda$. Thus,
$H_1(\lambda)=\alpha H_2(\lambda)+c\lambda$
for some $\alpha\in\mathbb C\setminus \{0\}$ and $c\in\mathbb C$.

\smallskip
\noindent
\textit{Type $\mathcal B_2(f(\lambda))$.}
Since $\varphi(X_{1})_\lambda\varphi(X_1)=\varphi(X_1)$, comparing the coefficients of $X_2$ and
$Y_2$  gives
\begin{eqnarray*}
A(\partial)=A(-\lambda)A(\partial+\lambda),\\
B(\partial)=A(-\lambda)B(\partial+\lambda)f_2(\lambda).
\end{eqnarray*}
If $A(\partial)=0$, the second identity gives $B(\partial)=0$, contradicting with the invertibility.
Thus $A(\partial)\neq0$. The first identity forces $A(\partial)$ to be a nonzero constant. Therefore, $A(\partial)=1$.
Since $\varphi(Y_{1})_\lambda\varphi(Y_1)=0$, comparing the coefficients of $X_2$ gives $C(-\lambda)C(\partial+\lambda)=0$.
Hence $C(\partial)=0$. Then the invertibility  implies
$D(\partial)=d\in\mathbb C\setminus \{0\}$. Finally,
$\varphi(X_{1\lambda} Y_1)=
\varphi(X_1)_\lambda\varphi(Y_1)$ gives $f_1(\lambda)=f_2(\lambda)$.

It remains to prove sufficiency of the displayed parameter conditions.
This follows from explicit changes of $\mathbb{C}[\partial]$-basis. For instance, if
$B_1(\lambda, \partial)=cB_2(\lambda, \partial)$ for some $c\in \mathbb{C}\setminus \{0\}$, taking $\varphi(X_1)=c^{-1}X_2$ and $\varphi(Y_1)=Y_2$, then the $\mathbb{C}[\partial]$-module homomorphism $\varphi: \mathcal A_1(B_1(\lambda, \partial))\rightarrow \mathcal A_1(B_2(\lambda, \partial))$ is an isomorphism of left-symmetric conformal algebras.
If $V_1(\partial)=cV_2(\partial)$ for some $c\in \mathbb{C}\setminus \{0\}$, taking $\varphi(X_1)=c^{-1}X_2$ and $\varphi(Y_1)=Y_2$, then the $\mathbb{C}[\partial]$-module homomorphism $\varphi: \mathcal A_3(V_1(\partial))\rightarrow \mathcal A_3(V_2(\partial))$ is an isomorphism of left-symmetric conformal algebras.  If $f_1(\lambda)=cf_2(\lambda)$ for some $c\in \mathbb{C}\setminus \{0\}$ in
type $\mathcal A_4(f(\lambda))$, taking $\varphi(X_1)=cX_2$ and $\varphi(Y_1)=Y_2$, then the $\mathbb{C}[\partial]$-module homomorphism $\varphi: \mathcal A_4(f_1(\lambda))\rightarrow \mathcal A_4(f_2(\lambda))$ is an isomorphism of left-symmetric conformal algebras.
If $U_1(\lambda)=aU_2(\lambda)+c\lambda$ for some $a\in \mathbb{C}\setminus \{0\}$ and $c\in \mathbb{C}$,
taking
$\varphi(X_1)=a^{-1}X_2$, and $\varphi(Y_1)=\frac{c}{a}\partial X_2+Y_2$, then the $\mathbb{C}[\partial]$-module homomorphism $\varphi: \mathcal A_5(U_1(\lambda))\rightarrow \mathcal A_5(U_2(\lambda))$ is an isomorphism of left-symmetric conformal algebras. If
$H_1(\lambda)=aH_2(\lambda)+c\lambda$  for some $a\in \mathbb{C}\setminus \{0\}$ and $c\in \mathbb{C}$,
taking $\varphi(X_1)=a^{-1}X_2$ and
$\varphi(Y_1)=\frac{c}{2a}\partial^2X_2+Y_2$, then the $\mathbb{C}[\partial]$-module homomorphism $\varphi: \mathcal A_7(H_1(\lambda))\rightarrow \mathcal A_7(H_2(\lambda))$ is an isomorphism of left-symmetric conformal algebras.
For types $\mathcal A_6(g(\lambda))$ and $\mathcal B_2(f(\lambda))$, equality of the parameters gives
the identity isomorphism. Thus all the listed conditions are both
necessary and sufficient.
\end{proof}

\begin{pro}\label{prop:Nov}
Among the ten normal forms in Proposition~\ref{pro-1}, the Novikov conformal
algebras are characterized as follows:
\begin{enumerate}[label=(\roman*)]
    \item
    $\mathcal A_1(B(\lambda,\partial))$, $\mathcal A_2$,
    $\mathcal A_4(f(\lambda))$, and $\mathcal B_1$ are always
    Novikov conformal algebras;

    \item
    $\mathcal A_3(V(\partial))$ is a Novikov conformal algebra if and only if
    $V(\partial)\in\mathbb C$;

    \item
    $\mathcal A_6(g(\lambda))$ is a Novikov conformal algebra if and only if
    $g(\lambda)$ is a constant. In view of the defining restrictions of
    $\mathcal A_6$, this means
    $g(\lambda)\in\mathbb C\setminus\{0,1\}$;

    \item
    $\mathcal A_7(H(\lambda))$ is a Novikov conformal algebra if and only if
    \[
    H(\lambda)=a\lambda+b
    \]
    for some $a,b\in\mathbb C$;

    \item
    $\mathcal A_5(U(\lambda))$, $\mathcal B_2(f(\lambda))$, and
    $\mathcal S$ are not Novikov conformal algebras.
\end{enumerate}
\end{pro}

\begin{proof}
Since all ten normal forms are already left-symmetric conformal algebras,
it remains only to check
\begin{equation}\label{Novikov-id}
(a_\lambda b)_{\lambda+\mu}c
=
(a_\lambda c)_{-\mu-\partial}b,\;\;a, b, c\in R.
\end{equation}
By conformal
sesquilinearity, it is enough to check Eq. \eqref{Novikov-id} on triples of
basis elements $X,Y$.

It is direct to check that $\mathcal A_1(B(\lambda,\partial))$, $\mathcal A_2$, $\mathcal A_4(f(\lambda))$ and $\mathcal B_1$ are Novikov conformal algebras, and, $\mathcal A_5(U(\lambda))$, $\mathcal B_2(f(\lambda))$, and
    $\mathcal S$ are not Novikov conformal algebras.

Next, we consider $\mathcal A_3(V(\partial))$. All basis triples satisfy
Eq. \eqref{Novikov-id} identically except possibly $(Y,Y,Y)$. For this triple,
Eq. (\ref{Novikov-id}) becomes
\begin{eqnarray}
V(-\lambda-\mu)=V(\mu+\partial).
\end{eqnarray}
Therefore, $V(\partial)\in \mathbb{C}$. Conversely, if $V(\partial)\in \mathbb{C}$, Eq. \eqref{Novikov-id} holds. Hence all basis triples satisfy Eq. \eqref{Novikov-id}.
Therefore, $\mathcal A_3(V(\partial))$ is a Novikov conformal algebra if and only if $V(\partial)\in\mathbb C$.

For $\mathcal A_6(g(\lambda))$, take $(a,b,c)=(Y,X,Y)$. Then
Eq. \eqref{Novikov-id} gives $g(\lambda)=g(-\mu-\partial)$.
Therefore $g(\lambda)$ must be a constant. Conversely, if $g(\lambda)$ is a constant, a direct
check of the eight basis triples shows that Eq. \eqref{Novikov-id} holds.
Thus $\mathcal A_6(g(\lambda))$ is a Novikov  conformal algebra if and only if $g(\lambda)\in \mathbb{C}$.
Since in the normal form $\mathcal A_6(g(\lambda))$, we already require
$g(\lambda)\neq0,1$. Thus $g(\lambda)\in \mathbb C\setminus\{0,1\}$.

Finally, we consider $\mathcal A_7(H(\lambda))$. All basis triples except
$(Y,Y,Y)$ satisfy Eq. \eqref{Novikov-id} identically. For $(Y,Y,Y)$, the
identity is equivalent to
\begin{equation}\label{eq:A7-Novikov}
\bigl(H(\lambda+\mu)-H(\lambda)\bigr)
(\lambda+\mu+\partial)
=
\mu\bigl(H(\lambda)-H(-\mu-\partial)\bigr).
\end{equation}
Comparing the coefficients of $\partial$ shows that $H(\lambda+\mu)-H(\lambda)$
is independent of $\lambda$. Hence there exists a polynomial $C(\mu)$
such that $H(\lambda+\mu)-H(\lambda)=C(\mu)$.
Setting $\lambda=0$ gives $C(\mu)=H(\mu)-H(0)$.
Therefore, we have $H(\lambda+\mu)
=
H(\lambda)+H(\mu)-H(0)$.
Let $K(\lambda)=H(\lambda)-H(0)$.
Then we have $K(\lambda+\mu)=K(\lambda)+K(\mu)$.
Since $K(\lambda)$ is a polynomial over $\mathbb C$, it follows that
$K(\lambda)=a\lambda$
for some $a\in\mathbb C$. Hence, we have
$H(\lambda)=a\lambda+b$.
Conversely, if $H(\lambda)=a\lambda+b$, then both sides of
Eq. \eqref{eq:A7-Novikov} equal $a\mu(\lambda+\mu+\partial)$. Thus,
 Eq. \eqref{Novikov-id} holds. Hence
$\mathcal A_7(H(\lambda))$
is a Novikov conformal algebra if and only if
$H(\lambda)=a\lambda+b$.

The proof is completed.
\end{proof}
\begin{cor}\label{cor-Nov}
    All  compatible Novikov conformal algebra structures on solvable Lie conformal algebras that are free of rank two as $\mathbb{C}[\partial]$-modules are isomorphic to one of the following:
    \begin{itemize}
\item $\mathcal{A}_1(B(\lambda,\partial))$, where $B(\lambda,\partial)\in\mathbb{C}[\lambda,\partial]$;
\item $\mathcal{A}_2$;
\item $\mathcal{A}_3(V(\partial))$, where $V(\partial)\in\mathbb{C}$;
\item $\mathcal{A}_4(f(\lambda))$, where $f(\lambda)\in\mathbb{C}[\lambda]\setminus\{0\}$;
\item $\mathcal{A}_6(g(\lambda))$, where $g(\lambda)\in\mathbb{C}\setminus\{0,1\}$;
\item $\mathcal{A}_7(H(\lambda))$, where $H(\lambda)=a\lambda+b$ for some $a,b\in\mathbb{C}$;
\item $\mathcal{B}_1$.
\end{itemize}
Moreover, the isomorphism relations among the parametrized normal forms are those induced by Theorem \ref{pp4.5}.
\end{cor}
\begin{proof}
This follows directly from Theorem~\ref{pp4.5} and Proposition~\ref{prop:Nov}.
\end{proof}
\delete{\begin{thm}\label{pp4.5}
    All  compatible left-symmetric conformal algebra structures on  solvable Lie conformal algebras which are free of rank two as $\mathbb{C}[\partial]$-modules are isomorphic to one of the following ten normal forms:
\begin{eqnarray*}
\mathcal{A}_1(B(\lambda,\partial)),\ \mathcal{A}_2,\ \mathcal{A}_3(V(\partial)),\ \mathcal{A}_4(f(\lambda)),\ \mathcal{A}_5(U(\lambda)),\ \mathcal{A}_6(g(\lambda)),\ \mathcal{A}_7(H(\lambda)),\ \mathcal{B}_1,\ \mathcal{B}_2(f(\lambda)),\ \mathcal{S}.
\end{eqnarray*}
Moreover, any two different types of left-symmetric conformal algebras are not isomorphic and two left-symmetric conformal algebras of the same type are isomorphic if and only if the parameters satisfy the conditions listed in the following table.
\begin{eqnarray*}
\begin{array}{|c|c|}
\hline
\text{Algebra type} & \text{Isomorphism condition} \\
\hline
\mathcal{A}_1(B_1(\lambda,\partial)) \cong \mathcal{A}_1(B_2(\lambda,\partial)) & B_1(\lambda,\partial)=c\,B_2(\lambda,\partial),\quad c\in \mathbb{C}\setminus\{0\}\\
\hline
\mathcal{A}_3(V_1(\partial)) \cong \mathcal{A}_3(V_2(\partial)) & V_1(\partial)=c\,V_2(\partial),\quad c\in \mathbb{C}\setminus\{0\}\\
\hline
\mathcal{A}_4(f_1(\lambda)) \cong \mathcal{A}_4(f_2(\lambda)) & f_1(\lambda)=c\,f_2(\lambda),\quad c\in \mathbb{C}\setminus\{0\}\\
\hline
\mathcal{A}_5(U_1(\lambda)) \cong \mathcal{A}_5(U_2(\lambda)) & U_1(\lambda)=a\,U_2(\lambda)+c\,\lambda,\quad a\in \mathbb{C}\setminus\{0\},\ c\in\mathbb{C} \\
\hline
\mathcal{A}_6(g_1(\lambda)) \cong \mathcal{A}_6(g_2(\lambda)) & g_1(\lambda)=g_2(\lambda) \\
\hline
\mathcal{A}_7(H_1(\lambda)) \cong \mathcal{A}_7(H_2(\lambda)) & H_1(\lambda)=a\,H_2(\lambda)+c\,\lambda,\quad a\in \mathbb{C}\setminus\{0\},\ c\in\mathbb{C} \\
\hline
\mathcal{B}_2(f_1(\lambda)) \cong \mathcal{B}_2(f_2(\lambda)) & f_1(\lambda)=f_2(\lambda) \\
\hline
\end{array}
\end{eqnarray*}
\end{thm}
\begin{proof}
    By Proposition~\ref{pro-1}, any family of left-symmetric conformal algebras classified in Theorem \ref{thm-1} can be transformed into one of these ten normal forms. By Propositions~\ref{pp1}, \ref{pro-p2}, \ref{pp4.3} and \ref{pp4.4}, these ten classes are mutually non-isomorphic.

  For the isomorphism conditions, we first illustrate the computation for the type \(\mathcal{B}_2(f(\lambda))\). The other cases are similar.

     Let $\varphi:\mathcal{B}_2(f_1(\lambda))\to\mathcal{B}_2(f_2(\lambda))$ be an isomorphism of left-symmetric conformal algebras. Denote the basis of $\mathcal{B}_2(f_1(\lambda))$ by $\{X_1,Y_1\}$ and that of $\mathcal{B}_2(f_2(\lambda))$ by $\{X_2,Y_2\}$, and write
$$
\varphi(X_1)=A(\partial)X_2+B(\partial)Y_2,\qquad
\varphi(Y_1)=C(\partial)X_2+D(\partial)Y_2,
$$
where \(A(\partial),B(\partial),C(\partial),D(\partial)\in\mathbb{C}[\partial]\) and the transformation matrix is invertible.

By $\varphi(X_{1\lambda}X_1)=\varphi(X_1)_\lambda\varphi(X_1)$, comparing the coefficients of $X_2$ and $Y_2$ yields
\begin{align}\label{73-}
    A(\partial)=A(-\lambda)A(\partial+\lambda),
\end{align}
\begin{align}\label{74-}
    B(\partial)=A(-\lambda)B(\partial+\lambda)f_2(\lambda).
\end{align}
By Eqs.~(\ref{73-})--(\ref{74-}), we have $A(\partial)=1$.

By $\varphi(Y_{1\lambda}Y_1)=\varphi(Y_1)_\lambda\varphi(Y_1)$, comparing the coefficients of $X_2$  yields $C(-\lambda)C(\partial+\lambda)=0$, i.e. $C(\partial)=0$.
Then the transformation matrix is  $\begin{pmatrix}1&B(\partial)\\0&D(\partial)\end{pmatrix}$, whose invertibility forces $D(\partial)\in \mathbb{C}\setminus\{0\}$ and we denote it by $d$.

By  $\varphi(X_{1\lambda}Y_1)=\varphi(X_1)_\lambda\varphi(Y_1)$, comparing the coefficients of $Y_2$  yields $f_1(\lambda)=f_2(\lambda)$.
Conversely, if $f_1(\lambda)=f_2(\lambda)$, the map $\varphi(X_1)=X_2$, $\varphi(Y_1)=Y_2$ is an isomorphism.
Therefore, $\mathcal{B}_2(f_1(\lambda))\cong\mathcal{B}_2(f_2(\lambda))$ if and only if $f_1(\lambda)=f_2(\lambda)$.\\

Similarly, for the remaining types, let \(\varphi: \mathcal{A}_i \to \mathcal{A}_i\) be an isomorphism. Denote the bases by \(\{X_1,Y_1\}\) and \(\{X_2,Y_2\}\), and write
\[
\varphi(X_1) = A(\partial)X_2 + B(\partial)Y_2, \quad \varphi(Y_1) = C(\partial)X_2 + D(\partial)Y_2,
\]
where \(A(\partial),B(\partial),C(\partial),D(\partial)\in\mathbb{C}[\partial]\) and the transformation matrix is invertible.

\smallskip
\textit{Common Step for \(\mathcal{A}_1(B(\lambda,\partial)), \mathcal{A}_3(V(\partial)), \mathcal{A}_5(U(\lambda)), \mathcal{A}_6(g(\lambda)), \mathcal{A}_7(H(\lambda))\):}\\
For these types, from \(\varphi(X_{1\lambda}X_1) = \varphi(X_1)_\lambda\varphi(X_1)\), comparing the coefficients we have \(B(\partial)=0\). Since the transformation matrix is invertible, its determinant \(A(\partial)D(\partial) \in \mathbb{C}\setminus\{0\}\), which implies \(A(\partial), D(\partial) \in \mathbb{C}\setminus\{0\}\). We denote them by \(a\) and \(d\), respectively. Thus, we uniformly have
\[
\varphi(X_1) = aX_2, \quad \varphi(Y_1) = C(\partial)X_2 + dY_2 \quad (a, d \in \mathbb{C}\setminus\{0\}).
\]
Now we proceed with the remaining calculations for each type individually.

\smallskip
\textit{Type \(\mathcal{A}_1(B(\lambda,\partial))\):}
Substituting into \(\varphi(Y_{1\lambda}Y_1) = \varphi(Y_1)_\lambda\varphi(Y_1)\) and comparing the \(X_2\)-coefficients yields \(B_1(\lambda,\partial) = \frac{d^2}{a}B_2(\lambda,\partial)\).

\smallskip
\textit{Type \(\mathcal{A}_3(V(\partial))\):}
Comparing the \(Y_2\)-coefficient of \(\varphi(Y_{1\lambda}Y_1)=\varphi(Y_1)_\lambda\varphi(Y_1)\) yields \(d^2=d\), so \(d=1\). Comparing the \(X_2\)-coefficients yields \(V_2(\partial) - aV_1(\partial) = C(\partial) - C(-\lambda)\). Since the left side is independent of \(\lambda\), the right side must also be independent of \(\lambda\), which forces \(C(\partial)\) to be a constant. Hence \(V_2(\partial) = aV_1(\partial)\).

\smallskip
\textit{Type \(\mathcal{A}_4(f(\lambda))\):}
From \(\varphi(X_{1\lambda}Y_1) = \varphi(X_1)_\lambda\varphi(Y_1)\), comparing the \(X_2\)-coefficients yields \(f_1(\lambda)C(\partial)=0\). Since \(f_1(\lambda)\neq0\), then we have \(C(\partial)=0\) and \(A(\partial)=a, D(\partial)=d, (a,d \in \mathbb{C}\setminus\{0\}\)). Comparing the \(Y_2\)-coefficients  gives \(f_1(\lambda) = a f_2(\lambda)\).

\smallskip
\textit{Type \(\mathcal{A}_5(U(\lambda))\):}
From \(\varphi(Y_{1\lambda}X_1)= \varphi(Y_1)_\lambda\varphi(X_1)\), comparing the \(X_2\)-coefficients gives \(d=1\). Comparing the \(X_2\)-coefficients of \(\varphi(Y_{1\lambda}Y_1)=\varphi(Y_1)_\lambda\varphi(Y_1)\) yields \(U_2(\lambda) - aU_1(\lambda) = C(\partial) - C(\partial+\lambda)\). Since the left side depends only on \(\lambda\), the degree of \(C(\partial)\) is at most \(1\). Setting \(C(\partial)=c_0+c_1\partial\) yields \(U_2(\lambda) = aU_1(\lambda) + c_1\lambda\).

\smallskip
\textit{Type \(\mathcal{A}_6(g)\):}
From \(\varphi(X_{1\lambda}Y_1)=\varphi(X_1)_\lambda\varphi(Y_1)\), comparing the \(X_2\)-coefficients gives \(d=1\). Substituting into \(\varphi(Y_{1\lambda}X_1) = \varphi(Y_1)_\lambda\varphi(X_1)\) we have \(g_1(\lambda) = g_2(\lambda)\).

\smallskip
\textit{Type \(\mathcal{A}_7(H)\):}
From \(\varphi(X_{1\lambda}Y_1)=\varphi(X_1)_\lambda\varphi(Y_1)\), comparing the \(X_2\)-coefficients gives \(d=1\). Comparing the \(X_2\)-coefficients of \(\varphi(Y_{1\lambda}Y_1)=\varphi(Y_1)_\lambda\varphi(Y_1)\) yields
\begin{align}\label{7.5}
    C(\partial+\lambda)-C(\partial)+C(-\lambda) = [aH_1(\lambda)-H_2(\lambda)](\lambda+\partial).
\end{align}
Let \(C(\partial)=\sum_{i=0}^{n} c_i\partial^i\). If \(n\ge3\), comparing the degrees of \(\partial\) in Eq.~(\ref{7.5}) gives a contradiction, hence \(n\le2\). Then we have \(H_2(\lambda)=aH_1(\lambda)-2c_2\lambda\).
\end{proof}
\begin{pro}
    We identify which of the ten normal forms in Proposition~\ref{pro-1} are Novikov conformal algebras as follows:
\begin{enumerate}[label=(\roman*)]
    \item \(\mathcal{A}_1(B(\lambda,\partial))\), \(\mathcal{A}_2\), \(\mathcal{A}_4(f(\lambda))\), and \(\mathcal{B}_1\) are always Novikov conformal algebras;
    \item \(\mathcal{A}_3(V(\partial))\) is Novikov if and only if \(V(\partial)\) is a constant;
    \item \(\mathcal{A}_6(g(\lambda))\) is Novikov if and only if \(g(\lambda)\) is a constant with \(g(\lambda) \neq 0,1\);
    \item \(\mathcal{A}_7(H(\lambda))\) is Novikov if and only if \(H(\lambda) = a\lambda+b\) for some \(a,b \in \mathbb{C}\);
    \item \(\mathcal{A}_5(U(\lambda))\), \(\mathcal{B}_2(f(\lambda))\), and \(\mathcal{S}\) are never Novikov conformal algebras.
\end{enumerate}
\end{pro}

\begin{proof}
Recall that a Novikov conformal algebra is a left-symmetric conformal algebra that additionally satisfies the axiom $(a_\lambda b)_{\lambda+\mu}c = (a_\lambda c)_{-\mu-\partial}b$.
Since all ten normal forms in Proposition~\ref{pro-1} are already left-symmetric conformal algebras, it suffices to verify whether they satisfy the axiom.

For \(\mathcal{A}_5(U(\lambda))\), \(\mathcal{B}_2(f(\lambda))\), and \(\mathcal{S}\), explicit counterexamples are given by:
\begin{itemize}
    \item \(\mathcal{A}_5(U(\lambda))\): taking \(a=b=Y, c=X\), then we have $X=0$, a contradiction.
    \item\(\mathcal{B}_2(f(\lambda))\): taking \(a=b=X, c=Y\), then we have  \(f(\lambda+\mu)Y=0\) , a contradiction.
    \item  \(\mathcal{S}\): taking \(a=b=X, c=Y\), then we have  \(\frac{1}{2}Y=0\) , a contradiction.
\end{itemize}

For\(\mathcal{A}_3(V(\partial))\), taking \(a=b=c=Y\) yields \(V(-\lambda-\mu) = V(\mu+\partial)\) for all \(\lambda,\mu,\partial\) , then we know that \(V(\partial)\) is a constant.

For\(\mathcal{A}_6(g(\lambda))\), taking \(a=Y, b=X, c=Y\) yields \(g(\lambda) = g(-\mu-\partial)\) for all \(\lambda,\mu,\partial\), forcing \(g(\lambda)\) to be a constant.

For \(\mathcal{A}_7(H(\lambda))\), taking \(a=b=c=Y\), the axiom reduces to
\begin{align}\label{7.6}
    (H(\lambda+\mu)-H(\lambda))(\lambda+\mu+\partial) = \mu(H(\lambda)-H(-\mu-\partial)).
\end{align}
Comparing the coefficients of \(\partial\) on both sides of Eq.~(\ref{7.6}), we see that the coefficient on the left side is \(H(\lambda+\mu)-H(\lambda)\), while the coefficient on the right side depends only on \(\mu\). Therefore, \(H(\lambda+\mu)-H(\lambda)\) must be independent of \(\lambda\).  Let \(H(\lambda+\mu)-H(\lambda) = C(\mu)\). Setting \(\lambda = 0\) yields
\begin{align}
    C(\mu) = H(\mu) - H(0).
\end{align}
 Consequently, we have
\begin{align}\label{eq7.8}
    H(\lambda+\mu) = H(\lambda) + H(\mu) - H(0).
\end{align}

Set \(K(\lambda) = H(\lambda) - H(0)\). Then from Eq.~(\ref{eq7.8}) we have \(K(\lambda+\mu) = K(\lambda) + K(\mu)\). Since \(K(\lambda)\) is a polynomial, it must be of the form \(K(\lambda) = c\lambda\) for some constant \(c \in \mathbb{C}\). Hence, \(H(\lambda) = c\lambda + H(0)\), which implies that \(H(\lambda) = a\lambda + b\) for some \(a, b \in \mathbb{C}\).

Finally, for \(\mathcal{A}_1(B(\lambda,\partial))\), \(\mathcal{A}_2\), \(\mathcal{A}_4(f(\lambda))\), and \(\mathcal{B}_1\), a direct substitution of the basis elements into the axiom confirms that it holds identically.
\end{proof}}

\noindent {\bf Acknowledgments.} This work is supported by the Zhejiang Provincial Natural Science Foundation of China (No.
LZ25A010004).

\smallskip

\noindent
{\bf Declaration of interests. } The authors have no conflicts of interest to disclose.

\smallskip

\noindent
{\bf Data availability. } No new data were created or analyzed in this study.

\end{document}